\documentclass[12pt]{article}
\usepackage{amsmath,amsxtra,amssymb,latexsym,amscd,amsthm}
\usepackage{mathptmx}
 \usepackage{mathtools}
 \usepackage{pifont}
 \usepackage{authblk}
 \usepackage{url}
 \usepackage{multicol}
\usepackage{mathrsfs}
\usepackage{color}
\usepackage{multicol}
\usepackage{float}
\usepackage{tikz}
\usepackage{tikz-cd}
\usetikzlibrary{calc,fadings,decorations.pathreplacing}
\usetikzlibrary{arrows}
\usetikzlibrary{shapes.arrows}   
\usetikzlibrary{positioning}
\usepackage{caption}
\usepackage{subcaption}
\usepackage{hyperref}
\usepackage{mathrsfs}
\usepackage{wasysym}
\usepackage{array}
\usepackage{mathtools}

\definecolor{newcol}{HTML}{42E393}
\definecolor{1}{rgb}{1,0.2,0.3}
\definecolor{2}{rgb}{0.1,0.3,0.5}
\definecolor{3}{rgb}{1,1,0}
\definecolor{4}{rgb}{255,255,255}
\def\ra{.7}

\newtheorem{theorem}{Theorem}[section]
\newtheorem{corollary}{Corollary}[theorem]
\newtheorem{lemma}[theorem]{Lemma}

\theoremstyle{definition}
\newtheorem{definition}{Definition}[section]

 \newdimen\R
\theoremstyle{remark}
\newtheorem{remark}{Remark}

\newcommand{\hardholesOne}[1]{

	\ifodd #1
		\pgfmathsetmacro{\a}{#1/2-1/2}
	\else
		\pgfmathsetmacro{\a}{#1/2-1}
	\fi
		
	\foreach \i in {1,...,\a}{
		\foreach \j in {1,...,\a}{
		\fill[newcol!40] (\i*2*\ra-\ra,\j*2*\ra-\ra) --++(0:\ra) --++(90:\ra) --++(180:\ra) -- cycle;
		}
	}
}

\newcommand{\hardholesTwo}[2]{

		\pgfmathsetmacro{\a}{#1/2-3/2}
        \pgfmathsetmacro{\b}{#2/2-3/2}
        \foreach \i in {1,...,\b}{
		    \foreach \j in {1,...,\a}{
		        \fill[newcol!40] (\i*2*\ra,\j*2*\ra) --++(0:\ra) --++(90:\ra) --++(180:\ra) -- cycle;
		    }
	    }

}

\newcommand{\fullborder}[2]{

	\foreach \i in {2,...,#1}{
		\fill[black!80] (0,\i*\ra-\ra) --++(0:\ra) --++(90:\ra) --++(180:\ra) -- cycle;
		\fill[black!80] (#2*\ra-\ra,\i*\ra-\ra) --++(0:\ra) --++(90:\ra) --++(180:\ra) -- cycle;
		}
		
	\foreach \i in {2,...,#2}{
		\fill[black!80] (\i*\ra-2*\ra,0) --++(0:\ra) --++(90:\ra) --++(180:\ra) -- cycle;
		}
	\foreach \i in {3,...,#2}{
		\fill[black!80] (\i*\ra-2*\ra,#1*\ra-\ra) --++(0:\ra) --++(90:\ra) --++(180:\ra) -- cycle;
		}
}

\newcommand{\fullborderLines}[2]{
	
	\foreach \i in {2,...,#1}{
		\draw[white] (0,\i*\ra-\ra) --++(0:\ra) --++(90:\ra) --++(180:\ra) --++(270:\ra);
		\draw[white] (#2*\ra-\ra,\i*\ra-\ra) --++(0:\ra) --++(90:\ra) --++(180:\ra) --++(270:\ra);
		}
		
	\foreach \i in {2,...,#2}{
		\draw[white] (\i*\ra-2*\ra,0) --++(0:\ra) --++(90:\ra) --++(180:\ra) --++(270:\ra);
		}
	\foreach \i in {3,...,#2}{
		\draw[white] (\i*\ra-2*\ra,#1*\ra-\ra) --++(0:\ra) --++(90:\ra) --++(180:\ra) --++(270:\ra);
		}
}

\newcommand{\snakeborder}[2]{
	
	\pgfmathsetmacro{\a}{#1-1}
	\pgfmathsetmacro{\b}{#2-1}
	
	\foreach \i in {1,...,\a}{
		\fill[black!80] (0,\i*\ra) --++(0:\ra) --++(90:\ra) --++(180:\ra) -- cycle;
		}	
	
	\foreach \k in {2,...,\a}{
		\fill[black!80] (\b*\ra,\k*\ra-\ra) --++(0:\ra) --++(90:\ra) --++(180:\ra) -- cycle;
		}
		
	\foreach \j in {2,...,\b}{
		\fill[black!80] (\j*\ra-\ra,0) --++(0:\ra) --++(90:\ra) --++(180:\ra) -- cycle;
		\fill[black!80] (\j*\ra-\ra,\a*\ra) --++(0:\ra) --++(90:\ra) --++(180:\ra) -- cycle;
		}

}

\newcommand{\snakeborderLines}[2]{
	
	\pgfmathsetmacro{\a}{#1-1}
	\pgfmathsetmacro{\b}{#2-1}
	
	\foreach \i in {1,...,\a}{
		\draw[white] (0,\i*\ra) --++(0:\ra) --++(90:\ra) --++(180:\ra) --++(270:\ra);
		}	
	
	\foreach \k in {2,...,\a}{
		\draw[white] (\b*\ra,\k*\ra-\ra) --++(0:\ra) --++(90:\ra) --++(180:\ra) --++(270:\ra);
		}
		
	\foreach \j in {2,...,\b}{
		\draw[white] (\j*\ra-\ra,0) --++(0:\ra) --++(90:\ra) --++(180:\ra) --++(270:\ra);
		\draw[white] (\j*\ra-\ra,\a*\ra) --++(0:\ra) --++(90:\ra) --++(180:\ra) --++(270:\ra);
		}

}

\newcommand{\gardenborder}[2]{
	
	\pgfmathsetmacro{\a}{#1-2}
	\pgfmathsetmacro{\b}{#2-2}
	
	\foreach \i in {1,...,\a}{
		\fill[black!80] (0,\i*\ra) --++(0:\ra) --++(90:\ra) --++(180:\ra) -- cycle;
		\fill[black!80] (\b*\ra+\ra,\i*\ra) --++(0:\ra) --++(90:\ra) --++(180:\ra) -- cycle;
		}
		
	\foreach \j in {1,...,\b}{
		\fill[black!80] (\j*\ra,0) --++(0:\ra) --++(90:\ra) --++(180:\ra) -- cycle;
		\fill[black!80] (\j*\ra,\a*\ra+\ra) --++(0:\ra) --++(90:\ra) --++(180:\ra) -- cycle;
		}

}

\newcommand{\gardenborderLines}[2]{
	
	\pgfmathsetmacro{\a}{#1-2}
	\pgfmathsetmacro{\b}{#2-2}
	
	\foreach \i in {1,...,\a}{
		\draw[white] (0,\i*\ra) --++(0:\ra) --++(90:\ra) --++(180:\ra) --++(270:\ra);
		\draw[white] (\b*\ra+\ra,\i*\ra) --++(0:\ra) --++(90:\ra) --++(180:\ra) --++(270:\ra);
		}
		
	\foreach \j in {1,...,\b}{
		\draw[white] (\j*\ra,0) --++(0:\ra) --++(90:\ra) --++(180:\ra) --++(270:\ra);
		\draw[white] (\j*\ra,\a*\ra+\ra) --++(0:\ra) --++(90:\ra) --++(180:\ra) --++(270:\ra);
		}

}

\newcommand{\spiralborder}[2]{
	
	\pgfmathsetmacro{\a}{#1-1}
	\pgfmathsetmacro{\b}{#2-2}
	
	\foreach \i in {1,...,\a}{
		\fill[black!80] (0,\i*\ra) --++(0:\ra) --++(90:\ra) --++(180:\ra) -- cycle;
		\fill[black!80] (\b*\ra+\ra,\i*\ra) --++(0:\ra) --++(90:\ra) --++(180:\ra) -- cycle;
		}
		
	\foreach \j in {1,...,\b}{
		\fill[black!80] (\j*\ra,\a*\ra) --++(0:\ra) --++(90:\ra) --++(180:\ra) -- cycle;
		}
	\foreach \j in {4,...,\b}{
		\fill[black!80] (\j*\ra-\ra,0) --++(0:\ra) --++(90:\ra) --++(180:\ra) -- cycle;
		}
	\fill[black!80] (2*\ra,\ra) --++(0:\ra) --++(90:\ra) --++(180:\ra) -- cycle;

}

\newcommand{\spiralborderLines}[2]{
	
	\pgfmathsetmacro{\a}{#1-1}
	\pgfmathsetmacro{\b}{#2-2}
	
	\foreach \i in {1,...,\a}{
		\draw[white] (0,\i*\ra) --++(0:\ra) --++(90:\ra) --++(180:\ra) --++(270:\ra);
		\draw[white] (\b*\ra+\ra,\i*\ra) --++(0:\ra) --++(90:\ra) --++(180:\ra) --++(270:\ra);
		}
		
	\foreach \j in {1,...,\b}{
		\draw[white] (\j*\ra,\a*\ra) --++(0:\ra) --++(90:\ra) --++(180:\ra) --++(270:\ra);
		}
	\foreach \j in {4,...,\b}{
		\draw[white] (\j*\ra-\ra,0) --++(0:\ra) --++(90:\ra) --++(180:\ra) --++(270:\ra);
		}
	\draw[white] (2*\ra,\ra) --++(0:\ra) --++(90:\ra) --++(180:\ra) --++(270:\ra);

}

\newcommand{\wtopcheck}[2]{

	\ifodd #1
		\pgfmathsetmacro{\x}{#1/2-3/2}
		\pgfmathsetmacro{\a}{#1/2-1/2}
	\else
		\pgfmathsetmacro{\x}{#1/2-1}
		\pgfmathsetmacro{\a}{#1/2-1}		
	\fi
	
	\ifodd #2
		\pgfmathsetmacro{\y}{#2/2-1/2}
		\pgfmathsetmacro{\b}{#2/2-1}	
	\else
		\pgfmathsetmacro{\y}{#2/2-1}	
		\pgfmathsetmacro{\b}{#2/2-1}	
	\fi
		\foreach \i in {1,...,\x}{
			\foreach \j in {1,...,\y}{
				\fill[black!80] (\j*2*\ra-\ra,#1*\ra-\ra-\i*2*\ra) --++(0:\ra) --++(90:\ra) --++(180:\ra) -- cycle;
			}
		}
		
		\foreach \i in {1,...,\a}{
			\foreach \j in {1,...,\b}{
				\fill[black!80] (\j*2*\ra,#1*\ra-\i*2*\ra) --++(0:\ra) --++(90:\ra) --++(180:\ra) -- cycle;
			}
		}
}

\newcommand{\wtopcheckLines}[2]{

	\ifodd #1
		\pgfmathsetmacro{\x}{#1/2-3/2}
		\pgfmathsetmacro{\a}{#1/2-1/2}
	\else
		\pgfmathsetmacro{\x}{#1/2-1}
		\pgfmathsetmacro{\a}{#1/2-1}		
	\fi
	
	\ifodd #2
		\pgfmathsetmacro{\y}{#2/2-1/2}
		\pgfmathsetmacro{\b}{#2/2-1}	
	\else
		\pgfmathsetmacro{\y}{#2/2-1}	
		\pgfmathsetmacro{\b}{#2/2-1}	
	\fi
		\foreach \i in {1,...,\x}{
			\foreach \j in {1,...,\y}{
				\draw[white] (\j*2*\ra-\ra,#1*\ra-\ra-\i*2*\ra) --++(0:\ra) --++(90:\ra) --++(180:\ra) --++(270:\ra);
			}
		}
		
		\foreach \i in {1,...,\a}{
			\foreach \j in {1,...,\b}{
				\draw[white] (\j*2*\ra,#1*\ra-\i*2*\ra) --++(0:\ra) --++(90:\ra) --++(180:\ra) --++(270:\ra);
			}
		}
}

\newcommand{\gtopcheck}[2]{

	\ifodd #1
		\pgfmathsetmacro{\x}{#1/2-3/2}
		\pgfmathsetmacro{\a}{#1/2-1/2}
	\else
		\pgfmathsetmacro{\x}{#1/2-1}
		\pgfmathsetmacro{\a}{#1/2-1}		
	\fi
	
	\ifodd #2
		\pgfmathsetmacro{\y}{#2/2-1/2}
		\pgfmathsetmacro{\b}{#2/2-1}	
	\else
		\pgfmathsetmacro{\y}{#2/2-1}	
		\pgfmathsetmacro{\b}{#2/2-1}	
	\fi
		\foreach \i in {1,...,\x}{
			\foreach \j in {1,...,\y}{
				\fill[newcol] (\j*2*\ra-\ra,#1*\ra-\ra-\i*2*\ra) --++(0:\ra) --++(90:\ra) --++(180:\ra) -- cycle;
			}
		}
		
		\foreach \i in {1,...,\a}{
			\foreach \j in {1,...,\b}{
				\fill[newcol] (\j*2*\ra,#1*\ra-\i*2*\ra) --++(0:\ra) --++(90:\ra) --++(180:\ra) -- cycle;
			}
		}
}

\newcommand{\uptrees}[2]{

	\ifodd #1
		\pgfmathsetmacro{\y}{#1/2+1/2}
		\pgfmathsetmacro{\x}{#1/2-1/2}
	\else
		\pgfmathsetmacro{\y}{#1/2}
		\pgfmathsetmacro{\x}{#1/2}
	\fi
	
	\foreach \i in {1,...,\y}{
			\foreach \j in {1,...,#2}{
				\fill[black!80] (\i*6*\ra-3*\ra,2*\j*\ra) --++(0:\ra) --++(90:\ra) --++(180:\ra) -- cycle;
				}
		
			}
	
	\foreach \i in {1,...,\x}{
		\foreach \j in {1,...,#2}{
			\fill[black!80] (\i*6*\ra,2*\j*\ra+\ra) --++(0:\ra) --++(90:\ra) --++(180:\ra) -- cycle;
			}
		
		}	
		
}

\newcommand{\Swl}[2]{

	    \pgfmathsetmacro{\w}{4+2*#2}
	    \pgfmathsetmacro{\z}{4+3*#1}
	    \snakeborder{\w}{\z}

    \wtopcheck{\w}{\z}
    \uptrees{#1}{#2}
    
    \snakeborderLines{\w}{\z}
    \wtopcheckLines{\w}{\z}
}

\newcommand{\coiledsnake}[1]{
	
	\pgfmathsetmacro{\w}{4+3*#1}
	\pgfmathsetmacro{\z}{3+3*#1}
	
	\snakeborder{\w}{\z}
	\wtopcheck{\w}{\z}
	
\ifnum #1=1
	\fill[black!80] (3*\ra,3*\ra) --++(0:\ra) --++(90:\ra) --++(180:\ra) -- cycle;	
		
\else	
	\ifodd #1
		\pgfmathsetmacro{\f}{#1/2+1/2}
		\pgfmathsetmacro{\g}{#1/2-1/2}
	\else
		\pgfmathsetmacro{\f}{#1/2}
		\pgfmathsetmacro{\g}{#1/2}	
	\fi
	
	\foreach \j in {1,...,\f}{
		\pgfmathsetmacro{\l}{3*\j-2}
		\foreach \k in {1,...,\l}{
				\fill[black!80] (\z*\ra-\ra-2*\k*\ra,6*\j*\ra-3*\ra) --++(0:\ra) --++(90:\ra) --++(180:\ra) -- cycle;

		}
		\ifnum \j=1
		\else
			\foreach \k in {2,...,\l}{
			\fill[black!80] (\z*\ra-\ra-2*\l*\ra,6*\j*\ra-\ra-2*\k*\ra) --++(0:\ra) --++(90:\ra) --++(180:\ra) -- cycle;
			}
		\fi
	}

	\foreach \j in {1,...,\g}{
		\pgfmathsetmacro{\l}{3*\j}
		\foreach \k in {1,...,\l}{
				\fill[black!80] (\z*\ra-6*\j*\ra,2*\k*\ra) --++(0:\ra) --++(90:\ra) --++(180:\ra) -- cycle;

		}
		\foreach \k in {3,...,\l}{
		\fill[black!80] (\z*\ra-6*\j*\ra+2*\k*\ra-4*\ra,2*\l*\ra) --++(0:\ra) --++(90:\ra) --++(180:\ra) -- cycle;
		}
	}	
\fi	

	\snakeborderLines{\w}{\z}
	\wtopcheckLines{\w}{\z}

}

\newcommand{\coiledsnakeTwo}[1]{
	
	\pgfmathsetmacro{\w}{4+3*#1}
	\pgfmathsetmacro{\z}{5+3*#1}
	
	\snakeborder{\w}{\z}
	\wtopcheck{\w}{\z}
	
\ifnum #1=1
	\fill[black!80] (3*\ra,3*\ra) --++(0:\ra) --++(90:\ra) --++(180:\ra) -- cycle;
	\fill[black!80] (5*\ra,3*\ra) --++(0:\ra) --++(90:\ra) --++(180:\ra) -- cycle;	
		
\else	
	\ifodd #1
		\pgfmathsetmacro{\f}{#1/2+1/2}
		\pgfmathsetmacro{\g}{#1/2-1/2}
	\else
		\pgfmathsetmacro{\f}{#1/2}
		\pgfmathsetmacro{\g}{#1/2}	
	\fi
	
	\foreach \j in {1,...,\f}{
		\pgfmathsetmacro{\l}{3*\j-1}
		\foreach \k in {1,...,\l}{
				\fill[black!80] (\z*\ra-\ra-2*\k*\ra,6*\j*\ra-3*\ra) --++(0:\ra) --++(90:\ra) --++(180:\ra) -- cycle;

		}
		\ifnum \j=1
		\else
			\foreach \k in {3,...,\l}{
			\fill[black!80] (\z*\ra-\ra-2*\l*\ra,6*\j*\ra-\ra-2*\k*\ra+2*\ra) --++(0:\ra) --++(90:\ra) --++(180:\ra) -- cycle;
			}
		\fi
	}

	\foreach \j in {1,...,\g}{
		\pgfmathsetmacro{\l}{3*\j}
		\foreach \k in {1,...,\l}{
				\fill[black!80] (\z*\ra-6*\j*\ra-2*\ra,2*\k*\ra) --++(0:\ra) --++(90:\ra) --++(180:\ra) -- cycle;

		}
		\foreach \k in {2,...,\l}{
		\fill[black!80] (\z*\ra-6*\j*\ra+2*\k*\ra-4*\ra,2*\l*\ra) --++(0:\ra) --++(90:\ra) --++(180:\ra) -- cycle;
		}
	}	
\fi	

	\snakeborderLines{\w}{\z}
	\wtopcheckLines{\w}{\z}

}

\newcommand{\EvenSqZero}[1]{

	\gardenborder{#1}{#1}
	\wtopcheck{#1}{#1}

	\foreach \i in {1,...,4}{
		\fill[black!80] (8*\ra,#1*\ra-2*\i*\ra-\ra) --++(0:\ra) --++(90:\ra) --++(180:\ra) -- cycle;
	}
	\foreach \i in {1,...,3}{
		\fill[black!80] (2*\i*\ra,#1*\ra-9*\ra) --++(0:\ra) --++(90:\ra) --++(180:\ra) -- cycle;
	}
	\fill[black!80] (2*\ra,#1*\ra-5*\ra) --++(0:\ra) --++(90:\ra) --++(180:\ra) -- cycle;
	\fill[black!80] (4*\ra,#1*\ra-3*\ra) --++(0:\ra) --++(90:\ra) --++(180:\ra) -- cycle;
	\fill[black!80] (6*\ra,#1*\ra-5*\ra) --++(0:\ra) --++(90:\ra) --++(180:\ra) -- cycle;
	\fill[black!80] (4*\ra,#1*\ra-7*\ra) --++(0:\ra) --++(90:\ra) --++(180:\ra) -- cycle;

    \ifnum #1=12

    \else
    	\pgfmathsetmacro{\a}{#1/2-2}
    	\pgfmathsetmacro{\b}{#1/6-2}

	    \foreach \i in {1,...,\b}{
		    \foreach \l in {1,...,\a}{
			    \fill[black!80] (#1*\ra-6*\i*\ra+2*\ra,#1*\ra-\ra-2*\l*\ra) --++(0:\ra) --++(90:\ra) --++(180:\ra) -- cycle;
		    	\fill[black!80] (#1*\ra-6*\i*\ra-\ra,2*\l*\ra) --++(0:\ra) --++(90:\ra) --++(180:\ra) -- cycle;
		    }
	
	    	\foreach \j in {1,...,4}{
		    	\fill[black!80] (2*\j*\ra,6*\i*\ra-3*\ra) --++(0:\ra) --++(90:\ra) --++(180:\ra) -- cycle;
			    \fill[black!80] (2*\j*\ra+\ra,6*\i*\ra) --++(0:\ra) --++(90:\ra) --++(180:\ra) -- cycle;
	    	}
    	}

    \fi	

	\gardenborderLines{#1}{#1}
	\wtopcheckLines{#1}{#1}

}

\newcommand{\EvenSqTwo}[1]{

	\gardenborder{#1}{#1}
	\wtopcheck{#1}{#1}

	\fill[black!80] (2*\ra,#1*\ra-5*\ra) --++(0:\ra) --++(90:\ra) --++(180:\ra) -- cycle;
	\fill[black!80] (4*\ra,#1*\ra-5*\ra) --++(0:\ra) --++(90:\ra) --++(180:\ra) -- cycle;
	\fill[black!80] (4*\ra,#1*\ra-3*\ra) --++(0:\ra) --++(90:\ra) --++(180:\ra) -- cycle;

	\ifnum #1=8
	
	\else

		\pgfmathsetmacro{\a}{#1/2-2}
		\pgfmathsetmacro{\b}{#1/6-2/6-1}

		\foreach \i in {1,...,\b}{
			\foreach \l in {1,...,\a}{
				\fill[black!80] (#1*\ra-6*\i*\ra+2*\ra,#1*\ra-\ra-2*\l*\ra) --++(0:\ra) --++(90:\ra) --++(180:\ra) -- cycle;
				\fill[black!80] (#1*\ra-6*\i*\ra-\ra,2*\l*\ra) --++(0:\ra) --++(90:\ra) --++(180:\ra) -- cycle;
			}
	
			\foreach \j in {1,2}{
				\fill[black!80] (2*\j*\ra,6*\i*\ra-3*\ra) --++(0:\ra) --++(90:\ra) --++(180:\ra) -- cycle;
				\fill[black!80] (2*\j*\ra+\ra,6*\i*\ra) --++(0:\ra) --++(90:\ra) --++(180:\ra) -- cycle;
			}
		}
	\fi

	\gardenborderLines{#1}{#1}
	\wtopcheckLines{#1}{#1}

}

\newcommand{\twospirals}[1]{

\pgfmathsetmacro{\Nplus}{#1+1}

	\ifodd #1 
		\spiralborder{\Nplus}{#1}
		\wtopcheck{\Nplus}{#1}
		\pgfmathsetmacro{\a}{#1/2-5/2} 
		\pgfmathsetmacro{\b}{#1/6-5/6}
	
		\foreach \j in {1,...,\b}{
			\foreach \i in {1,2}{
				\fill[black!80] (2*\ra+3*\j*\ra,2*\i*\ra+3*\j*\ra-3*\ra) --++(0:\ra) --++(90:\ra) --++(180:\ra) -- cycle;
			}
			\pgfmathsetmacro{\x}{\a+3-3*\j}	
			\foreach \i in {1,...,\x}{
				\fill[black!80] (3*\j*\ra,2*\i*\ra+3*\j*\ra-\ra) --++(0:\ra) --++(90:\ra) --++(180:\ra) -- cycle;
			}
			
			\foreach \i in {2,...,\x}{
				\fill[black!80] (3*\j*\ra+2*\i*\ra-2*\ra,2*\a*\ra+5*\ra-3*\j*\ra) --++(0:\ra) --++(90:\ra) --++(180:\ra) -- cycle;
			}
			
			\foreach \i in {3,...,\x}{
				\fill[black!80] (#1*\ra-\ra-3*\j*\ra,#1*\ra-3*\j*\ra-2*\i*\ra+4*\ra) --++(0:\ra) --++(90:\ra) --++(180:\ra) -- cycle;
				\fill[black!80] (#1*\ra+6*\ra-3*\j*\ra-2*\i*\ra,3*\j*\ra) --++(0:\ra) --++(90:\ra) --++(180:\ra) -- cycle;
			}			

		}

		\spiralborderLines{\Nplus}{#1}
		\wtopcheckLines{\Nplus}{#1}

	\else
		\spiralborder{#1}{\Nplus}
		\wtopcheck{#1}{\Nplus}
	
		\pgfmathsetmacro{\a}{#1/2-3} 
		\pgfmathsetmacro{\b}{#1/6-2/6-1} 

		\fill[black!80] (#1/2*\ra-\ra,#1/2*\ra) --++(0:\ra) --++(90:\ra) --++(180:\ra) -- cycle;

		\foreach \j in {0,...,\b}{
			\foreach \i in {1,2}{
				\fill[black!80] (5*\ra+3*\j*\ra,2*\i*\ra+3*\j*\ra) --++(0:\ra) --++(90:\ra) --++(180:\ra) -- cycle;
			}
		
		}
		\foreach \j in {1,...,\b}{
			\pgfmathsetmacro{\x}{\a+3-3*\j}	
			\foreach \i in {1,...,\x}{
				\fill[black!80] (3*\j*\ra,2*\i*\ra+3*\j*\ra-\ra) --++(0:\ra) --++(90:\ra) --++(180:\ra) -- cycle;
				\fill[black!80] (3*\j*\ra+2*\i*\ra,2*\a*\ra+5*\ra-3*\j*\ra) --++(0:\ra) --++(90:\ra) --++(180:\ra) -- cycle;
			}
			\foreach \i in {3,...,\x}{
				\fill[black!80] (6*\ra-3*\j*\ra+2*\a*\ra,2*\a*\ra-2*\i*\ra+9*\ra-3*\j*\ra) --++(0:\ra) --++(90:\ra) --++(180:\ra) -- cycle;
			}
			\foreach \i in {2,...,\x}{
				\fill[black!80] (#1*\ra+5*\ra-3*\j*\ra-2*\i*\ra,3*\j*\ra) --++(0:\ra) --++(90:\ra) --++(180:\ra) -- cycle;
			}
		}
		\spiralborderLines{#1}{\Nplus}
		\wtopcheckLines{#1}{\Nplus}

	\fi

}

\newcommand{\twospiralsNoB}[1]{

\pgfmathsetmacro{\Nplus}{#1+1}

	\ifodd #1 
		\wtopcheck{\Nplus}{#1}
		\pgfmathsetmacro{\a}{#1/2-5/2} 
		\pgfmathsetmacro{\b}{#1/6-5/6}
	
		\foreach \j in {1,...,\b}{
			\foreach \i in {1,2}{
				\fill[black!80] (2*\ra+3*\j*\ra,2*\i*\ra+3*\j*\ra-3*\ra) --++(0:\ra) --++(90:\ra) --++(180:\ra) -- cycle;
			}
			\pgfmathsetmacro{\x}{\a+3-3*\j}	
			\foreach \i in {1,...,\x}{
				\fill[black!80] (3*\j*\ra,2*\i*\ra+3*\j*\ra-\ra) --++(0:\ra) --++(90:\ra) --++(180:\ra) -- cycle;
			}
			
			\foreach \i in {2,...,\x}{
				\fill[black!80] (3*\j*\ra+2*\i*\ra-2*\ra,2*\a*\ra+5*\ra-3*\j*\ra) --++(0:\ra) --++(90:\ra) --++(180:\ra) -- cycle;
			}
			
			\foreach \i in {3,...,\x}{
				\fill[black!80] (#1*\ra-\ra-3*\j*\ra,#1*\ra-3*\j*\ra-2*\i*\ra+4*\ra) --++(0:\ra) --++(90:\ra) --++(180:\ra) -- cycle;
				\fill[black!80] (#1*\ra+6*\ra-3*\j*\ra-2*\i*\ra,3*\j*\ra) --++(0:\ra) --++(90:\ra) --++(180:\ra) -- cycle;
			}			

		}

		\wtopcheckLines{\Nplus}{#1}

	\else
		\wtopcheck{#1}{\Nplus}
	
		\pgfmathsetmacro{\a}{#1/2-3} 
		\pgfmathsetmacro{\b}{#1/6-2/6-1} 

		\fill[black!80] (#1/2*\ra-\ra,#1/2*\ra) --++(0:\ra) --++(90:\ra) --++(180:\ra) -- cycle;

		\foreach \j in {0,...,\b}{
			\foreach \i in {1,2}{
				\fill[black!80] (5*\ra+3*\j*\ra,2*\i*\ra+3*\j*\ra) --++(0:\ra) --++(90:\ra) --++(180:\ra) -- cycle;
			}
		
		}
		\foreach \j in {1,...,\b}{
			\pgfmathsetmacro{\x}{\a+3-3*\j}	
			\foreach \i in {1,...,\x}{
				\fill[black!80] (3*\j*\ra,2*\i*\ra+3*\j*\ra-\ra) --++(0:\ra) --++(90:\ra) --++(180:\ra) -- cycle;
				\fill[black!80] (3*\j*\ra+2*\i*\ra,2*\a*\ra+5*\ra-3*\j*\ra) --++(0:\ra) --++(90:\ra) --++(180:\ra) -- cycle;
			}
			\foreach \i in {3,...,\x}{
				\fill[black!80] (6*\ra-3*\j*\ra+2*\a*\ra,2*\a*\ra-2*\i*\ra+9*\ra-3*\j*\ra) --++(0:\ra) --++(90:\ra) --++(180:\ra) -- cycle;
			}
			\foreach \i in {2,...,\x}{
				\fill[black!80] (#1*\ra+5*\ra-3*\j*\ra-2*\i*\ra,3*\j*\ra) --++(0:\ra) --++(90:\ra) --++(180:\ra) -- cycle;
			}
		}
		\wtopcheckLines{#1}{\Nplus}

	\fi

}

\newcommand{\rearrangeZero}[1]{
              
	\pgfmathsetmacro{\c}{#1-3}
	\pgfmathsetmacro{\e}{#1/6-15/6}
		
		\foreach \i in {2,...,\c}{
			\fill[black!80] (0,\i*\ra) --++(0:\ra) --++(90:\ra) --++(180:\ra) -- cycle;
		}
		
		\foreach \i in {2,...,\c}{
			\fill[black!80] (\i*\ra,#1*\ra-\ra) --++(0:\ra) --++(90:\ra) --++(180:\ra) -- cycle;
		}
		
		\fill[black!80] (#1*\ra-3*\ra,#1*\ra-3*\ra) --++(0:\ra) --++(90:\ra) --++(180:\ra) -- cycle;	
			
		\foreach \i in {4,...,\c}{
			\fill[black!80] (#1*\ra-\ra,\i*\ra) --++(0:\ra) --++(90:\ra) --++(180:\ra) -- cycle;
		}
		
		\fill[black!80] (#1*\ra-2*\ra,3*\ra) --++(0:\ra) --++(90:\ra) --++(180:\ra) -- cycle;
		
		\foreach \i in {4,...,\c}{
			\fill[black!80] (\i*\ra,0) --++(0:\ra) --++(90:\ra) --++(180:\ra) -- cycle;
		}
		\fill[black!80] (\ra,\ra) --++(0:\ra) --++(90:\ra) --++(180:\ra) -- cycle;
		\fill[black!80] (3*\ra,\ra) --++(0:\ra) --++(90:\ra) --++(180:\ra) -- cycle;		

		\fill[black!80] (#1/2*\ra+\ra/2,#1/2*\ra-3/2*\ra) --++(0:\ra) --++(90:\ra) --++(180:\ra) -- cycle;
		\fill[black!80] (#1/2*\ra+3*\ra/2,#1/2*\ra-9/2*\ra) --++(0:\ra) --++(90:\ra) --++(180:\ra) -- cycle;

		\wtopcheck{#1}{#1}
		
		\pgfmathsetmacro{\d}{#1/2-5/2}
					
			\foreach \i in {1,...,\d} {
				\fill[black!80] (3*\ra,2*\i*\ra+\ra) --++(0:\ra) --++(90:\ra) --++(180:\ra) -- cycle;
			}
			\foreach \i in {3,...,\d} {
				\fill[black!80] (2*\i*\ra+3*\ra-4*\ra,2*\d*\ra+\ra) --++(0:\ra) --++(90:\ra) --++(180:\ra) -- cycle;
			}
			\fill[black!80] (#1*\ra-6*\ra,2*\d*\ra-\ra) --++(0:\ra) --++(90:\ra) --++(180:\ra) -- cycle;
			\foreach \i in {2,...,\d} {
				\fill[black!80] (#1*\ra-4*\ra,2*\i*\ra-\ra) --++(0:\ra) --++(90:\ra) --++(180:\ra) -- cycle;
			}
			\foreach \i in {2,...,\d}{
				\fill[black!80] (6*\ra,2*\i*\ra-2*\ra) --++(0:\ra) --++(90:\ra) --++(180:\ra) -- cycle;
			}
			
		\ifnum #1 = 15
		\else
			\foreach \j in {1,...,\e}	{
				\pgfmathsetmacro{\d}{3*\e+3-3*\j}
				\foreach \i in {1,...,\d}{
					\fill[black!80] (3*\ra+3*\j*\ra+2*\i*\ra,#1*\ra-4*\ra-3*\j*\ra) --++(0:\ra) --++(90:\ra) --++(180:\ra) -- cycle;
				}
			}
			\foreach \j in {1,...,\e}	{
				\pgfmathsetmacro{\d}{3*\e+4-3*\j}
					\foreach \i in {1,...,\d}{
						\fill[black!80] (6*\ra+3*\j*\ra,3*\j*\ra+2*\i*\ra) --++(0:\ra) --++(90:\ra) --++(180:\ra) -- cycle;
					}
			}
			
			\foreach \j in {1,...,\e}	{
				\fill[black!80] (#1*\ra-6*\ra-3*\j*\ra,#1*\ra-4*\ra-3*\j*\ra-2*\ra) --++(0:\ra) --++(90:\ra) --++(180:\ra) -- cycle;
				\pgfmathsetmacro{\d}{3*\e+4-3*\j}
					\foreach \i in {1,...,\d}{
						\fill[black!80] (#1*\ra-4*\ra-3*\j*\ra,#1*\ra-4*\ra-3*\j*\ra-2*\i*\ra) --++(0:\ra) --++(90:\ra) --++(180:\ra) -- cycle;
					}
			}
			\foreach \j in {1,...,\e}	{
				\pgfmathsetmacro{\d}{3*\e+4-3*\j}
					\foreach \i in {1,...,\d}{
						\fill[black!80] (4*\ra+3*\j*\ra+2*\i*\ra,3*\j*\ra) --++(0:\ra) --++(90:\ra) --++(180:\ra) -- cycle;
					}
			}
			
		\fi
		\foreach \i in {2,...,\c}{
			\draw[white] (0,\i*\ra) --++(0:\ra) --++(90:\ra) --++(180:\ra) --++(270:\ra);
		}
		
		\foreach \i in {2,...,\c}{
			\draw[white] (\i*\ra,#1*\ra-\ra) --++(0:\ra) --++(90:\ra) --++(180:\ra) --++(270:\ra);
		}
		
		\draw[white] (#1*\ra-3*\ra,#1*\ra-3*\ra) --++(0:\ra) --++(90:\ra) --++(180:\ra) --++(270:\ra);
				
		\foreach \i in {4,...,\c}{
			\draw[white] (#1*\ra-\ra,\i*\ra) --++(0:\ra) --++(90:\ra) --++(180:\ra) --++(270:\ra);
		}
		
		\draw[white] (#1*\ra-2*\ra,3*\ra) --++(0:\ra) --++(90:\ra) --++(180:\ra) --++(270:\ra);	
		
		\foreach \i in {4,...,\c}{
			\draw[white] (\i*\ra,0) --++(0:\ra) --++(90:\ra) --++(180:\ra) --++(270:\ra);
		}
		
		\draw[white] (\ra,\ra) --++(0:\ra) --++(90:\ra) --++(180:\ra) --++(270:\ra);
		\draw[white] (3*\ra,\ra) --++(0:\ra) --++(90:\ra) --++(180:\ra) --++(270:\ra);		
	
		\wtopcheckLines{#1}{#1}

}

\newcommand{\rearrangeOne}[1]{
              
	\pgfmathsetmacro{\c}{#1-3}
	\pgfmathsetmacro{\e}{#1/6-13/6}
		
		\foreach \i in {2,...,\c}{
			\fill[black!80] (0,\i*\ra) --++(0:\ra) --++(90:\ra) --++(180:\ra) -- cycle;
		}
		
		\foreach \i in {2,...,\c}{
			\fill[black!80] (\i*\ra,#1*\ra-\ra) --++(0:\ra) --++(90:\ra) --++(180:\ra) -- cycle;
		}
		
		\fill[black!80] (#1*\ra-2*\ra,#1*\ra-2*\ra) --++(0:\ra) --++(90:\ra) --++(180:\ra) -- cycle;
		
		\foreach \i in {4,...,\c}{
			\fill[black!80] (#1*\ra-\ra,\i*\ra) --++(0:\ra) --++(90:\ra) --++(180:\ra) -- cycle;
		}
		
		\fill[black!80] (#1*\ra-2*\ra,3*\ra) --++(0:\ra) --++(90:\ra) --++(180:\ra) -- cycle;
		
		\foreach \i in {4,...,\c}{
			\fill[black!80] (\i*\ra,0) --++(0:\ra) --++(90:\ra) --++(180:\ra) -- cycle;
		}
		\fill[black!80] (\ra,\ra) --++(0:\ra) --++(90:\ra) --++(180:\ra) -- cycle;
		\fill[black!80] (3*\ra,\ra) --++(0:\ra) --++(90:\ra) --++(180:\ra) -- cycle;		

		\wtopcheck{#1}{#1}
		
		\pgfmathsetmacro{\d}{#1/2-5/2}
			\foreach \i in {1,...,\d} {
				\fill[black!80] (3*\ra,2*\i*\ra+\ra) --++(0:\ra) --++(90:\ra) --++(180:\ra) -- cycle;
			}
			\foreach \i in {2,...,\d} {
				\fill[black!80] (2*\i*\ra+3*\ra-2*\ra,2*\d*\ra+\ra) --++(0:\ra) --++(90:\ra) --++(180:\ra) -- cycle;
			}
			\foreach \i in {1,...,\d} {
				\fill[black!80] (#1*\ra-4*\ra,2*\i*\ra+\ra) --++(0:\ra) --++(90:\ra) --++(180:\ra) -- cycle;
			}
			\foreach \i in {2,...,\d}{
				\fill[black!80] (6*\ra,2*\i*\ra-2*\ra) --++(0:\ra) --++(90:\ra) --++(180:\ra) -- cycle;
			}
			
		\ifnum #1 = 13
		\else
			\foreach \j in {1,...,\e}	{
				\pgfmathsetmacro{\d}{#1/2-7/2-3*\j}
				\foreach \i in {1,...,\d}{
					\fill[black!80] (3*\ra+3*\j*\ra+2*\i*\ra,#1*\ra-4*\ra-3*\j*\ra) --++(0:\ra) --++(90:\ra) --++(180:\ra) -- cycle;
				}
			}
			\foreach \j in {1,...,\e}	{
				\pgfmathsetmacro{\d}{3*\e+3-3*\j}
					\foreach \i in {1,...,\d}{
						\fill[black!80] (6*\ra+3*\j*\ra,3*\j*\ra+2*\i*\ra) --++(0:\ra) --++(90:\ra) --++(180:\ra) -- cycle;
					}
			}
			
			\foreach \j in {1,...,\e}	{
				\pgfmathsetmacro{\d}{3*\e+3-3*\j}
					\foreach \i in {1,...,\d}{
						\fill[black!80] (#1*\ra-4*\ra-3*\j*\ra,#1*\ra-4*\ra-3*\j*\ra-2*\i*\ra) --++(0:\ra) --++(90:\ra) --++(180:\ra) -- cycle;
					}
			}
			\foreach \j in {1,...,\e}	{
				\pgfmathsetmacro{\d}{3*\e+3-3*\j}
					\foreach \i in {1,...,\d}{
						\fill[black!80] (4*\ra+3*\j*\ra+2*\i*\ra,3*\j*\ra) --++(0:\ra) --++(90:\ra) --++(180:\ra) -- cycle;
					}
			}

		\fi
		\foreach \i in {2,...,\c}{
			\draw[white] (0,\i*\ra) --++(0:\ra) --++(90:\ra) --++(180:\ra) --++(270:\ra);
		}
		
		\foreach \i in {2,...,\c}{
			\draw[white] (\i*\ra,#1*\ra-\ra) --++(0:\ra) --++(90:\ra) --++(180:\ra) --++(270:\ra);
		}
		
		\draw[white] (#1*\ra-2*\ra,#1*\ra-2*\ra) --++(0:\ra) --++(90:\ra) --++(180:\ra) --++(270:\ra);
		
		\foreach \i in {4,...,\c}{
			\draw[white] (#1*\ra-\ra,\i*\ra) --++(0:\ra) --++(90:\ra) --++(180:\ra) --++(270:\ra);
		}
		
		\draw[white] (#1*\ra-2*\ra,3*\ra) --++(0:\ra) --++(90:\ra) --++(180:\ra) --++(270:\ra);	
		
		\foreach \i in {4,...,\c}{
			\draw[white] (\i*\ra,0) --++(0:\ra) --++(90:\ra) --++(180:\ra) --++(270:\ra);
		}
		
		\draw[white] (\ra,\ra) --++(0:\ra) --++(90:\ra) --++(180:\ra) --++(270:\ra);
		\draw[white] (3*\ra,\ra) --++(0:\ra) --++(90:\ra) --++(180:\ra) --++(270:\ra);		
	
		\wtopcheckLines{#1}{#1}

}

\newcommand{\rearrangeTwo}[1]{
              
	\pgfmathsetmacro{\c}{#1-3}
	\pgfmathsetmacro{\e}{#1/6-17/6}
		
		\foreach \i in {2,...,\c}{
			\fill[black!80] (0,\i*\ra) --++(0:\ra) --++(90:\ra) --++(180:\ra) -- cycle;
		}
		
		\foreach \i in {2,...,\c}{
			\fill[black!80] (\i*\ra,#1*\ra-\ra) --++(0:\ra) --++(90:\ra) --++(180:\ra) -- cycle;
		}
		
		\fill[black!80] (#1*\ra-3*\ra,#1*\ra-3*\ra) --++(0:\ra) --++(90:\ra) --++(180:\ra) -- cycle;	
			
		\foreach \i in {4,...,\c}{
			\fill[black!80] (#1*\ra-\ra,\i*\ra) --++(0:\ra) --++(90:\ra) --++(180:\ra) -- cycle;
		}
		
		\fill[black!80] (#1*\ra-2*\ra,3*\ra) --++(0:\ra) --++(90:\ra) --++(180:\ra) -- cycle;
		
		\foreach \i in {4,...,\c}{
			\fill[black!80] (\i*\ra,0) --++(0:\ra) --++(90:\ra) --++(180:\ra) -- cycle;
		}
		\fill[black!80] (\ra,\ra) --++(0:\ra) --++(90:\ra) --++(180:\ra) -- cycle;
		\fill[black!80] (3*\ra,\ra) --++(0:\ra) --++(90:\ra) --++(180:\ra) -- cycle;		

		\fill[black!80] (#1/2*\ra-\ra/2,#1/2*\ra-5/2*\ra) --++(0:\ra) --++(90:\ra) --++(180:\ra) -- cycle;
		\fill[black!80] (#1/2*\ra+3*\ra/2,#1/2*\ra-5/2*\ra) --++(0:\ra) --++(90:\ra) --++(180:\ra) -- cycle;
		\fill[black!80] (#1/2*\ra-\ra/2,#1/2*\ra+3/2*\ra) --++(0:\ra) --++(90:\ra) --++(180:\ra) -- cycle;
		\fill[black!80] (#1/2*\ra+3*\ra/2,#1/2*\ra-\ra/2) --++(0:\ra) --++(90:\ra) --++(180:\ra) -- cycle;

		\wtopcheck{#1}{#1}
		
		\pgfmathsetmacro{\d}{#1/2-5/2}
					
			\foreach \i in {1,...,\d} {
				\fill[black!80] (3*\ra,2*\i*\ra+\ra) --++(0:\ra) --++(90:\ra) --++(180:\ra) -- cycle;
			}
			\foreach \i in {3,...,\d} {
				\fill[black!80] (2*\i*\ra+3*\ra-4*\ra,2*\d*\ra+\ra) --++(0:\ra) --++(90:\ra) --++(180:\ra) -- cycle;
			}
			\fill[black!80] (#1*\ra-6*\ra,2*\d*\ra-\ra) --++(0:\ra) --++(90:\ra) --++(180:\ra) -- cycle;
			\foreach \i in {2,...,\d} {
				\fill[black!80] (#1*\ra-4*\ra,2*\i*\ra-\ra) --++(0:\ra) --++(90:\ra) --++(180:\ra) -- cycle;
			}
			\foreach \i in {2,...,\d}{
				\fill[black!80] (6*\ra,2*\i*\ra-2*\ra) --++(0:\ra) --++(90:\ra) --++(180:\ra) -- cycle;
			}
			
		\ifnum #1 = 17
			\fill[black!80] (9*\ra,3*\ra) --++(0:\ra) --++(90:\ra) --++(180:\ra) -- cycle;
			\fill[black!80] (11*\ra,3*\ra) --++(0:\ra) --++(90:\ra) --++(180:\ra) -- cycle;
			
		\else
			\foreach \j in {1,...,\e}	{
				\pgfmathsetmacro{\d}{3*\e+4-3*\j}
				\foreach \i in {1,...,\d}{
					\fill[black!80] (3*\ra+3*\j*\ra+2*\i*\ra,#1*\ra-4*\ra-3*\j*\ra) --++(0:\ra) --++(90:\ra) --++(180:\ra) -- cycle;
				}
			}
			\foreach \j in {1,...,\e}	{
				\pgfmathsetmacro{\d}{3*\e+5-3*\j}
					\foreach \i in {1,...,\d}{
						\fill[black!80] (6*\ra+3*\j*\ra,3*\j*\ra+2*\i*\ra) --++(0:\ra) --++(90:\ra) --++(180:\ra) -- cycle;
					}
			}
			
			\foreach \j in {1,...,\e}	{
				\fill[black!80] (#1*\ra-6*\ra-3*\j*\ra,#1*\ra-4*\ra-3*\j*\ra-2*\ra) --++(0:\ra) --++(90:\ra) --++(180:\ra) -- cycle;
				\pgfmathsetmacro{\d}{3*\e+5-3*\j}
					\foreach \i in {1,...,\d}{
						\fill[black!80] (#1*\ra-4*\ra-3*\j*\ra,#1*\ra-4*\ra-3*\j*\ra-2*\i*\ra) --++(0:\ra) --++(90:\ra) --++(180:\ra) -- cycle;
					}
			}
			\foreach \j in {0,...,\e}	{
				\pgfmathsetmacro{\d}{3*\e+2-3*\j}
					\foreach \i in {1,...,\d}{
						\fill[black!80] (7*\ra+3*\j*\ra+2*\i*\ra,3*\j*\ra+3*\ra) --++(0:\ra) --++(90:\ra) --++(180:\ra) -- cycle;
					}
			}
			
		\fi
		\foreach \i in {2,...,\c}{
			\draw[white] (0,\i*\ra) --++(0:\ra) --++(90:\ra) --++(180:\ra) --++(270:\ra);
		}
		
		\foreach \i in {2,...,\c}{
			\draw[white] (\i*\ra,#1*\ra-\ra) --++(0:\ra) --++(90:\ra) --++(180:\ra) --++(270:\ra);
		}
		
		\draw[white] (#1*\ra-3*\ra,#1*\ra-3*\ra) --++(0:\ra) --++(90:\ra) --++(180:\ra) --++(270:\ra);
				
		\foreach \i in {4,...,\c}{
			\draw[white] (#1*\ra-\ra,\i*\ra) --++(0:\ra) --++(90:\ra) --++(180:\ra) --++(270:\ra);
		}
		
		\draw[white] (#1*\ra-2*\ra,3*\ra) --++(0:\ra) --++(90:\ra) --++(180:\ra) --++(270:\ra);	
		
		\foreach \i in {4,...,\c}{
			\draw[white] (\i*\ra,0) --++(0:\ra) --++(90:\ra) --++(180:\ra) --++(270:\ra);
		}
		
		\draw[white] (\ra,\ra) --++(0:\ra) --++(90:\ra) --++(180:\ra) --++(270:\ra);
		\draw[white] (3*\ra,\ra) --++(0:\ra) --++(90:\ra) --++(180:\ra) --++(270:\ra);		
	
		\wtopcheckLines{#1}{#1}

}

\begin{document}
\vspace{-1cm}
\tikzset
{
  x=.23in,
  y=.23in,
}

\title{Topology and Geometry of Crystallized Polyominoes}
\author{Greg Malen \footnote{ Mathematics Department, Duke University, gmalen@math.duke.edu}\quad and\quad\'Erika Rold\'an-Roa\footnote{Mathematics Department, The Ohio State University, roldanroa.1@osu.edu }}
\date{\small{October 2019}}
\maketitle
\makeatletter
\newcommand{\leqnos}{\tagsleft@true\let\veqno\@@leqno}
\newcommand{\reqnos}{\tagsleft@false\let\veqno\@@eqno}
\reqnos
\makeatother

\everymath{\displaystyle}

\vspace{-.5cm}
\begin{abstract}
 We give a complete solution to the extremal topological combinatorial problem of finding the minimum number of tiles needed to construct a polyomino with $h$ holes. We denote this number by $g(h)$ and say that a polyomino is crystallized if it has $h$ holes and $g(h)$ tiles. We analyze structural properties of crystallized polyominoes and characterize their efficiency by a topological isoperimetric inequality that relates minimum perimeter, the area of the holes, and the structure of the dual graph of a polyomino. We also develop a new dynamical method of creating sequences of polyominoes which is invariant with respect to crystallization and efficient structure. Using this technique, we prove that the sequence constructed in \cite{kahle2018polyominoes} are the unique free crystallized polyominoes with $h_l=(2^{2l}-1)/3$ holes. For $h\leq 8$ the values of $g(h)$ were originally computed by Tomas Olivera e Silva \cite{bworld}, and for the sequence $h_l=(2^{2l}-1)/3$  by Kahle and R\'oldan-Roa \cite{kahle2018polyominoes}, who also showed that asymptotically $g(h) \approx 2h$.

 \vspace{.3cm}

\end{abstract}
\section{Introduction}\label{S:intro}

A \textit{polyomino} is a planar shape formed by gluing together a finite number of congruent squares along their edges. If two squares of a polyomino intersect, then their intersection is an entire edge, and the gluing requirement implies that a polyomino must have a connected interior. We refer to the squares on a polyomino as either \textit{squares} or \textit{tiles}. 

In this paper, we are interested in the extremal topological problem of finding the minimum number of tiles required for a polyomino to have a specified number of holes. To be precise about the topology, we consider the tiles of a polyomino to be closed. Polyominoes are finite unions of these closed tiles, so they are compact. The \textit{holes} of a polyomino are the bounded, connected components of the polyomino's complement in the plane. For a polyomino $A$ we denote its number of holes and tiles by $h(A)$ and $|A|$, respectively.

\begin{definition}
For $h\ge 1$, we define the sequence of the minimum number of tiles needed for constructing a polyomino with $h$ holes as
\begin{equation}
g(h):= \min_{h(A)=h}|A|.
\end{equation} 
\end{definition}
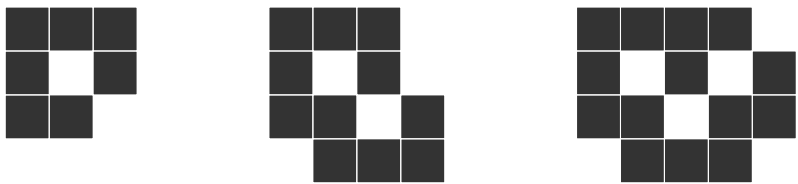
\begin{figure}[H] 
\begin{center} 
\begin{tikzpicture} 
\foreach \i/\j in {-1/-1, 0/-1, 1/1,  -1 / 0, -1/1, 0/1, 1/0 } { \fill[black!80] (\i,\j+1) --++(0:1) --++(90:1) --++(180:1) -- cycle;}

\foreach \i/\j in {-1/-1, 0/-1,1/1,  -1 / 0, -1/1, 0/1, 1/0 } { \draw[white] (\i,\j+1) --++(0:1) --++(90:1) --++(180:1) --++(270:1);}


\foreach \i/\j in {-1/-1, 0/-1, 1/1,  -1 / 0, -1/1, 0/1, 1/0, 0/-2, 1/-2, 2/-2, 2/-1 } { \fill[black!80] (\i+6,\j+1) --++(0:1) --++(90:1) --++(180:1) -- cycle;}


\foreach \i/\j in {-1/-1, 0/-1,1/1,  -1 / 0, -1/1, 0/1, 1/0, 0/-2, 1/-2, 2/-2 , 2/-1 } { \draw[white] (\i+6,\j+1) --++(0:1) --++(90:1) --++(180:1) --++(270:1);}

\foreach \i/\j in {-1/0, -2/0,-2/1, -2/2, 0 / 1, 1 / 0, 2 / 0, 2 / 1, 0 / 2, 1  / 2, -1/2, -1/-1, 0/-1, 1/-1 } { \fill[black!80] (\i+14,\j) --++(0:1) --++(90:1) --++(180:1) -- cycle;}

\foreach \i/\j in {-1/0, -2/0,-2/1, -2/2, 0 / 1, 1 / 0, 2 / 0, 2 / 1, 0 / 2, 1  / 2, -1/2, -1/-1, 0/-1, 1/-1 } { \draw[white] (\i+14,\j) --++(0:1) --++(90:1) --++(180:1) --++(270:1);}

\end{tikzpicture}

\end{center}

\caption{ Crystallized polyominoes for $1 \leq h\leq 3$. } 
\label{three_holes} 
\end{figure}


\begin{definition}
A polyomino with $h$ holes is \textit{crystallized} if it has $g(h)$ tiles.
\end{definition}

In \cite{bworld}, all polyominoes with $n\leq 28$ tiles were enumerated by their number of holes. These computations give the values of $g(h)$ for $1\leq h\leq 8$ and also count the number of crystallized polyominoes for $1\leq h\leq 8$. To be precise, we enumerate here the class of \textit{free polyominoes}, in which two polyominoes are the same if they are equivalent up to rotations, reflections, and translations (see Table \ref{tab:my_label}).

In \cite{kahle2018polyominoes}, the related function $f(n)$, defined as the maximum number of holes that a polyomino with $n$ tiles can have, was introduced and its asymptotic behavior was studied. It was proved that $f(n) = (1/2) n + O(\sqrt{n}) + o(1)$. In the same paper, a sequence of crystallized polyominoes was constructed with $g(h)+h=(2^l + 1)^2-1$ for all $l\geq 1$ (see Table \ref{tab:my_label}). We will refer to this sequence of crystallized polyominoes as the K--R sequence, and we denote its elements by $S_l$ for $l\geq 1$.

\begin{table}[H]
    \centering
\begin{tabular} {  c | c|c|c|c|c|c|c | c| c  }

 $h$ & 1 & 2 & 3 & 4 & 5 & 6 & 7 & 8 & $(2^{2l}-1)/3$ \\
 \hline
 $g(h)$ & 7  & 11 & 14 & 17 & 19 & 23 & 25 & 28 & $[(2^{2l+1}+ 3 \cdot 2 ^ {l+1} + 4)/3] -1$  \\
\hline
$\mid \mathcal{G}_{h} \mid$ & 1 & 4 & 3 & 8 & 1 & 64 & 4 & 37 &  $1^*$\\
\end{tabular}
    \caption{We denote by $\mid \mathcal{G}_{h} \mid$ the number of free polyominoes that have $h$ holes and $g(h)$ tiles. The table shows all previously known values for $g(h)$ and $\mid \mathcal{G}_{h} \mid$.  The values for $1\leq h\leq8$ are from \cite{bworld}, and the first two rows of the last column are from \cite{kahle2018polyominoes}. *We prove in Theorem \ref{T:uniqueness} that $\mid \mathcal{G}_{(2^{2l}-1)/3} \mid=1$ for all $l\geq 1$.}
    \label{tab:my_label}
\end{table}

In \cite{malen2019polyiamonds}, the analogous problem for polyiamonds, polyforms constructed with equilateral triangles, was completely solved. In that setting, all crystallized polyiamonds were also found to satisfy a set of optimal geometric and topological structural conditions. 


Here, we completely answer these questions for polyominoes. We find all values of $g(h)$, and we characterize the values of $h$ for which optimal structural conditions, which we define precisely in Section \ref{Ss:mainthms}, are satisfied. In the process of doing so, we are also able to solve the enumerative combinatorial question of how many crystallized polyominoes exist with $h_l=(2^{2l}-1)/3$ holes, proving that the sequence of crystallized polyominoes constructed in \cite{kahle2018polyominoes} are unique.

\section{Statement of Main Results}\label{Ss:mainthms}

We use the sequences of perfect squares and pronic numbers, $N^2$ and $N(N+1)$, to produce benchmark values for $g(h)$, tracking the maximum number of holes that can fit in a certain area. For $\alpha$ a square or pronic number, let $h_{\alpha}$ denote the maximum number of holes which can exist in a polyomino $A$ which is contained in the square or pronic rectangle of area $\alpha$.




\begin{theorem}\label{T:holesminper}
For any positive integer $N\ge 3$,
\begin{equation}\label{hsq}
    h_{N^2}= 
    \begin{cases}
        \frac{(N-1)^2}{3}-1 & \text{ if $N\equiv 1$ mod 3}\\[8pt]
        \frac{N(N-2)}{3} & \text{ if $N=2^l+1$ for $l \geq 1$}\\[8pt]
        \frac{N(N-2)}{3}-1 & \text{else},
    \end{cases}
\end{equation}
and
\begin{equation}\label{hpr}
    h_{N(N+1)}= 
    \begin{cases}
        \frac{N(N-1)}{3}-1 & \text{ if $N\equiv 0$ or 1 mod 3}\\[8pt]
        \frac{(N+1)(N-2)}{3}-1 & \text{ if $N\equiv 2$ mod 3}.
    \end{cases}
\end{equation}
\end{theorem}

For all square and pronic numbers up to $\alpha=36$, the values of $h_{\alpha}$ and $g(h_{\alpha})$ can be determined from the values computed in \cite{bworld}, listed in Table \ref{tab:my_label}. For those first few values it can happen that $h_{N^2}=h_{N(N+1)}$. To avoid this complication, and because those values are already known, the next theorem gives the values of $g(h_\alpha)$ for $\alpha\geq 36$.

\begin{theorem}\label{T:gpronicperfectsquare}
For $\alpha\in\{N^2, N(N+1) : N \ge 6\}$, $g(h_\alpha)=\alpha-h_\alpha-C$ for
\begin{equation}\label{constants}
    C= 
    \begin{cases}
        1 & \text{ if } \alpha=N^2 \text{ and } N=2^l+1\\[6pt]
        3 & \text{ if } \alpha=N^2 \text{ and } N\equiv 1 \text{ mod } 3\\
          & \text{ or } \alpha=N(N+1) \text{ and } N\not\equiv 2 \text{ mod } 3\\[6pt]
        4 & \text{ if } \alpha=N^2 \text{ and } N\not\equiv 1 \text{ mod } 3, N\neq2^l+1\\[6pt]
        5 & \text{ if } \alpha=N(N+1) \text{ and } N\equiv 2 \text{ mod } 3.\\    
        \end{cases}
\end{equation}
\end{theorem}

In the above, $C$ is the number of tiles missing from the boundary in a crystallized polyomino which has the maximum number of holes in a square or pronic rectangle. For instance, the $\alpha=6^2$ case gives $h_{\alpha}=7$ and $C=4$, and confirms that $g(7)=36-7-4=25$ from Table \ref{tab:my_label}. We are now able to give the values for $g(h)$ for all $h \geq  1$, in terms of the values of $g(h_{\alpha})$.

\begin{theorem}\label{ghsolved}
 For $h\geq 1$, let $\alpha =\min\{N^2, N(N+1) : h \le h_{\alpha}\}$. Then 
 \begin{equation}
     g(h)=g(h_{\alpha})-2(h_{\alpha}-h). 
 \end{equation}
\end{theorem}
Values of $g(h)$ are given in Table \ref{Table:113} for $9\leq h\leq 113$, with values of $h_{\alpha}$ in bold. Then in Theorems \ref{T:efficient} and \ref{T:efficientall} we give structural characterizations of crystallization for all $h\ge1$.

 

The \textit{dual graph} of a polyomino is the graph whose vertices are indexed by the tiles of the polyomino, with edges between two vertices if and only if the respective tiles in the polyomino share an edge. We say that a polyomino is \textit{acyclic} if its dual graph is a tree, and refer to cycles in the dual graph as \emph{dual cycles}.

The \emph{area of a hole} is defined to be the number of tiles needed to fill it.

A polyomino with $h$ holes and $n$ tiles has \textit{minimal outer perimeter} if it has an outer perimeter equal to $2 \left\lceil2\sqrt{n+h}\right\rceil$.

\begin{definition}\label{D:efficientlystructured}
A polyomino is \textit{efficiently structured} if it is acyclic, each hole has an area of one, and it has minimal outer perimeter.
\end{definition}

\begin{theorem}\label{T:efficient}
 Efficiently structured polyominoes are crystallized. 
\end{theorem}

However, the converse of Theorem \ref{T:efficient} does not hold. We prove in Theorem \ref{T:efficientall} that there is an exceptional set of crystallized polyominoes which are acyclic with all holes having of area of one, but which fail to attain minimal outer perimeter. As a corollary, this implies that all crystallized polyominoes are in fact acyclic with all holes having an area of one.

For the cases from Theorem \ref{T:gpronicperfectsquare} in which $C=4$ or $C=5$, define
\begin{align*}
S &= \{N^2 \mid N \equiv 0 \text{ or } 2 \text{ mod } 3 \text{ and } N \neq 2^l + 1 \text{ for any } l \geq 1\},\\
R &= \{N(N+1) \mid N \equiv 2 \text{ mod } 3 \}.
\end{align*}

\begin{theorem}\label{T:efficientall}
 For $\alpha \in S \cup R$, a crystallized polyomino with $h_{\alpha}+1$ holes is acyclic and each of its holes have an area of one, but it does not attain minimal outer perimeter. For all other $h$, a crystallized polyomino is efficiently structured.
\end{theorem}

Finally, we consider the enumeration of crystallized polyominoes, proving that elements of the K--R sequence in \cite{kahle2018polyominoes} are the unique crystallized polyominoes with $h=(2^{2l}-1)/3$ for all $l\geq 1$.

\begin{theorem}\label{T:uniqueness}[K--R Sequence Uniqueness]
For a fixed integer $l\ge 1$, there is only one free crystallized polyomino $A$ with $h(A)=(2^{2l}-1)/3$.
\end{theorem}
The rest of the paper is organized as follows. In Section \ref{S:Theoremefficient}, we give background, establish preliminary results, and prove Theorem \ref{T:efficient}. In Section \ref{sect:obstruct}, we exhibit various obstructions to efficient crystallization and prove Theorem \ref{T:efficientall}. In Section \ref{S:Expansion}, we prove Theorems \ref{T:holesminper}, \ref{T:gpronicperfectsquare}, and \ref{T:uniqueness}. Finally, in Section \ref{S:destroy}, we prove Theorem \ref{ghsolved} by dismantling the crystallized polyominoes that we construct in Section \ref{constructions}.

\section{Background and Preliminary Results}\label{S:Theoremefficient}

\subsection{Basic Definitions and Terminology}

To better discuss various permissible sub-arrangements of polyominoes, we introduce some terminology and notation. A space in the square lattice is said to be \emph{filled} if it contains a tile, and \emph{empty} if not. 

The \emph{total area} of a polyomino is the number of tiles plus the aggregate area of all its holes, and the \textit{perimeter} $p(A)$ of a polyomino $A$ is defined to be the number of edges that are part of the topological boundary of $A$. An edge of the perimeter is on the \textit{hole perimeter} of $A$ if it is bounding a hole, and it is on the \textit{outer perimeter} otherwise. The number of edges on the hole perimeter is denoted by $p_h(A)$, and the number of edges on the outer perimeter of $A$ is denoted by $p_{o}(A)$. As an example, let $A$ be the polyomino in the center of Figure \ref{three_holes}, then $p(A)=30$, $p_o(A)=18$, and $p_h(A)=12$.

The \emph{bounding rectangle} of a polyomino $A$ is the smallest rectangle in which $A$ fits.  The \emph{boundary layer} of a polyomino is the set of squares that have at least one edge on the outer perimeter, and the \emph{interior} of a polyomino is the set of spaces which are not in the boundary layer. For a polyomino with a rectangular interior, say $(w-2)\times(l-2)$, define $D_1$ to be the arrangement in the boundary layer of the $w\times l$ rectangle in which all spaces except for the four corners are filled. Similarly, define $D_2$ to be the boundary arrangement in which all but one of the corners are filled. For $D_1$ and $D_2$ we suppress the dimensions of the rectangle, as these will always be clear from context.

\begin{figure}[H]
    \centering
    \begin{subfigure}[t]{.4\textwidth}
        \centering
        \begin{tikzpicture}
            \gardenborder{5}{6}
            \foreach \i in {1,...,4}{
                \foreach \j in {1,2,3}{
                   \fill[newcol!40] (\i*\ra,\j*\ra) --++(0:\ra) --++(90:\ra) --++(180:\ra) -- cycle; 
                }
            }
            \gardenborderLines{5}{6}
            \foreach \i in {1,...,4}{
                \foreach \j in {1,2,3}{
                   \draw[white] (\i*\ra,\j*\ra) --++(0:\ra) --++(90:\ra) --++(180:\ra) --++(270:\ra); 
                }
            }
        \end{tikzpicture}
        \label{fig:d1}
    \end{subfigure}
\hspace{.3cm}
    \begin{subfigure}[t]{.4\textwidth}
    \centering
        \begin{tikzpicture}
            \fullborder{5}{6}
            \foreach \i in {1,...,4}{
                \foreach \j in {1,2,3}{
                   \fill[newcol!40] (\i*\ra,\j*\ra) --++(0:\ra) --++(90:\ra) --++(180:\ra) -- cycle; 
                }
            }
            \fullborderLines{5}{6}
            \foreach \i in {1,...,4}{
                \foreach \j in {1,2,3}{
                   \draw[white] (\i*\ra,\j*\ra) --++(0:\ra) --++(90:\ra) --++(180:\ra) --++(270:\ra); 
                }
            }
        \end{tikzpicture}
        \label{fig:D2}
    \end{subfigure}
    \label{dmincheck}
    \caption{$D_1$ and $D_2$ for a $3\times 4$ rectangular interior.}
\end{figure}
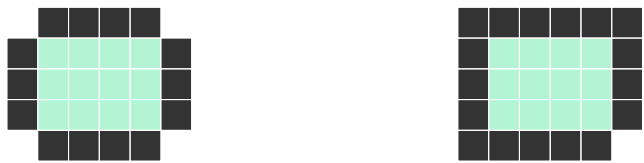

\begin{remark}\label{area}
Observe that $N^2-(N-r)(N+r)=r^2$, and $N(N+1)-(N-r)(N+1+r)=r^2+r$. These values track the decrease in area if a square or pronic rectangle is replaced by a narrower rectangle with the same perimeter. In particular the difference is always at least one for squares, and at least two for pronic rectangles.
\end{remark}

\subsection{Hole Connectivity and the Dual Graph of a Polyomino.}\label{S:dualgraphs} 
We defined the dual graph of a polyomino in Section \ref{Ss:mainthms}. In the literature, the dual graph of a polyomino is commonly referred to as a lattice animal. Let $b(A)$ be the number of edges in the dual graph of a polyomino $A$. Since the interior of $A$ must be connected, its dual graph must be connected and therefore have a spanning tree. Thus if $A$ has $n$ tiles, the dual graph has $n$ vertices and its spanning tree has $n-1$ edges. Therefore
\begin{equation}\label{I:connections}
b(A)\geq (n-1). 
\end{equation}

It is important to notice that the dual graph does not capture the topology of a polyomino. In Figure \ref{dual_graph} for example, $A$ has no holes but its dual graph contains a cycle, and the dual graph of $B$ is acyclic but $B$ has five holes.

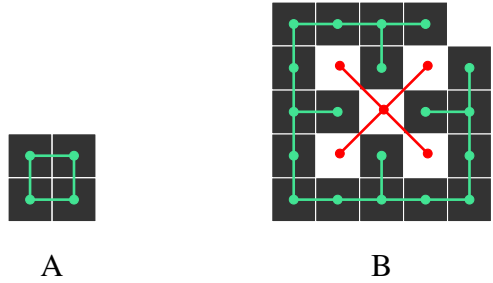
\begin{figure}[H] 
\begin{center} 
\begin{tikzpicture} 

\foreach \i/\j in {0 / 0, 0 / 1, 1 / 0,1/1} { 
\fill[black!80] (\i,\j) --++(0:1) --++(90:1) --++(180:1) -- cycle;
\draw [fill=newcol, newcol] (\i+.5, \j+.5) circle [radius=.1];
}
\foreach \i/\j in {0 / 0, 0 / 1, 1 / 0,1/1} { 
\draw[white] (\i,\j) --++(0:1) --++(90:1) --++(180:1) --++(270:1);
}
\foreach \x/\y/\z/\w in {0/0/0/1, 0/0/1/0, 1/0/1/1, 0/1/1/1  }{ 
\draw [line width=1pt, color=newcol] (.5+\x,.5+\y) -- (.5+\z, .5+\w);
}

\foreach \i/\j in {0 / 0, 0 / 1, 1 / 0, 2 / 0, 2 / 1, 0 / 2, 1 / 2, 4 / 0, 4 / 1, 3 / 0, 4 / 2, 3 / 2, 0 / 4, 0 / 3, 1 / 4, 2 / 4, 2 / 3, 4 / 3, 3 / 4 } { 
\fill[black!80] (\i+6,\j) --++(0:1) --++(90:1) --++(180:1) -- cycle;
\draw [fill=newcol, newcol] (\i+6.5, \j+.5) circle [radius=.1];
\draw[white] (\i+6,\j) --++(0:1) --++(90:1) --++(180:1) --++(270:1);
}

\foreach \x/\y in { 1/1, 3/1, 2/2, 3/3, 1/3}{
\draw [fill=red, red] (3+4+\x-0.45, 1+\y-0.45) circle [radius=.1];
}
\foreach \x/\y/\z/\w in {0/0/0/4, 0/0/4/0 , 0/4/3/4, 4/0/4/3, 0/2/1/2, 2/0/2/1, 2/3/2/4, 4/2/3/2 }{ 
\draw [line width=1pt, color=newcol] (6.5+\x,\y+.5) -- (6+.5 +\z, \w+.5);
}
\foreach \x/\y/\z/\w in {1/1/3/3, 1/3/3/1}{ 
\draw [line width=1pt, color=red] (6.55+\x,.55+\y) -- (6.55+\z, .55+\w);
}
\node (a) at (1,-1.5) [above] {A};
\node (b) at (8.5,-1.5) [above] {B};
\end{tikzpicture} 
\end{center}
\caption{Polyominoes $A$ and $B$ with their dual graphs colored in green, and the hole graph of $B$ colored in red.} \label{dual_graph} 
\end{figure}

\begin{definition}\label{D:holegraph}
The \textit{hole graph} of a polyomino $A$ is the graph whose vertices are indexed by the holes of $A$, with edges connecting two holes if their boundaries share a common vertex of the polyomino. 
\end{definition}

We refer to the hole adjacency condition as being \emph{corner adjacent}, as opposed to the \emph{edge adjacent} condition for tiles. A set of holes in a polyomino is said to be \emph{connected} if the corresponding induced subgraph of the hole graph is connected. The hole graph of a polyomino $A$ is necessarily acyclic, as a cycle corresponds to a Jordan curve in the plane which would disconnect the interior of $A$.

\subsection{A Topological Isoperimetric Inequality}

For all $n \ge 1$, we denote by $p_{\min}(n)$, the minimum perimeter that a polyomino with $n$ tiles can have. In 1976, F. Harary and H. Harborth \cite{harary1976extremal} proved that the minimum perimeter possible in a polyomino with $n$ tiles is given by 
\begin{equation}\label{perimeter1976}
p_{min}(n)=2\left\lceil2\sqrt{n}\right\rceil.
\end{equation}

In \cite{kurz2008counting}, it was proven that free polyominoes with area $\alpha\in\{N^2, N(N+1)\}$ that attain minimum perimeter are unique. These polyominoes are precisely those with the shape of a square or a pronic rectangle.

It is clear that polyominoes that are not simply connected, that have at least one hole, cannot achieve minimum perimeter as given in (\ref{perimeter1976}). The next lemma gives a lower bound for the minimum perimeter that a polyomino can have given its area and its number of holes.

\begin{lemma}\label{nop}
If $A$ is a polyomino with $n$ tiles and $h$ holes, then $p_o(A)\geq p_{min}(n+h)$. 
\end{lemma}
\begin{proof}
Let $H$ be the aggregate area of all the holes of $A$. Then $H\geq h$ because the minimum area that a hole can have is an area of one. Let $B$ be the polyomino obtained by completely covering the holes of $A$ with $H$ tiles. Observe that $B$ has $n+H$ tiles, and by equation (\ref{perimeter1976}), $p(B) \geq p_{min}(n+ H)$. Using the fact that the function $h(x)=2\left\lceil2\sqrt{x}\right\rceil$ is a non-decreasing function we get that $p(B) \geq p_{min}(n+h)$. Then, because $p(B)=p_o(B)$ and $A$ and $B$ have the same outer perimeter, we conclude that 
\begin{equation} \label{holesnominp}
p_o(A)\geq p_{min}(n+h).
\end{equation}

\end{proof}



Lemma \ref{nop} and equation (\ref{perimeter1976}) are the reasons of why we have defined in Section \ref{Ss:mainthms} that a polyomino with $n$ tiles and $h$ holes has minimal outer perimeter if its outer perimeter is $2\left\lceil2\sqrt{n+h}\right\rceil$. In what follows, using techniques introduced in \cite{kahle2018polyominoes}, we give an upper bound for the number of holes that a polyomino can have. 

Let $A$ be a polyomino with $n$ tiles. Then $4n = 2b(A) + p(A)$ because each square tile has 4 edges and these edges are either on the perimeter of $A$ or connecting two tiles of $A$. Thus, using $p(A)= p_o(A)+ p_h(A)$ we get
%
\begin{equation}\label{I:p_h}
    p_h(A)= 4n-2b(A)-p_{o}(A).
\end{equation}
The minimum number of edges that a hole can have is four, thus from (\ref{I:p_h}) we get
\begin{equation*}\label{E:ubh(A)}
    h(A)\leq \frac{p_h(A)}{4} = \frac{4n-2b(A)-p_{o}(A)}{4}.
\end{equation*}
Then, from inequalities (\ref{I:connections}) and (\ref{holesnominp}) we get
\begin{equation}\label{E:ubh(A)2}
h(A)\leq \frac{4n-2(n-1)-p_{min}(n+h)}{4}.
\end{equation}

 This is a topological isoperimetric inequality that bounds the number of holes that a polyomino can have depending on the structure of its dual graph, the area of its holes, and its outer perimeter. The next definition establishes notation for this upper bound as a function of $n$ and $h$.

\begin{definition}\label{E:M(n,h)}
For all natural numbers $n$ and $h$, define
\begin{equation}\label{ublb}
M(n,h) = \frac{2n+2-p_{min}(n+h)}{4}.
\end{equation}
\end{definition}
From inequality (\ref{E:ubh(A)2}), this function then gives a necessary condition for the existence of a polyomino, which is intricately tied to efficient structure. 

\begin{lemma}\label{L:M(n,h)}
If $M(n,h)<h$ for some natural numbers $n$ and $h$, then a polyomino with $n$ tiles and $h$ holes does not exist. And for all $h \geq 1$, we have $h \leq M(g(h),h)$.  
\end{lemma}

\begin{lemma}\label{Meff}
A polyomino with $n$ tiles and $h$ holes is efficiently structured if and only if $h = M(n, h)$. 
\end{lemma}


\begin{proof}
 Let $A$ be a polyomino with $n$ tiles and $h$ holes. It is essentially by definition that if $A$ is efficiently structured then $M(n,h)=h$, as the three characterestics of efficient structure are exactly where minimal values are plugged into equation (\ref{I:p_h}) to get inequality (\ref{E:ubh(A)2}).
 
 Now suppose $A$ fails to achieve one of the three condtitions of efficient structure. If $A$ is not acyclic, then $b(A) > n-1$. If it does not have minimal outer perimeter, then $p_o(A) > p_{min}(n+h)$. And finally if there is a hole of area at least two, then there is at least one hole with at least six edges, and thus $h \leq (p_h(A)-2)/4$. By equation (\ref{I:p_h}) and inequality (\ref{E:ubh(A)2}), all three issues cause $h < M(n,h)$, by decreasing the numerator in the first two cases and increasing the denominator in the third. So if $M(n,h)=h$, then $A$ is efficiently structured.
 
\end{proof}

In \cite{kahle2018polyominoes}, it was proven that the crystallized polyominoes defined in the K-R sequence satisfy $M(n,h)=h$, and are therefore efficiently structured.

\begin{corollary}\label{C:onebreak}
 If a polyomino $A$ exists with $n$ tiles and $h$ holes such that $M(n,h)=h+1/2$, then exactly one of the following three things occurs: 1) $A$ has a single dual cycle; 2) $A$ has a single hole with area two, and the rest have area one; or 3) $p_o(A)=p_{min}(n+h)+2$.
\end{corollary}
 
This is clear from the construction of $M(n,h)$ and the proof of Lemma \ref{Meff}. To understand how how equation (\ref{ublb}) changes with respect to the minimal outer perimeter of a polyomino, we observe that
\begin{equation}\label{E:pminatsqpr}
    p_{min}(n+h+1)=
    \begin{cases}
    p_{min}(n+h)+2 & \text{ if } n+h \text{ is square or pronic}\\
    p_{min}(n+h) & \text{ else. } 
    \end{cases}
\end{equation}

Adding a single tile also increases $2n$ to $2n+2$. Then for a fixed $h \geq 1$, $M(n,h)$ is a non-decreasing function of $n$ such that
\begin{equation}\label{E:Mnchange}
    M(n+1,h)-M(n,h)=
    \begin{cases}
    0 & \text{ if } n+h \text{ is a square or pronic number}\\
    1/2 & \text{ else } 
    \end{cases}
\end{equation}

We now give a proof of Theorem \ref{T:efficient}, which states that any efficiently structured polyomino is crystallized.

\begin{proof}[\textbf{Proof of Theorem \ref{T:efficient}}]
 Let $A$ be an efficiently structured polyomino with $n$ tiles and $h$ holes. By Lemma \ref{Meff}, this is equivalent to $h=M(n,h)$, and by equation (\ref{E:Mnchange}) we know that $M(n-1,h)\leq M(n,h)$. If $M(n-1,h)< M(n,h)$, then by Lemma \ref{L:M(n,h)} there does not exist a polyomino with $n-1$ tiles and $h$ holes, and thus $A$ is crystallized.
 
 Suppose instead that $M(n-1,h) = M(n,h)$, and assume that there exists a polyomino $B$ with $n-1$ tiles and $h$ holes. Then by equation (\ref{E:Mnchange}) we get that $n-1+h$ is a square or pronic number, and $B$ is efficiently structured by Lemma \ref{Meff}. By Remark \ref{area}, a polyomino with minimal outer perimeter and total area equal to a square and pronic number must be constructed in a square or pronic rectangle, and thus the boundary layer of $B$ is completely filled, which gives a cycle. This contradicts efficient structure for $B$, and therefore there does not exist a polyomino with $n-1$ tiles and $h$ holes. Thus, $A$ is crystallized.  
\end{proof}
 
 The converse of Theorem \ref{T:efficient} is not true. Figure \ref{F:crystallizednotM} shows crystallized polyominoes which fail to attain minimal outer perimeter and thus are not efficiently structured. 
\vspace{-.2cm}
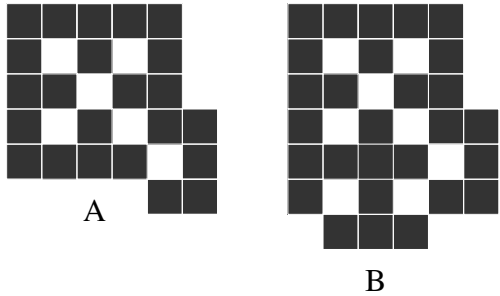
\begin{figure}[H] 
\begin{center} 
\begin{tikzpicture}[scale=.8]

\foreach \i \j in { 0 / 2, 0 / 3, 0 / 4, 0 / 5, 0/6, 1/6, 2/6, 3/6, 4/6, 4/4, 4/3,  5/2, 5/3, 1/2, 1/4,  2/2, 2/3 , 2/5, 2/6, 3/2, 3/4, 4/5 , 5/1, 4/1 } { 
\fill[black!80] (\i,\j) --++(0:1) --++(90:1) --++(180:1) -- cycle;
} 

\foreach  \i \j in {1/0, 2/0,3/0, 0/1, 0 / 2, 0 / 3, 0 / 4, 0 / 5, 0/6, 1/6, 2/6, 3/6, 4/6, 4/4, 4/3, 4/1, 5/1, 5/2, 5/3, 1/2, 1/4, 2/1, 2/2, 2/3 , 2/5, 2/6, 3/2, 3/4, 4/5, 5/1, 4/1  } {
\draw[white] (\i,\j) --++(0:1) --++(90:1) --++(180:1) --++(270:1);
		}



\foreach \i \j in {1/0, 2/0,3/0, 0/1, 0 / 2, 0 / 3, 0 / 4, 0 / 5, 0/6, 1/6, 2/6, 3/6, 4/6, 4/4, 4/3, 4/1, 5/1, 5/2, 5/3, 1/2, 1/4, 2/1, 2/2, 2/3 , 2/5, 2/6, 3/2, 3/4, 4/5  } { 
\fill[black!80] (\i+8,\j) --++(0:1) --++(90:1) --++(180:1) -- cycle;
} 

\foreach  \i \j in {1/0, 2/0,3/0, 0/1, 0 / 2, 0 / 3, 0 / 4, 0 / 5, 0/6, 1/6, 2/6, 3/6, 4/6, 4/4, 4/3, 4/1, 5/1, 5/2, 5/3, 1/2, 1/4, 2/1, 2/3 , 2/5, 2/6, 3/2, 3/4, 4/5, 4/2 } {
\draw[white] (\i+8,\j) --++(0:1) --++(90:1) --++(180:1) --++(270:1);
		}

\node (a) at (2.5,.5) [above] {A};
\node (b) at (10.5,-1.5) [above] {B};

\end{tikzpicture} 
\end{center} 
\vspace{-.5cm}
\caption{Polyominoes $A$ and $B$ are crystallized polyominoes that are not efficiently structured. Although both are acyclic with each hole having an area of one, neither has minimal outer perimeter. $A$ has 23 tiles and 6 holes and $B$ has 28 tiles and 8 holes. One can see in the calculation that $6 < M(6,23)=6.5$, and $8 < M(8,28)=8.5$.} 
\label{F:crystallizednotM} 
\end{figure}




In the next section we develop tools to understand for which values of $h$ a crystallized polyomino with $h$ holes may fail to be efficiently structured.

\section{Obstructions to Crystallization}\label{sect:obstruct}

Define $m(h)=\min\{n : M(n,h)\ge h\}$. By equation (\ref{E:ubh(A)2}) this is a theoretical lower bound for $g(h)$, and therefore if a polyomino exists with $m(h)$ tiles and $h$ holes, this is an immediate proof that $g(h)=m(h)$.

In this section we establish several obstructions to $g(h)$ hitting the optimal value of $m(h)$. In particular, we aim to show that Theorem \ref{ghsolved} is in fact equivalent to stating that for all $h\ge 1$,
\begin{equation}\label{E:ginm}
    g(h)=
    \begin{cases}
        m(h)+1  & \text{ if } h=h_{\alpha}+1 \text{ for square or pronic } \alpha\neq(2^l+1)^2\\
        m(h) & \text{ else.}\\
\end{cases}
\end{equation}

This asserts that $m(h)$ is the right answer except when trying to fit one too many holes into an optimal shape of minimal outer perimeter, and Theorem \ref{T:efficientall} will follow by characterizing when $M(m(h)+1,h)=h$. To establish these exceptions at values of $h_{\alpha}+1$, we begin by examining the numerical and geometric constraints of $m(h)$.

\begin{lemma}\label{L:propertiesm}
For every $h$, $M(m(h),h)=h$. Furthermore, if $0 < \alpha-(m(h)+h) < 3$ for some square or pronic number $\alpha$, then $m(h+1)=m(h)+3$, and otherwise $m(h+1)=m(h)+2$.
\end{lemma}

\begin{proof}
 The perimeter of a polyomino is always even, so $2n+2-p_{min}(n+h)$ is even, and $M(n,h)=C/2$ for some integer $C$. Fix $h \ge 1$. For $n=1$, $p_{min}(1+h) > 6$ and $$2n+2-p_{min}(n+h)=4-p_{min}(1+h) < 0.$$ Then by equation (\ref{E:Mnchange}), increasing $n$ by one increases $M(n,h)$ by either 0 or $1/2$, and thus $M(n,h)$ hits every positive multiple of $1/2$ as $n$ increases indefinitely. So $M(m(h),h)=h$ and the inequality in its definition is really equality.
 
 Suppose we have $n=m(h)$ for some $h$. Then since $p_{min}$ is non-decreasing,
 \begin{equation}
    \begin{split}
        M(n+1,h+1)&=\frac{2n+2+2-p_{min}(n+h+2)}{4}\\[8pt]
        &\leq\frac{2n+2-p_{min}(n+h)}{4}+\frac{1}{2}\\[8pt]
        &=h+\frac{1}{2}.
    \end{split}
 \end{equation}
 So $m(h+1)\ge m(h)+2$ for all $h$. By the same calculation, $$M(m(h)+2,h+1)=h+1$$ whenever $p_{min}(n+h+3)=p_{min}(n+h)$. By equation (\ref{E:pminatsqpr}) this will be the case unless $0 < \alpha-(m(h)+h) < 3$ for some square or pronic number $\alpha$, in which case $p_{min}(n+h+3)=p_{min}(n+h)+2$, and then 
 \begin{equation}
     \begin{split}
         M(m(h)+2,h+1)&=\frac{2n+4+2-p_{min}(n+h+3)}{4}\\[8pt]
         &=\frac{(2n+2-p_{min}(n+h))+2}{4}\\[8pt]
         &=h+1/2.
     \end{split}
 \end{equation} 
 Adding a third tile does not cross another square or pronic threshold, so by equation (\ref{E:Mnchange}) this adds another 1/2 to the function and $m(h+1)=m(h)+3$.
\end{proof}

Similar to the polyiamond case in \cite{malen2019polyiamonds}, this points to each additional hole requiring two extra tiles in general, and then a third tile when the total area of the shape expands past a threshold of $N^2$ or $N(N+1)$. To study the structure of crystallization near these thresholds, for a fixed $\alpha\in\{N^2,N(N+1)\}$ define 
\begin{equation}
    t_{\alpha}=\max\{h:m(h)+h\le \alpha\}.
\end{equation} 
Since $g(h)\ge m(h)$, this is a theoretical upper bound for $h_\alpha$. We derive equations for this sequence, and characterize these cases by how close $m(t_{\alpha})+t_{\alpha}$ gets to $\alpha$.

Suppose that $m(h)+h=N^2$, and note that such an $h$ is necessarily $h=t_{N^2}$. Then $M(N^2-h,h)=h$ by definition of $m(h)$, so
\begin{equation}\label{E:1sttalpha}
\begin{split}
4h & = 2(N^2-h)+2-4N\\
& = 2N^2-2h+2-4N\\
\Longrightarrow \ \ \ h &= \frac{(N-1)^2}{3}.
\end{split}
\end{equation}
This is an integer if and only if $N\equiv 1 \text{ mod 3}$. It is straightforward to check by similar calculations that the values
\begin{equation}\label{Toffset}
    m(t_{\alpha})+t_{\alpha}=
    \begin{cases}
        N^2-1 & \text{ if } \alpha=N^2 \text{ and } N\not\equiv 1 \text{ mod } 3\\
        N(N+1) & \text{ if } \alpha=N(N+1) \text{ and } N\not\equiv 2 \text{ mod } 3\\
        N(N+1)-2 & \text{ if } \alpha=N(N+1) \text{ and } N\equiv 2 \text{ mod } 3,\\[10pt]
    \end{cases}
\end{equation}
along with equation (\ref{E:1sttalpha}) yield that

\begin{equation}\label{tsq}
    t_{N^2}= 
    \begin{cases}
        \frac{(N-1)^2}{3} & \text{ if $N\equiv 1$ mod 3}\\[8pt]
        \frac{N(N-2)}{3} & \text{else},\\[8pt]
    \end{cases}
\end{equation}
and

\begin{equation}\label{tpr}
    t_{N(N+1)}= 
    \begin{cases}
        \frac{N(N-1)}{3} & \text{ if $N\equiv 0$ or 1 mod 3}\\[8pt]
        \frac{(N+1)(N-2)}{3} & \text{ if $N\equiv 2$ mod 3}.\\[8pt]
    \end{cases}
\end{equation}

Lemma \ref{L:propertiesm} describing the jumps in $m(h)$ ensures that the values in equation (\ref{Toffset}) are all maximal for $m(h)+h\le\alpha$, and so the values of $h$ solved for in these calculations give the values of $t_{\alpha}$ in (\ref{tsq}) and (\ref{tpr}). 

Observe that Theorem \ref{T:holesminper} is now equivalent to stating that $h_{\alpha}=t_{\alpha}$ if $\alpha=(2^l+1)^2$, and otherwise $h_{\alpha}=t_{\alpha}-1$. Or in words, only the K--R sequence from \cite{kahle2018polyominoes} is able attain the maximal number of holes theoretically possible in a square or pronic rectangle. 

And by the jumps of $m(h)$, Theorem \ref{T:gpronicperfectsquare} asserts that $g(h_{\alpha})=m(h_{\alpha})$, which we prove by construction in Section \ref{constructions}. Then the fact that equation (\ref{E:ginm}) is equivalent to Theorem \ref{ghsolved} follows immediately from Theorems \ref{T:holesminper} and \ref{T:gpronicperfectsquare} and the jumps in $m(h)$. Assuming these three main theorems, for which we provide proofs in Sections \ref{constructions} and \ref{S:destroy}, we now prove the structural characterization given in Theorem \ref{T:efficientall}.

\begin{proof}[\textbf{Proof of Theorem \ref{T:efficientall}}]
By equation (\ref{E:ginm}), a crystallized polyomino with $h$ holes will either have $m(h)$ tiles and be efficiently structured by Lemma \ref{Meff}, or it will have $g(h)=m(h)+1$ tiles. By equation (\ref{E:Mnchange}), $M(m(h)+1,h) = h$ when $m(h)+h$ is a square or pronic number, and otherwise $M(m(h)+1,h) > h$. And by equations (\ref{E:1sttalpha}) and (\ref{Toffset}), $m(h)+h$ is a square or pronic number precisely when $h=t_{\alpha}$ for $\alpha\notin S\cup R\cup \{N^2 : N=2^l+1\}$. 

Therefore, by equation (\ref{E:ginm}), the set $S\cup R$ is exactly the set of $\alpha$ for which $$g(h_{\alpha}+1)=m(h_{\alpha}+1)+1,$$
and $m(h_{\alpha}+1)+1$ tiles and $h_{\alpha}+1$ holes is not efficiently structured. 

It only remains to show that the efficient condition which fails for $\alpha\in S\cup R$ is minimal outer perimeter. Then let $h=h_{\alpha}+1$ for $\alpha\in S\cup R$. By Lemma \ref{L:propertiesm} and equation (\ref{E:ginm}), $g(h_\alpha +1)-g(h_{\alpha})=3$. But observe that the constant $C$ in Theorem \ref{T:gpronicperfectsquare} is either four or five for $\alpha\in S\cup R$. Then $$g(h_{\alpha}+1)+(h_{\alpha}+1)=(g(h_{\alpha})+3)+h_{\alpha}+1=g(h_{\alpha})+h_{\alpha}+4,$$ and
$$\alpha -1 \le g(h_{\alpha})+h_{\alpha}+4\le \alpha.$$ 
Therefore $$p_{min}(g(h_{\alpha}+1)+(h_{\alpha}+1))=p_{min}(\alpha).$$ 

By Remark \ref{area}, all bounding rectangles with outer perimeter $p_{min}(\alpha)$ which are not a square or pronic rectangle have area at most $\alpha-1$. A construction with $g(h_{\alpha}+1)$ tiles and $h_{\alpha}+1$ holes cannot exist in a shape of area at most $\alpha -1$, because it either has area $\alpha$ and is too big, or it has area $\alpha-1$ and by the reasoning in the proof of Theorem \ref{T:efficient} it would necessarily fill the entire boundary layer and create a cycle. By definition of $h_{\alpha}$, an extra hole cannot fit in the square or pronic rectangle of area $\alpha$, and therefore such a construction does not achieve minimal outer perimeter.

By equation (\ref{E:Mnchange}), we know that $$M(m(h_{\alpha})+1,h_{\alpha}+1)=M(m(h_{\alpha}+1),h_{\alpha}+1)+1/2=(h_{\alpha}+1)+1/2.$$
Then by Corollary \ref{C:onebreak}, only one efficient condition can fail. So a crystallized polyomino with $h_{\alpha}+1$ holes for $\alpha\in S\cup R$ is acyclic, each of its holes has an area of one, and it fails to achieve minimal outer perimeter, and all other crystallized polyominoes are efficiently structured.

\end{proof}


We proceed by showing that $g(t_{\alpha}) > m(t_{\alpha})$ whenever $\alpha \neq (2^l+1)^2$, and determine that this will also require increasing the outer perimeter to get a crystallized polyomino with $t_{\alpha}$ holes.

\subsection{Checkerboard Obstructions}
Let the $\emph{checkerboard partition}$ of a rectangle in the square lattice refer to the bipartition of its squares into sets $W$ and $B$, where the squares alternate between $W$ and $B$ in every row and column like a checkerboard. For convenience we always assume that the top left corner is in $W$. When a polyomino is forced to use up a sufficient amount of its boundary layer, then holes in the interior will necessarily be contained in one set of this partition. Recall that $D_1$ refers to the corner-less rectangular boundary layer, and $D_2$ has all spaces of it's boundary filled except for a single corner.

\begin{lemma}[Checkerboard Lemma]\label{checkered}
Let $A$ be an acyclic polyomino with $n$ tiles and $h$ holes, each having an area of one. If $A$ has a rectangular interior, then the outermost layer of the interior alternates between holes and tiles, and the set of holes in $A$ is completely contained in the set $W$ in the checkerboard partition of the interior.
\end{lemma}
\begin{proof}
 In general any empty spaces in the interior of a polyomino are contained in the bounded components of its complement in the plane, and are thus part of the holes. So two adjacent empty spaces anywhere in the interior contribute to a hole of area at least two. Then if $A$ has rectangular interior, its boundary contains $D_1$, so in the outermost layer of the interior any two adjacent filled spaces form a dual cycle with their two adjacent boundary tiles. These are both contradictions to our assumptions, and hence the spaces in the outermost layer of the interior must alternate.
 
 Furthermore, at least one of the the corners of this layer must be a hole, since filling in all four would connect the boundary sections of $D_1$ into a cycle (see Figure \ref{fig:dmincheck}). Then by rotations and reflections, we may assume the top left corner is empty and therefore in the outermost layer of the interior all spaces in $W$ are holes and all spaces in $B$ are tiles. 
 
 \begin{figure}[H]
    \centering
    \begin{subfigure}[t]{.4\textwidth}
        \centering
        \begin{tikzpicture}
            \gardenborder{7}{8}

            \fill[newcol!80] (\ra,\ra) --++(0:\ra) --++(90:\ra) --++(180:\ra) -- cycle;
            \fill[newcol!80] (\ra,5*\ra) --++(0:\ra) --++(90:\ra) --++(180:\ra) -- cycle;
            \fill[newcol!80] (6*\ra,\ra) --++(0:\ra) --++(90:\ra) --++(180:\ra) -- cycle;
            \fill[newcol!80] (6*\ra,5*\ra) --++(0:\ra) --++(90:\ra) --++(180:\ra) -- cycle;
            \gardenborderLines{7}{8}
            \draw[white] (\ra,\ra) --++(0:\ra) --++(90:\ra) --++(180:\ra) --++(270:\ra);
            \draw[white] (\ra,5*\ra) --++(0:\ra) --++(90:\ra) --++(180:\ra) --++(270:\ra);
            \draw[white] (6*\ra,\ra) --++(0:\ra) --++(90:\ra) --++(180:\ra) --++(270:\ra);
            \draw[white] (6*\ra,5*\ra) --++(0:\ra) --++(90:\ra) --++(180:\ra) --++(270:\ra);


        \end{tikzpicture}
    \end{subfigure}
\hspace{.3cm}
    \begin{subfigure}[t]{.4\textwidth}
    \centering
        \begin{tikzpicture}
            \gardenborder{7}{8}
            \foreach \i in {0,1}{
                \foreach \j in {1,2,3}{
                    \fill[newcol!80] (2*\j*\ra,\ra+4*\i*\ra) --++(0:\ra) --++(90:\ra) --++(180:\ra) -- cycle;
                }
                \fill[newcol!80] (\ra,2*\i*\ra+2*\ra) --++(0:\ra) --++(90:\ra) --++(180:\ra) -- cycle;
            }
            \fill[newcol!80] (6*\ra,3*\ra) --++(0:\ra) --++(90:\ra) --++(180:\ra) -- cycle;

            \gardenborderLines{7}{8}
            \foreach \i in {0,1}{
                \foreach \j in {1,2,3}{
                    \draw[white] (2*\j*\ra,\ra+4*\i*\ra) --++(0:\ra) --++(90:\ra) --++(180:\ra) --++ (0,-\ra);
                }
                \draw[white] (\ra,2*\i*\ra+2*\ra) --++(0:\ra) --++(90:\ra) --++(180:\ra) --++ (0,-\ra);
            }
            \draw[white] (6*\ra,3*\ra) --++(0:\ra) --++(90:\ra) --++(180:\ra) --++ (0,-\ra);

        \end{tikzpicture}
    \end{subfigure}
    \caption{Left: A dual cycle is created if all four interior corners are filled. Right: An alternating interior layer, with holes in $W$ and green tiles in $B$.}
    \label{fig:dmincheck}
\end{figure}
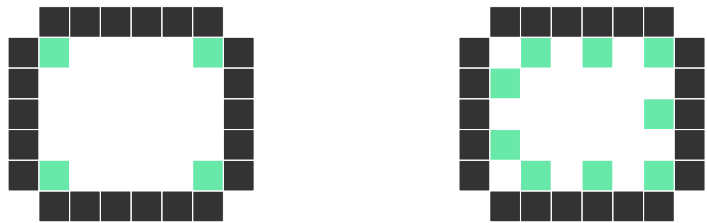
 
 Now suppose $A$ has holes contained somewhere in $B$, and let $B'\subseteq B$ be a maximal connected set of holes in $A$. All spaces in $W$ which are adjacent to $B'$ must be filled, since every hole must have an area of one. And since $B'$ is maximal with respect to being connected, all spaces in $B$ which are corner adjacent to $B'$ must either be filled, or not be contained in the interior of $A$. But spaces in $B'$ are separated from the boundary by the alternating layer where all spaces of $B$ are filled, so all such spaces in $B$ which are corner adjacent to $B'$ are in the interior of $A$ and must be filled.
 
 The hole graph of a polyomino is acyclic, as noted following Definition \ref{D:holegraph}. Therefore the portion of the hole graph corresponding to $B'$ is a tree. We claim that for any acyclic set of holes, each having an area of one, which has all of its corner adjacent spaces filled with tiles, the tiles surrounding the holes form a cycle. This is clearly the case for a single hole, and any tree can be constructed by fixing a root and adding vertices of degree one, one at a time. 
 
 So assume that the statement holds for every such acyclic set of at most $k$ holes, each with an area of one. A new hole cannot be attached to more than one existing vertex, as this would create a cycle. Thus any new hole requires filling in the seven surrounding tiles which are not the corner adjacent hole it is being attached to, some of which are already filled in $A$. This adds an indent to the cycle where the new hole was attached, but it remains a cycle, and by induction if $A$ has an acyclic set of holes, each with an area of one, with all of its corner adjacent tiles filled, then $A$ has a dual cycle.
 
\begin{figure}[H]
    \centering
    \begin{subfigure}[t]{.25\textwidth}
    \centering
        \begin{tikzpicture}
            \foreach \i in {1,2}{
                \fill[black!80] (\i*\ra-\ra,0) --++(0:\ra) --++(90:\ra) --++(180:\ra) -- cycle;
                \fill[black!80] (0,\i*\ra) --++(0:\ra) --++(90:\ra) --++(180:\ra) -- cycle;
                \fill[black!80] (\i*\ra,2*\ra) --++(0:\ra) --++(90:\ra) --++(180:\ra) -- cycle;
                \fill[black!80] (2*\ra,2*\ra-\i*\ra) --++(0:\ra) --++(90:\ra) --++(180:\ra) -- cycle;
            }

            \foreach \i in {1,2}{
                \draw[white] (\i*\ra-\ra,0) --++(0:\ra) --++(90:\ra) --++(180:\ra) --++(270:\ra);
                \draw[white] (0,\i*\ra) --++(0:\ra) --++(90:\ra) --++(180:\ra) --++(270:\ra);
                \draw[white] (\i*\ra,2*\ra) --++(0:\ra) --++(90:\ra) --++(180:\ra) --++(270:\ra);
                \draw[white] (2*\ra,2*\ra-\i*\ra) --++(0:\ra) --++(90:\ra) --++(180:\ra) --++(270:\ra);
            }
\draw[very thick, newcol, rounded corners=1] (.75*\ra,0.5*\ra) --++ (-.25*\ra,0) --++ (0,2*\ra) --++ (2*\ra,0) --++ (0,-2*\ra) --++(-2*\ra,0);
            
        \end{tikzpicture}
    \end{subfigure}
    \begin{subfigure}[t]{.3\textwidth}
        \centering
        \begin{tikzpicture}
            \foreach \i in {2,3,4}{
                \foreach \j in {0,1,3,4}{
                    \fill[black!80] (\i*\ra+\j*\ra,6*\ra-\i*\ra) --++(0:\ra) --++(90:\ra) --++(180:\ra) -- cycle;
                }
            }
            \foreach \j in {1,3,4}{
                    \fill[black!80] (\ra+\j*\ra,5*\ra) --++(0:\ra) --++(90:\ra) --++(180:\ra) -- cycle;
            }
            \foreach \j in {0,1,3}{
                    \fill[black!80] (5*\ra+\j*\ra,\ra) --++(0:\ra) --++(90:\ra) --++(180:\ra) -- cycle;
            }

            \foreach \i in {2,3,4}{
                    \fill[black!80] (\i*\ra,6*\ra) --++(0:\ra) --++(90:\ra) --++(180:\ra) -- cycle;
                    \fill[black!80] (\i*\ra+4*\ra,0) --++(0:\ra) --++(90:\ra) --++(180:\ra) -- cycle;
            }

           \foreach \i in {2,3,4}{
                \foreach \j in {0,1,3,4}{
                    \draw[white] (\i*\ra+\j*\ra,6*\ra-\i*\ra) --++(0:\ra) --++(90:\ra) --++(180:\ra) --++(270:\ra);
                }
            }
            \foreach \j in {1,3,4}{
                    \draw[white] (\ra+\j*\ra,5*\ra) --++(0:\ra) --++(90:\ra) --++(180:\ra) --++(270:\ra);
            }
            \foreach \j in {0,1,3}{
                    \draw[white] (5*\ra+\j*\ra,\ra) --++(0:\ra) --++(90:\ra) --++(180:\ra) --++(270:\ra);
            }

            \foreach \i in {2,3,4}{
                    \draw[white] (\i*\ra,6*\ra) --++(0:\ra) --++(90:\ra) --++(180:\ra) --++(270:\ra);
                    \draw[white] (\i*\ra+4*\ra,0) --++(0:\ra) --++(90:\ra) --++(180:\ra) --++(270:\ra);
            }

\draw[very thick, newcol, rounded corners=1] (2.5*\ra,6.2*\ra) --++ (0,0.3*\ra) --++ (2*\ra,0) 
--++ (0,-\ra) --++(\ra,0) 
--++ (0,-\ra) --++(\ra,0) 
--++ (0,-\ra) --++(\ra,0) 
--++ (0,-\ra) --++(\ra,0)
--++ (0,-2*\ra) --++(-2*\ra,0)
--++ (0,\ra) --++(-\ra,0)
--++ (0,\ra) --++(-\ra,0)
--++ (0,\ra) --++(-\ra,0)
--++ (0,\ra) --++(-\ra,0)
--++ (0,2*\ra);

        \end{tikzpicture}
    \end{subfigure}
    \begin{subfigure}[t]{.3\textwidth}
    \centering
        \begin{tikzpicture}
            \foreach \i in {2,4}{
                \foreach \j in {0,1,3,4}{
                    \fill[black!80] (\i*\ra+\j*\ra,6*\ra-\i*\ra) --++(0:\ra) --++(90:\ra) --++(180:\ra) -- cycle;
                }
            }
            \foreach \j in {0,1,3}{
                    \fill[black!80] (3*\ra+\j*\ra,3*\ra) --++(0:\ra) --++(90:\ra) --++(180:\ra) -- cycle;
            }
            \foreach \j in {1,3,4}{
                    \fill[black!80] (\ra+\j*\ra,5*\ra) --++(0:\ra) --++(90:\ra) --++(180:\ra) -- cycle;
            }
            \foreach \j in {0,1,3}{
                    \fill[black!80] (5*\ra+\j*\ra,\ra) --++(0:\ra) --++(90:\ra) --++(180:\ra) -- cycle;
            }

            \foreach \i in {2,3,4}{
                    \fill[black!80] (\i*\ra,6*\ra) --++(0:\ra) --++(90:\ra) --++(180:\ra) -- cycle;
                    \fill[black!80] (\i*\ra+4*\ra,0) --++(0:\ra) --++(90:\ra) --++(180:\ra) -- cycle;
            }
            \fill[black!80] (7*\ra,4*\ra) --++(0:\ra) --++(90:\ra) --++(180:\ra) -- cycle;
            \fill[black!80] (8*\ra,4*\ra) --++(0:\ra) --++(90:\ra) --++(180:\ra) -- cycle;
            \fill[black!80] (8*\ra,3*\ra) --++(0:\ra) --++(90:\ra) --++(180:\ra) -- cycle;

            \foreach \i in {2,4}{
                \foreach \j in {0,1,3,4}{
                    \draw[white] (\i*\ra+\j*\ra,6*\ra-\i*\ra) --++(0:\ra) --++(90:\ra) --++(180:\ra) --++(270:\ra);
                }
            }
            \foreach \j in {0,1,3}{
                    \draw[white] (3*\ra+\j*\ra,3*\ra) --++(0:\ra) --++(90:\ra) --++(180:\ra) --++(270:\ra);
            }
            \foreach \j in {1,3,4}{
                    \draw[white] (\ra+\j*\ra,5*\ra) --++(0:\ra) --++(90:\ra) --++(180:\ra) --++(270:\ra);
            }
            \foreach \j in {0,1,3}{
                    \draw[white] (5*\ra+\j*\ra,\ra) --++(0:\ra) --++(90:\ra) --++(180:\ra) --++(270:\ra);
            }

            \foreach \i in {2,3,4}{
                    \draw[white] (\i*\ra,6*\ra) --++(0:\ra) --++(90:\ra) --++(180:\ra) --++(270:\ra);
                    \draw[white] (\i*\ra+4*\ra,0) --++(0:\ra) --++(90:\ra) --++(180:\ra) --++(270:\ra);
            }

            \draw[white] (7*\ra,4*\ra) --++(0:\ra) --++(90:\ra) --++(180:\ra) --++(270:\ra);
            \draw[white] (8*\ra,4*\ra) --++(0:\ra) --++(90:\ra) --++(180:\ra) --++(270:\ra);
            \draw[white] (8*\ra,3*\ra) --++(0:\ra) --++(90:\ra) --++(180:\ra) --++(270:\ra);

\draw[very thick, newcol, rounded corners=1] (2.5*\ra,6.2*\ra) --++ (0,0.3*\ra) --++ (2*\ra,0) 
--++ (0,-\ra) --++(\ra,0) 
--++ (0,-\ra) --++(3*\ra,0) 
--++ (0,-4*\ra) --++(-2*\ra,0)
--++ (0,\ra) --++(-\ra,0)
--++ (0,\ra) --++(-\ra,0)
--++ (0,\ra) --++(-\ra,0)
--++ (0,\ra) --++(-\ra,0)
--++ (0,2*\ra);
        \end{tikzpicture}
    \end{subfigure}
    \caption{Dual cycle around one hole, around a tree, and then extended around a single extra hole.}
    \label{fig:inductcyc}
\end{figure}
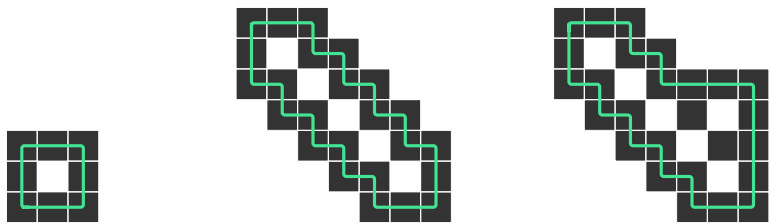
 
 This is in contradiction with our assumption of $A$ being acyclic. So all holes of $A$ must be contained in $W$ and all spaces of $B$ must be filled.
\end{proof}

Recall that it was pointed out in the proof of Theorem \ref{T:efficient} that the total area of an efficiently structured polyomino cannot be equal to a square or pronic rectangle, since minimal outer perimeter and area more than $\alpha -1$ requires it to be constructed in the square or pronic rectangle of area $\alpha$, and this would require the entire boundary to be filled, creating a dual cycle. 

Furthermore, suppose the total area of an efficiently structured polyomino is $\alpha-1$. Then it must again be constructed in the square or pronic rectangle of area $\alpha$, since a smaller rectangle with minimal perimeter has total area at most $\alpha-1$ by Remark \ref{area}, and by the same argument a dual cycle would be created. And then it has exactly one empty boundary space, which must be a corner to preserve minimal perimeter, and therefore it has boundary $D_2$.

Combining this observation regarding the total area with the \hyperref[checkered]{\hyperref[checkered]{Checkerboard Lemma}} is already sufficiently restrictive to rule out $m(t_{\alpha})$ in most cases.

\begin{lemma}
\label{nosqs}
 If $m(t_{\alpha})+t_{\alpha}=\alpha$, then $g(t_{\alpha}) > m(t_{\alpha})$ and $h_{\alpha}\le t_{\alpha}-1$.
\end{lemma}

\begin{proof}
That $g(t_\alpha) > m(t_\alpha)$ follows directly from the preceding observation regarding total area. The only difference is that here we acknowledge that the only time $m(h)+h$ is a square or pronic number is if $h=t_{\alpha}$ for $\alpha=N^2$ with $N\equiv 1\text{ mod } 3$, or $\alpha=N(N+1)$ with $N\not\equiv 2\text{ mod } 3$.

Then if $m(t_{\alpha})+t_{\alpha}=\alpha$, any crystallized polyomino with $t_{\alpha}$ holes has total area at least $\alpha+1$. This cannot fit in a square or pronic rectangle of area $\alpha$, and therefore $h_{\alpha}\le t_{\alpha}-1$.


\end{proof}

\begin{lemma}\label{nodiv2}
If $2 \mid \alpha$, then $g(t_{\alpha}) > m(t_{\alpha})$ and $h_{\alpha}\le t_{\alpha}-1$.
\end{lemma}

\begin{proof}
 Lemma \ref{nosqs} already covers the case here when $m(t_{\alpha})+t_{\alpha}=\alpha$. Otherwise, by the derivations of equations (\ref{tsq}) and (\ref{tpr}), if $\alpha=N^2$ then $m(t_{\alpha})+t_{\alpha}=\alpha-1$, and if $\alpha=N(N+1)$ then $m(t_{\alpha})+t_{\alpha}=\alpha-2$. \\
 
 \textit{Case 1:} Assume that $2 \mid \alpha=N^2$, and suppose that $A$ is a crystallized polyomino with $t_{\alpha}$ holes and $m(t_{\alpha})$ tiles such that $m(t_{\alpha})+t_{\alpha}=\alpha-1$. Then $A$'s bounding rectangle must be an $N\times N$ square by the observation regarding total area, which then also asserts that $A$ must have boundary $D_2$.
 
 By the \hyperref[checkered]{Checkerboard Lemma}, the spaces of $B$ in the checkerboard partition of the interior are all filled. So on each of the two parallel sides of even length, one of the corners of the interior must be filled. But three of the corners of the boundary layer are also filled, so by the pigeonhole principle this creates at least one dual cycle of length four in a corner, which contradicts efficient structure. Hence $g(t_{\alpha}) > m(t_{\alpha})$. 
 
 Furthermore, if $h_{\alpha}=t_{\alpha}$, then there is a polyomino $A$ fitting into the square of area $\alpha$ with $t_{\alpha}$ holes and $m(t_{\alpha})+1$ tiles. Thus $A$ must have total area $\alpha$, and all tiles of the boundary layer must be filled. But none of the corner tiles bound any holes, and taking one away will not disconnect $A$, which would imply that in fact $g(t_{\alpha})\le m(t_{\alpha})$. By the preceding argument this is a contradiction, and hence $h_{\alpha}\le t_{\alpha}-1$.\\
 
 \textit{Case 2:} If instead $\alpha=N(N+1)$ and $m(t_{\alpha})+t_{\alpha}=\alpha-2$, then similarly by Remark \ref{area} any other rectangle with this perimeter has area at most $\alpha-2$, and filling a shape of this area would create a dual cycle in the boundary layer, so a pronic rectangle is required. Two tiles on the boundary must be empty, and to maintain minimum perimeter they must be two corners or a corner and one of its adjacent spaces.  As before, two interior corners must be filled, and then both of their incident corners in the boundary must be empty to avoid dual cycles. But there are precisely two empty spaces which are not holes or tiles, so the rest of the boundary is filled and the indented corners create a dual cycle, which is again a contradiction. Hence $g(t_{\alpha}) > m(t_{\alpha})$.
 
 Similarly to the previous case, if $h_{\alpha}=t_{\alpha}$, then a polyomino $A$ in the pronic rectangle of area $\alpha$ which has $t_{\alpha}$ holes and $m(t_{\alpha})+1$ tiles has total area $\alpha-1$, and exactly one empty boundary space. Then if the outermost layer of the interior alternates between tiles and holes, there must again be two interior corners which are filled. At least one of those boundaries must have the corner and its two adjacent boundary spaces filled, creating a dual cycle of length four. And the corner tile there does not bound any holes and can be removed without disconnecting $A$, which implies that $g(t_{\alpha})\le m(t_{\alpha})$, which is a contradiction. Therefore it suffices to show that the outermost layer of the interior alternates.
 
 However, the \hyperref[checkered]{Checkerboard Lemma} no longer applies, since $M(n,h)$ will be $h+1/2$, which forces exactly one of the three conditions of efficient structure to fail by Corollary \ref{C:onebreak}. Since we are restricting the shape to be contained in a pronic rectangle, the only way to increase the perimeter while filling a total area of $\alpha-1$ is if the empty space of the boundary is not in a corner. But then $A$ would be acyclic and only have holes of area one, so the \hyperref[checkered]{Checkerboard Lemma} would assure that the outermost layer of the interior alternates. Otherwise the empty boundary space is in a corner and there is either a single dual cycle or the existence of a hole with area exactly two.
 
 Suppose that there are two adjacent tiles in the outermost layer of the interior, creating a dual cycle of length four with the adjacent tiles of the boundary. Since there cannot be any further cycles or holes of area two, the rest of this layer must alternate, and it has $2(N-2)+2(N-3)$ spaces, which is even. Then surrounding these two adjacent tiles must be two holes, one on either side. And then two tiles, one on the other side of each hole, and then two holes, and so on. Since there are an even number of spaces this will end with either a hole of area two or two tiles placed next to each other creating another cycle, which is a contradiction. 
 
 The same is true if we start with a hole of area two, and therefore the spaces in this layer alternate between holes and tiles, which implies by the above that $g(t_{\alpha})\le m(t_{\alpha})$, which is a contradiction. Hence $h_{\alpha}\le t_{\alpha}-1$.

\end{proof}

Lemmas \ref{nosqs} and \ref{nodiv2} prove that $g(t_{\alpha})\ge m(t_{\alpha})+1$ and $h_{\alpha}\le t_{\alpha}-1$ for all odd squares with side length $N\equiv 1 \text{ mod } 3$, all even squares, and all pronic rectangles. This leaves only odd squares which have side lengths $N\equiv 0$ or 2 mod 3.

\section{Expansion and Compression}\label{S:Expansion}

To determine what happens for the remaining odd squares, we develop a technique for lifting efficient arrangements from smaller crystallized polyominoes, and refer to this process as \emph{expansion}. We refer to a set of five tiles arranged to have a central tile with one tile adjacent at each edge as a ``plus,'' and a set of five holes arranged to have a central hole with one hole adjacent at each corner as an ``X'' (see Figure \ref{fig:xandt}). 

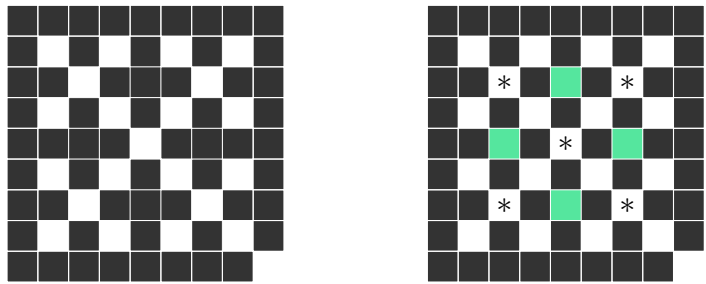
\begin{figure}[H]
    \centering
    \begin{subfigure}[t]{.4\textwidth}
        \centering
        \begin{tikzpicture}
            \fullborder{9}{9}
            \wtopcheck{9}{9}
            \foreach \i in {0,1}{
                \fill[black!80] (4*\ra,4*\i*\ra+2*\ra) --++(0:\ra) --++(90:\ra) --++(180:\ra) -- cycle;
                \fill[black!80] (4*\i*\ra+2*\ra,4*\ra) --++(0:\ra) --++(90:\ra) --++(180:\ra) -- cycle;
            }
            \fullborderLines{9}{9}
            \wtopcheckLines{9}{9}
        \end{tikzpicture}
    \end{subfigure}
    \begin{subfigure}[t]{.4\textwidth}
        \centering
        \begin{tikzpicture}
            \fullborder{9}{9}
            \wtopcheck{9}{9}
            \foreach \i in {0,1}{
                \fill[newcol!90] (4*\ra,4*\i*\ra+2*\ra) --++(0:\ra) --++(90:\ra) --++(180:\ra) -- cycle;
                \fill[newcol!90] (4*\i*\ra+2*\ra,4*\ra) --++(0:\ra) --++(90:\ra) --++(180:\ra) -- cycle;
                \node at (2.5*\ra,2.5*\ra+4*\i*\ra) {$\ast$};
                \node at (6.5*\ra,2.5*\ra+4*\i*\ra) {$\ast$};
            }
            \node at (4.5*\ra,4.5*\ra) {$\ast$};
            \fullborderLines{9}{9}
            \wtopcheckLines{9}{9}
        \end{tikzpicture}
    \end{subfigure}
    
        \caption{A polyomino with four pluses and five X's. On the right the central tile of each plus is marked in green, and the central hole of each X with an $\ast$.}
        \label{fig:xandt}
\end{figure}

For an odd integer $N$, let $P_N$ be the sub-arrangement of tiles and holes in the interior of the $N\times N$ rectangle in which: (a) all spaces in $B$ of the checkerboard partition are filled, and (b) all spaces in the odd rows of $W$, enumerated from the top down, are empty. The remaining spaces are considered to be undetermined.

\begin{figure}[H]
\centering

	\begin{subfigure}[t]{.4\textwidth}
		\centering
		\begin{tikzpicture}

            \fill[black!25] (0,0) -- (0,9*\ra) -- (9*\ra,9*\ra) -- (9*\ra,0) -- cycle 
                (\ra,\ra) -- (8*\ra,\ra) -- (8*\ra,8*\ra) -- (\ra,8*\ra) -- cycle;
			\wtopcheck{9}{9}
			\hardholesTwo{9}{9}
			\wtopcheckLines{9}{9}
            
		\end{tikzpicture}
	\end{subfigure}
    \caption{$P_{9}$ in the interior of the $9\times 9$ square, with tiles in black, holes in white, and undetermined spaces in green.}
    \label{pnex}
\end{figure}
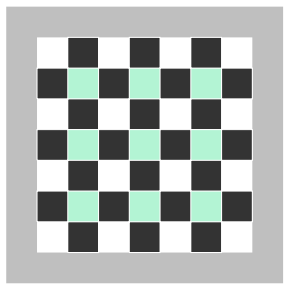

$P_{N}$ is not itself a polyomino since its set of tiles is not connected, but it will play an integral role in the expansion process. Observe that the undetermined spaces of $P_N$ form a square grid of dimensions $(N-3)/2\times(N-3)/2$. If any undetermined space is filled with a tile it will create a plus, while any hole in these spaces will create an X. We let $U_N$ denote the set of spaces left undetermined by $P_N$. The process of expansion will be to determine these spaces by using the interior of an $(N+1)/2\times(N+1)/2$ square as a template.

\begin{definition}
Let $A$ be an arrangement of tiles and holes in the $N\times N$ square with boundary $D$ for some $D_1\subseteq D \subseteq D_2$.  The \emph{expansion} of $A$ is the arrangement $E(A)$ in the $(2N-1)\times (2N-1)$ square with boundary $D$, subarrangement $P_{2N-1}$, and in which the spaces of $U_{2N-1}$ are equivalent to the corresponding spaces in the interior of $A$. 
\end{definition}

\begin{definition}
Conversely, an arrangement $A$ of tiles and holes in the $N\times N$ square which can be written as $A=D\cup P_N\cup U_A$ for some $D_1\subseteq D \subseteq D_2$ and a specified subarrangement $U_A$ in the spaces of $U_N$ is called \emph{compressible}, and the \emph{compression} of $A$ is $C(A)$, the arrangement in the $(N+1)/2\times (N+1)/2$ square with boundary $D$ and interior equivalent to $U_A$.
\end{definition}

These processes are inverses, with $E(C(A))=A$ for any compressible $A$. They also preserve the properties of connectivity in both the hole graph and the dual graph.

\begin{lemma}\label{excomp:adjacencies}
Let $A$ be a compressible arrangement. Then a pair of adjacent tiles in $C(A)$ expands to a path of length two in $A$, and a pair of corner adjacent empty spaces in $C(A)$ expands to a path of length two in the hole graph of $A$. 
\end{lemma}

\begin{proof}
 This comes directly from the structure of $P_N\subset A$. Two adjacent interior tiles of $C(A)$ expand to two pluses which intersect in a connecting tile from $P_{N}$, and a tile adjacent to the boundary layer of $C(A)$ corresponds to a tile of $U_A$ which is connected to the boundary layer by a tile of $P_{N}$ (see Figure \ref{fig:comp1}). 
 
 Moreover, for any filled space of $U_A$ the four tiles of $P_{N}$ surrounding it can be thought to extend the four edges of the corresponding tile from $C(A)$. So when adjacent spaces of $C(A)$ are filled, the intermediate tile of $P_{N}$ is a subdivision of that edge in the expansion. The same holds for empty spaces, where two corner adjacent empty spaces of $C(A)$ create two X's in $A$ which intersect in an empty space of $P_{N}$ (see Figure \ref{fig:discpath}).
\end{proof}

While we do not require that an arrangement $A$ is a polyomino in the definitions of expansion and compression, Lemma \ref{excomp:adjacencies} allows us to determine when the property of being an acyclic polyomino is preserved by these processes.

\begin{lemma}[Compression Lemma]\label{compressionlem}
A compressible arrangement $A$ is an acyclic polyomino if and only if $C(A)$ is an acyclic polyomino with each of its holes having an area of one.
\end{lemma}

\begin{proof}
We prove this by exhibiting the contrapositives in both directions. First Let $A$ be a compressible arrangement. Suppose $C(A)$ has at least two adjacent empty spaces, forming a hole of area at least two. Then the tile separating the corresponding holes of $U_A$ is isolated and $A$ fails to be a polyomino (see Figure \ref{fig:comp1}).

\def\ra{.7}
\begin{figure}[b]
\centering
\hspace*{.7cm}
\begin{subfigure}[t]{.2\textwidth}
\centering
\begin{tikzpicture}
	\fullborder{4}{4}
	\foreach \i in {1,2}{
		\fill[black!80] (\i*\ra,\ra) --++(0:\ra) --++(90:\ra) --++(180:\ra) -- cycle;
	}
	\fullborderLines{4}{4}
	\foreach \i in {1,2}{
		\draw[white] (\i*\ra,\ra) --++(0:\ra) --++(90:\ra) --++(180:\ra) --++(270:\ra);
	}
	\draw[very thick, red] (1.5*\ra,2.5*\ra) --++(0:\ra);
	\draw [fill=red, red] (1.5*\ra,2.5*\ra) circle [radius=.1];
	\draw [fill=red, red] (2.5*\ra,2.5*\ra) circle [radius=.1];
	\draw[thick, newcol] (0.5*\ra,0.5*\ra) --++(0:2*\ra) --++(90:\ra) --++(180:2*\ra) --++(270:\ra);
	\draw[thick, newcol] (1.5*\ra,0.5*\ra) --++(90:\ra);
	\foreach \i in {0.5,1.5,2.5}{
		\draw [fill=newcol, newcol] (\i*\ra,0.5*\ra) circle [radius=.1];
		\draw [fill=newcol, newcol] (\i*\ra,1.5*\ra) circle [radius=.1];
	}
	
    name%
	\node at (2*\ra,-\ra) {$C(A)$};
\end{tikzpicture}
\end{subfigure}
\hspace{.5cm}\raisebox{1.75cm}{$\longleftrightarrow$}
\begin{subfigure}[t]{.4\textwidth}
\centering
\begin{tikzpicture}
	
	\fullborder{7}{7}
	\wtopcheck{7}{7}
	\foreach \i in {2,4}{
		\fill[black!80] (\i*\ra,2*\ra) --++(0:\ra) --++(90:\ra) --++(180:\ra) -- cycle;
	}
	\fullborderLines{7}{7}
	\wtopcheckLines{7}{7}

	\draw[very thick, red] (2.5*\ra,4.5*\ra) -- (3.5*\ra,5.5*\ra) -- (4.5*\ra,4.5*\ra) -- (3.5*\ra,3.5*\ra) --  (2.5*\ra,4.5*\ra);
	\draw [fill=red, red] (2.5*\ra,4.5*\ra) circle [radius=.1];
	\draw [fill=red, red] (4.5*\ra,4.5*\ra) circle [radius=.1];
	\draw[thick, newcol] (0.5*\ra,0.5*\ra) --++(0:4*\ra) --++(90:2*\ra) --++(180:4*\ra) --++(270:2*\ra);
	\draw[thick, newcol] (2.5*\ra,0.5*\ra) --++(90:2*\ra);
	\foreach \i in {0.5,2.5,4.5}{
		\draw [fill=newcol, newcol] (\i*\ra,0.5*\ra) circle [radius=.1];
		\draw [fill=newcol, newcol] (\i*\ra,2.5*\ra) circle [radius=.1];
	}
	
	\node at (3.5*\ra,-\ra) {$A$};

\end{tikzpicture}
\end{subfigure}
\caption{Cycles are preserved by expansion and compression (in green). Two adjacent empty spaces in $C(A)$ disconnect $A$ (in red). Note that $C(A)$ is a polyomino, but $A$ is not.}
\label{fig:comp1}
\end{figure}
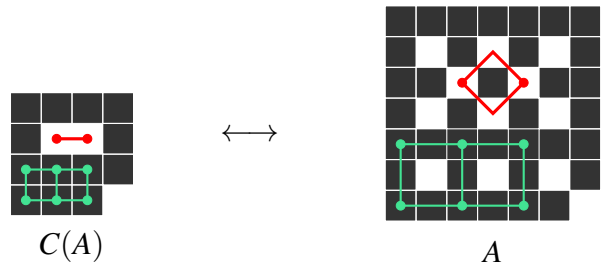

Next suppose that $C(A)$ has a dual cycle. If the cycle does not use any of the boundary layer, then $A$ contains a cycle by Lemma \ref{excomp:adjacencies}. Otherwise the cycle uses the boundary. But any connected section of the boundary used by the cycle in $C(A)$ is also connected in $A$ since by definition they have the same boundary structure, and this again forms a cycle (see Figure \ref{fig:comp1}). 

Now suppose that $C(A)$ is not connected, and thus not a polyomino. Then there must be a connected set of holes which either forms a cycle or is a path that connects to two empty boundary corner spaces at its ends, separating the tiles of $C(A)$ into at least two disjoint pieces (see Figure \ref{fig:discpath}). As in the preceding argument, by Lemma \ref{excomp:adjacencies} a cycle of holes expands to a cycle of holes in $A$, which would force $A$ to be disconnected. And a path of holes connecting two empty boundary corners expands to a path connecting two empty boundary corners of $A$, which is then disconnected. Thus, if $A$ has connected interior and is acyclic, then $C(A)$ is connected, acyclic, and all holes have area one.

\def\ra{.7}
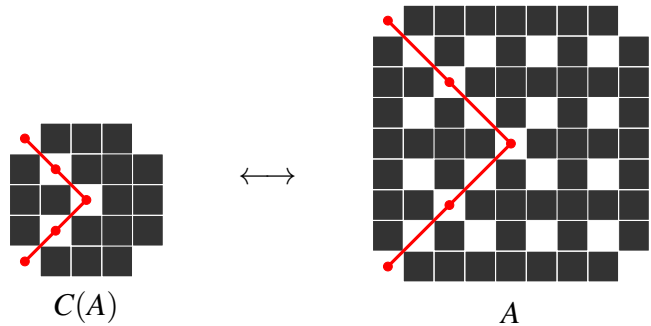
\begin{figure}[H]
\centering
\hspace*{1cm}
\begin{subfigure}[t]{.25\textwidth}
\centering
\begin{tikzpicture}
	\gardenborder{5}{5}
	\wtopcheck{5}{5}
	\foreach \i in {1,3}{
		\fill[black!80] (3*\ra,\i*\ra) --++(0:\ra) --++(90:\ra) --++(180:\ra) -- cycle;
	}
	\gardenborderLines{5}{5}
	\wtopcheckLines{5}{5}

	\draw[very thick, red] (0.5*\ra,0.5*\ra) -- (2.5*\ra,2.5*\ra) -- (0.5*\ra,4.5*\ra);
	\foreach \i in {0.5,1.5,2.5}{
		\draw [fill=red, red] (\i*\ra,\i*\ra) circle [radius=.1];
	}
	\foreach \i in {0.5,1.5}{
		\draw [fill=red, red] (\i*\ra,5*\ra-\i*\ra) circle [radius=.1];
	}
	
	\node at (2.5*\ra,-\ra) {$C(A)$};
\end{tikzpicture}
\end{subfigure}
\hspace{.1cm}
\raisebox{2cm}{$\longleftrightarrow$}
\begin{subfigure}[t]{.4\textwidth}
\centering
\begin{tikzpicture}
	
	\gardenborder{9}{9}
	\wtopcheck{9}{9}
	\fill[black!80] (2*\ra,4*\ra) --++(0:\ra) --++(90:\ra) --++(180:\ra) -- cycle;
	\foreach \i in {2,6}{
		\fill[black!80] (4*\ra,\i*\ra) --++(0:\ra) --++(90:\ra) --++(180:\ra) -- cycle;
	}
	\foreach \i in {2,4,6}{
		\fill[black!80] (6*\ra,\i*\ra) --++(0:\ra) --++(90:\ra) --++(180:\ra) -- cycle;
	}
	
	\gardenborderLines{9}{9}
	\wtopcheckLines{9}{9}

	\draw[very thick, red] (0.5*\ra,0.5*\ra) -- (4.5*\ra,4.5*\ra) -- (0.5*\ra,8.5*\ra);
	\foreach \i in {0.5,2.5,4.5}{
		\draw [fill=red, red] (\i*\ra,\i*\ra) circle [radius=.1];
	}
	\foreach \i in {0.5,2.5}{
		\draw [fill=red, red] (\i*\ra,9*\ra-\i*\ra) circle [radius=.1];
	}

	\node at (4.5*\ra,-\ra) {$A$};

\end{tikzpicture}
\end{subfigure}
\caption{A path of holes connecting two empty corners disconnects both $A$ and $C(A)$.}
\label{fig:discpath}
\end{figure}

On the other hand, if $A$ does not have connected interior, then it has either a cycle of holes or a path of holes connecting two empty boundary corners. Any path of holes must alternate between holes in $P_N$ and holes in $U_A$. Then suppose two holes in $U_A$ are mutually adjacent to a hole in $P_N$. These holes are either consecutive spaces of $U_A$ along a row or column, which would compress to a hole of area at least two in $C(A)$ (see Figure \ref{fig:comp1}), or they are diagonal from each other and thus corner adjacent in $C(A)$. If no holes of area at least two are created, then a cycle compresses to a cycle and a path connecting two empty boundary corners of $A$ compresses to a path connecting two boundary corners of $C(A)$ (see Figure \ref{fig:discpath}). Thus either a hole of area at least two is created, or $C(A)$ is disconnected.

Similarly, if there is a cycle of tiles in $A$, then the tiles of the cycle in the interior alternate between $P_N$ and $U_A$, and two tiles in this cycle from consecutive spaces of $U_A$ compress to adjacent tiles in $C(A)$. Since $A$ and $C(A)$ have the same boundary, this cycle compresses to a cycle in $C(A)$, as in the previous argument for expanding a cycle (see Figure \ref{fig:comp1}). Hence if $C(A)$ is conneceted, acyclic, and only has holes of area one, then $A$ has connected interior and is acyclic.
\end{proof}

Observe that if $A$ is compressible, then it cannot have holes of area more than one, and when $D=D_2$ both $A$ and $C(A)$ will necessarily achieve the minimum perimeter for their total area.

\begin{corollary}\label{comp:eff}
 A compressible acyclic polyomino $A$ with boundary $D_2$ is efficiently structured if and only if $C(A)$ is efficiently structured.
\end{corollary}
 
Moreover, containing the subarrangement $P_{N}$ is a necessary condition for crystallization of a polyomino contained in an odd $N\times N$ square with boundary layer $D_1\subseteq D \subseteq D_2$.

\begin{lemma}\label{oddholes} 
Let $N\ge5$ be an odd positive integer. If $A$ is an efficiently structured polyomino with square $(N-2)\times(N-2)$ interior, then all spaces in the odd rows of $W$ in the checkerboard partition of the interior must be holes.
\end{lemma}

\begin{proof}
 For $N$ odd, since we assume the top left corner of the interior is in $W$, all corners of the interior are in $W$ by parity. Then by the \hyperref[checkered]{Checkerboard Lemma} and by parity, the odd rows of $W$ are precisely those which have holes in the outermost layer of the interior.
 
 \begin{figure}[b]
    \centering
    \begin{subfigure}[t]{.4\textwidth}
        \centering
        \begin{tikzpicture}
            \fullborder{11}{11}
            \wtopcheck{11}{11}
            \foreach \i in {1,2,3}{
            	\fill[newcol!60] (3*\ra,2*\i*\ra+\ra) --++(0:\ra) --++(90:\ra) --++(180:\ra) -- cycle;
		\fill[newcol!60] (2*\i*\ra+\ra,7*\ra) --++(0:\ra) --++(90:\ra) --++(180:\ra) -- cycle;
            }
            \node at (3.5*\ra,7.5*\ra) {$t$};
            \fullborderLines{11}{11}
            \wtopcheckLines{11}{11}
        \end{tikzpicture}
    \end{subfigure}
    \begin{subfigure}[t]{.4\textwidth}
        \centering
        \begin{tikzpicture}
            \fullborder{11}{11}
            \wtopcheck{11}{11}
            \foreach \i in {1,2,3}{
            	\fill[newcol!60] (7*\ra,2*\i*\ra+\ra) --++(0:\ra) --++(90:\ra) --++(180:\ra) -- cycle;
		\fill[newcol!60] (2*\i*\ra+\ra,5*\ra) --++(0:\ra) --++(90:\ra) --++(180:\ra) -- cycle;
            }
            \node at (7.5*\ra,5.5*\ra) {$t$};
            \fullborderLines{11}{11}
            \wtopcheckLines{11}{11}
        \end{tikzpicture}
    \end{subfigure}
        \caption{Tiles in odd rows of $W$ cannot connect to the boundary.}
        \label{fig:oddrows}
\end{figure}
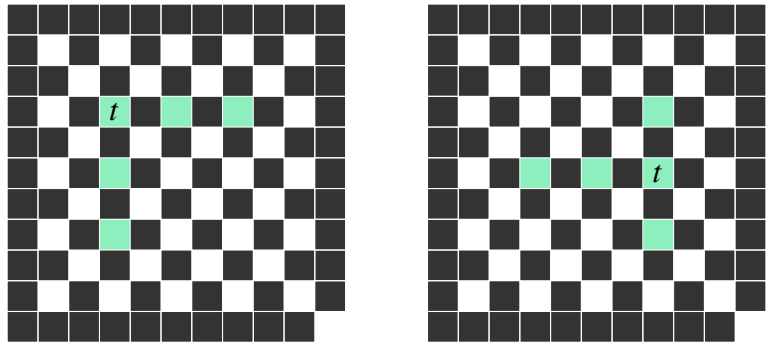
 
 If a space in an odd row of $W$ is filled by tile $t$, then the spaces which are corner adjacent to $t$ in even rows of $W$ must be empty to avoid creating dual cycles. So the only way to connect $t$ to the boundary in order to satisfy the connected interior condition for polyominoes is by filling a path of spaces in the odd rows and columns of $W$. But as noted above, these rows and columns coincide with the the holes in the outermost layer of $W$, and thus cannot connect to the boundary without creating dual cycles. Therefore they must all be holes.
\end{proof}



\begin{corollary}\label{odd:compressible}
 Let $N\ge 5$ be an odd integer. If $A$ is an efficiently structured polyomino with square $(N-2)\times(N-2)$ interior, then $A$ is compressible.
\end{corollary}

This follows immediately from the \hyperref[checkered]{Checkerdboard Lemma} and Lemma \ref{oddholes}. In particular, for $N$ odd with $N\not\equiv 1 \text{ mod } 3$, we have that $m(t_{N^2})+t_{N^2}=N^2-1$ and an efficiently structured polyomino with $t_{N^2}$ holes is compressible and has boundary $D_2$. This leads to a final obstruction to $g(t_{\alpha}) = m(t_{\alpha})$, and a proof of Theorem \ref{T:uniqueness}.

\begin{lemma}\label{L:oddintobst}
 Let $N$ be an odd positive integer such that $N\not\equiv 1 \text{ mod } 3$, and $N\neq2^l+1$ for any positive integer $l$. Then $g(t_{N^2}) > m(t_{N^2})$, and $h_{N^2}\le t_{N^2}-1$.
\end{lemma}

\begin{proof}
  Let $N=2^lp+1$ for some $l\ge1$ and some $p > 1$ for which $2\nmid p$, and suppose that $A$ is a polyomino with $t_{N^2}$ holes and $m(t_{N^2})$ tiles. For $N\not\equiv 1 \text{ mod } 3$ this gives total area $N^2-1$. By Corollary \ref{odd:compressible}, $A$ has boundary $D_2$ and is compressible. And $A$ is efficiently structured by Lemma \ref{Meff}, so by Corollary \ref{comp:eff} $C(A)$ is an efficiently structured polyomino with boundary $D_2$, and is itself compressible if $(N+1)/2$ is odd.
  
  Then let $N_0=N$, and let $N_i=(N_{i-1}+1)/2$ be the length of the square for $C^i(A)$, the $i$-th compression of $A$. Since $N=2^lp+1$, we have that $N_i=2^{l-i}p+1$, which is odd for all $i < l$. Therefore $l$ compressions can be performed, with $C^l(A)$ an efficiently structured polyomino with boundary $D_2$ in an $N_l\times N_l$ square for even $N_l = p+1 \ge4$. With $D_2$ as its boundary layer, the total area of $C^l(A)$ must be $N_{l}^{2}-1$, but by Lemma \ref{nodiv2} there cannot be an efficiently structured polyomino with this area in an even square. Therefore $g(t_{N^2}) > m(t_{N^2})$.
  
  And if $h_{N^2}=t_{N^2}$, then by the preceding argument a crystallized polyomino with $h_{N^2}$ holes must have total area at least $N^2$. If it has total area exactly $N^2$, then the boundary layer is completely filled and forms a cycle. As in the proof of Lemma \ref{nodiv2}, one of the corner tiles is extraneous and can be removed, which is a contradiction. Otherwise the total area is more than $N^2$ and such a polyomino cannot fit in this square. Therefore $h_{N^2}\le t_{N^2}-1$.
\end{proof}

\vspace{1cm}
\begin{figure}[H]
    \centering
\vspace{1cm}
\hspace*{1cm}
    \raisebox{.45cm}{\begin{subfigure}[t]{.25\textwidth}
        \centering
        \begin{tikzpicture}
            \fullborder{3}{3}
            \fullborderLines{3}{3}
            \node at (1.5*\ra,1.5*\ra) {$\ast$};
        \end{tikzpicture}
    \end{subfigure}
}
    \raisebox{.9cm}{$\xrightarrow{E(S_1)}$}
    \begin{subfigure}[t]{.25\textwidth}
        \centering
        \begin{tikzpicture}
            \fullborder{5}{5}
            \wtopcheck{5}{5}
            \fullborderLines{5}{5}
            \wtopcheckLines{5}{5}
            \node at (2.5*\ra,2.5*\ra) {$\ast$};
        \end{tikzpicture}
    \end{subfigure}
    \raisebox{.9cm}{$\longrightarrow$}
    \begin{subfigure}[t]{.25\textwidth}
        \centering
        \begin{tikzpicture}
            \fullborder{5}{5}
            \wtopcheck{5}{5}
            \fullborderLines{5}{5}
            \wtopcheckLines{5}{5}
        \end{tikzpicture}
    \end{subfigure}

\vspace{.3cm}
\hspace*{10.1cm} $\Big\downarrow$

    \begin{subfigure}[t]{.3\textwidth}
        \centering
        \begin{tikzpicture}
            \fullborder{9}{9}
            \wtopcheck{9}{9}
            \foreach \i in {0,1}{
                \fill[black!80] (4*\ra,4*\i*\ra+2*\ra) --++(0:\ra) --++(90:\ra) --++(180:\ra) -- cycle;
                \fill[black!80] (4*\i*\ra+2*\ra,4*\ra) --++(0:\ra) --++(90:\ra) --++(180:\ra) -- cycle;
            }
            \fullborderLines{9}{9}
            \wtopcheckLines{9}{9}
        \end{tikzpicture}
    \end{subfigure}
    \raisebox{1.72cm}{$\longleftarrow$}
    \begin{subfigure}[t]{.3\textwidth}
        \centering
        \begin{tikzpicture}
            \fullborder{9}{9}
            \wtopcheck{9}{9}
            \foreach \i in {0,1}{
                \fill[newcol!90] (4*\ra,4*\i*\ra+2*\ra) --++(0:\ra) --++(90:\ra) --++(180:\ra) -- cycle;
                \fill[newcol!90] (4*\i*\ra+2*\ra,4*\ra) --++(0:\ra) --++(90:\ra) --++(180:\ra) -- cycle;
                \node at (2.5*\ra,2.5*\ra+4*\i*\ra) {$\ast$};
                \node at (6.5*\ra,2.5*\ra+4*\i*\ra) {$\ast$};
            }
            \node at (4.5*\ra,4.5*\ra) {$\ast$};
            \fullborderLines{9}{9}
            \wtopcheckLines{9}{9}
        \end{tikzpicture}
    \end{subfigure}
    \raisebox{1.72cm}{$\xleftarrow{E(S_2)}$}
    \raisebox{0.825cm}{\begin{subfigure}[t]{.2\textwidth}
        \centering
        \begin{tikzpicture}
            \fullborder{5}{5}
            \gtopcheck{5}{5}
            \fullborderLines{5}{5}
            \node at (2.5*\ra,2.5*\ra) {$\ast$};
            \node at (1.5*\ra,1.5*\ra) {$\ast$};
            \node at (1.5*\ra,3.5*\ra) {$\ast$};
            \node at (3.5*\ra,1.5*\ra) {$\ast$};
            \node at (3.5*\ra,3.5*\ra) {$\ast$};
            \wtopcheckLines{5}{5}
        \end{tikzpicture}
    \end{subfigure}
}
    \caption{Expansion of $S_1$, with $E(S_1)=S_2$ and $E^2(S_1)=E(S_2)=S_3$.}
    \label{fig:exps1}
\end{figure}
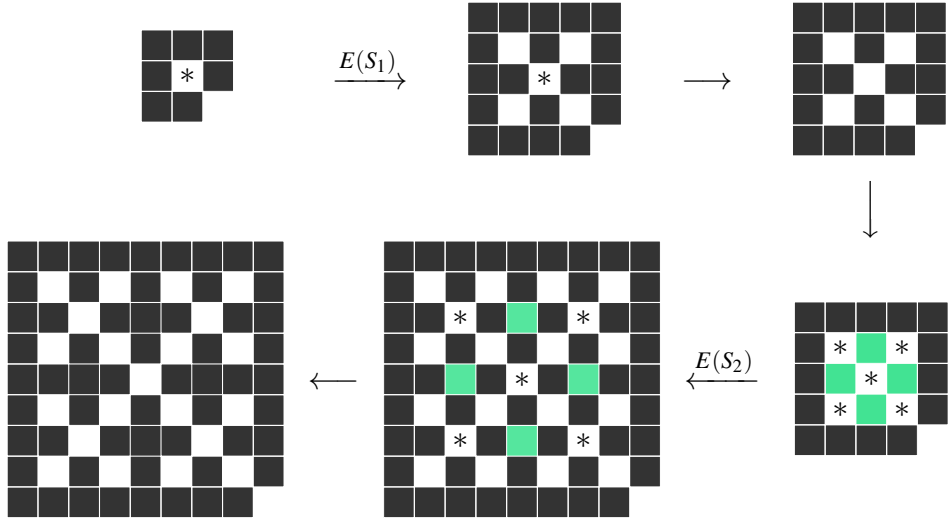

\begin{proof}[\textbf{Proof of Theorem \ref{T:uniqueness}}]
 For $N=2^l+1$, the existence of crystallized polyominoes with $t_{N^2}=(2^{2l}-1)/3$ holes and $m(t_{N^2})=[(2^{2l+1}+ 3 \cdot 2 ^ {l+1} + 4)/3] -1$ tiles was proved in \cite{kahle2018polyominoes}. We prove the uniqueness of those constructions by simply compressing them all down to $S_1$, the unique crystallized polyomino with a single hole (see Figure \ref{three_holes}).  
 
 In particular fix $l > 1$, let $N=2^l+1$ and $N_i=(N_{i-1}+1)/2=2^{l-i}+1$, and let $A$ be a compressible polyomino with $(2^{2l}-1)/3$ holes and $[(2^{2l+1}+ 3 \cdot 2 ^ {l+1} + 4)/3] -1$ tiles. Since $2^l+1\not\equiv 1 \text{ mod } 3$, this has total area $N^2-1$, and thus $A$ has boundary $D_2$. For $i < l$ all of the $N_i$ are odd, with $N_{l-1}=3$. Therefore $C^{i}(A)$ is compressible with boundary $D_2$ for all $i < l$, and thus $C^{l-1}(A)=S_1$. But compression and expansion are inverses, so if two polyominos $A$ and $A'$ have that $C^k(A)=C^k(A')$, then $A=A'$. In particular the polyominoes $S_l$ are the unique crystallized polyominoes with $(2^{2l}-1)/3$ holes, and $S_l=E^{l-1}(S_1)$. 
\end{proof}

The method of expansion and compression then gives an alternate way to construct the sequence $S_l$. For all other squares and pronic numbers, we give constructions in Section \ref{constructions} which show that $h_{\alpha}=t_{\alpha}-1$. Assuming these constructions, we use expansion to prove Theorems \ref{T:holesminper} and \ref{T:gpronicperfectsquare}. \\

\begin{proof}[\textbf{Proof of Theorems \ref{T:holesminper} and \ref{T:gpronicperfectsquare}}]
 By Lemmas \ref{nosqs} and \ref{nodiv2}, Given the constructions of Section \ref{constructions} and the sequence $S_l$, we have crystallized polyominoes with $h_\alpha$ holes and $\alpha-h_{\alpha}-C$ tiles for the appropriate $C$, except when $\alpha=N^2$ for odd $N$ such that $N\not\equiv 1 \text{ mod } 3$ and $N\neq 2^1+1$ for any positive integer $l$.
 
 Consider such an $N$, which can be written as $N=2^lp+1$ for $2,3\nmid p$. Then as in the proof of Lemma \ref{L:oddintobst}, $p+1$ is even and we can take a crystallized polyomino $A$ with $h_{\alpha}$ holes for $\alpha =(p+1)^2$ and expand it $l$ times to get a crystallized polyomino $C^l(A)$ in an $N\times N$ square. Since $3\nmid p$, $p+1\not\equiv 1 \text{ mod } 3$, and $A$ will be the appropriate construction from Section \ref{evsq2} or \ref{evsq0}, each with boundary $D_1$. Then $C^l(A)$ has boundary $D_1$ and total area $N^2-4$. By equation (\ref{tsq}), Lemma \ref{L:propertiesm}, and Lemma \ref{L:oddintobst}, $C^l(A)$ has $h_{N^2}=t_{N^2}-1$ holes and $N^2-h_{N^2}-4$ tiles.
\end{proof}

\section{Constructions of Crystallized Polyominoes}\label{constructions}
We call a sub-polyomino which is a sequence of overlapping pluses a \emph{plus tree}, and refer to the place where it connects to a section of the boundary layer as its root. The basic elements of our constructions involve choosing an appropriate boundary layer, and filling the interior with efficiently spaced plus trees. A plus tree growing in a certain direction is three spaces wide, and a disjoint plus tree is not allowed to connect to any of its tiles or fill any of the empty spaces between the pluses. 

These polyominoes all have $t_{\alpha}-1$ holes and will be efficiently structered by construction, and are thus an immediate proof of Theorems \ref{T:holesminper} and \ref{T:gpronicperfectsquare}. For $N\equiv l \text{ mod } 3$, we denote the $k$-th element of the even square sequences by $S_{l,k}$, and the $k$-th element of the pronic rectangle sequences by $R_{l,k}$, where $N$ is the shorter of the two side lengths. All equations for total area are taken from the appropriate cases in equations (\ref{Toffset}) and (\ref{tsq}), and Lemma \ref{L:propertiesm}.

\subsection{Even Squares, $N\equiv 1 \text{ mod } 3$}\label{evsq1}

\def\ra{.7}

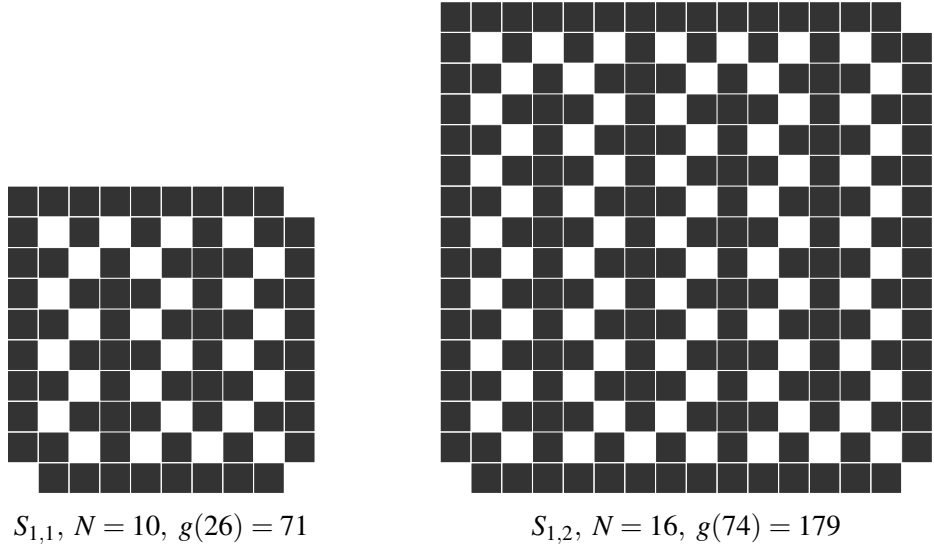
\begin{figure}[H]
    \centering
    \begin{subfigure}[t]{.4\textwidth}
        \centering
        \begin{tikzpicture}
            \Swl{2}{3}
        \end{tikzpicture}
        \caption{$S_{1,1}, \ N=10, \ g(26)=71$}
    \end{subfigure}
\hspace{.6cm}
    \begin{subfigure}[t]{.5\textwidth}
    \centering
        \begin{tikzpicture}
            \Swl{4}{6}
        \end{tikzpicture}
        \caption{$S_{1,2}, \ N=16, \ g(74)=179$}
    \end{subfigure}
\caption{The first two elements of the sequence $S_{1,k}$.}
\label{fig:evens1}
\end{figure}

The polyomino $S_{1,k}$ is made up of $2k$ vertical plus trees with roots alternating between the bottom and top boundaries from left to right, each with $3k$ pluses, and has boundary $D_1$ with the top left corner additionally filled. To satisfy the \hyperref[checkered]{Checkerboard Lemma} we also require that the spaces of the outermost layer of the interior alternate between tiles and holes, with a hole in the top left corner.

Each plus tree takes up three of its own columns, and along with the boundary and outermost layer of the interior on each side the total width is $N=4+6k$. Similarly, for a tree rooted at the bottom, count each hash of a plus along with the row above it, then there are two additional rows above and two below, and the height is also $4+6k$. This length is always 1 mod 3, so the minimal total area for $(t_{N^2}-1)$ holes is $N^2-3$. 

Since $N$ is even the bottom left and top right corners will be indented, and the boundary sections are all connected. Then $S_{1,k}$ achieves total area $(4+6k)^2-3$, and by construction it is efficiently structured. Hence $S_{1,k}$ is crystallized for all $k\ge 1$, with 
\begin{equation}
    \begin{split}
        h(S_{1,k})&=\frac{(N-1)^2}{3}-1=\frac{(3+6k)^2}{3}-1=12k^2+12k+2,\\
        |S_{l,k}|&=(4+6k)^2-h(S_{1,k})-3=24k^2+36k+11.
    \end{split}
\end{equation}


\subsection{Even Squares, $N\equiv 2 \text{ mod } 3$}\label{evsq2}

\def\ra{.7}

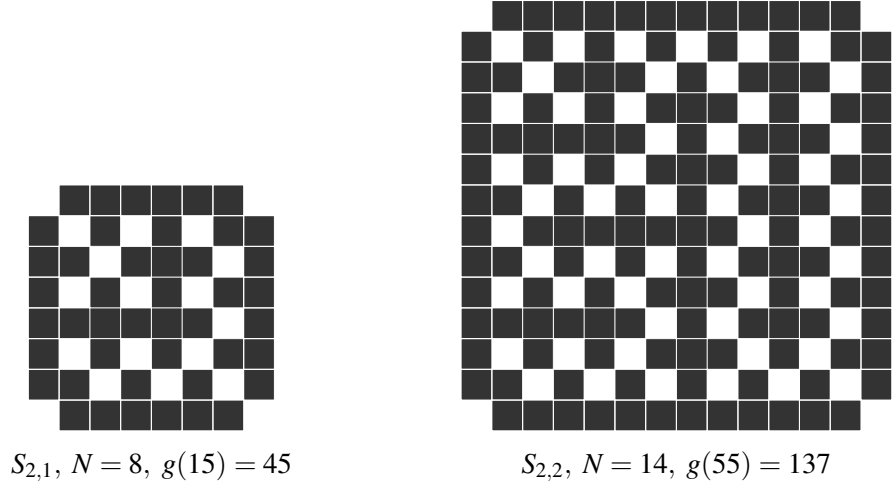
\begin{figure}[H]
    \centering
    \begin{subfigure}[t]{.4\textwidth}
        \centering
        \begin{tikzpicture}
            \EvenSqTwo{8}
        \end{tikzpicture}
        \caption{$S_{2,1}, \ N=8, \ g(15)=45$}
        \label{fig:esmod2}
    \end{subfigure}
\hspace{.6cm}
    \begin{subfigure}[t]{.5\textwidth}
    \centering
        \begin{tikzpicture}
            \EvenSqTwo{14}
        \end{tikzpicture}
        \caption{$S_{2,2}, \ N=14, \ g(55)=137$}
        \label{fig:esmod1}
    \end{subfigure}
\caption{The first two elements of the sequence $S_{2,k}$.}
\label{fig:S2}
\end{figure}

We construct $S_{2,k}$ with boundary $D_1$, and make the upper right hand corner equivalent to $S_2$ to connect the two disjoint boundary sections. To the right of this there are $2k-2$ vertical plus trees, with the roots alternating from the bottom to the top, and below the copy of $S_2$ there are $2k-2$ horizontal plus trees with two pluses each, with the roots alternating between the left boundary and the leftmost vertical plus tree from the bottom up.

The vertical plus trees, the 2 additional columns to the right boundary and one column separating them from $S_2$, which has width five, makes the total width $N=3(2k-2)+3+5=6k+2$, which is always 2 mod 3. There are the same number of horizontal plus trees as vertical ones, and the rest of spacing is the same, so the height is also $6k+2$. For $N\equiv 2 \text{ mod } 3$, the minimal total area for $t_{N^2}-1$ holes is $N^2-4$. So $S_{2,k}$ is efficiently structured by construction, and thus crystallized for all $k\ge 1$, with
\begin{equation}
\begin{split}
        h(S_{2,k})&=\frac{N(N-2)}{3}-1 = \frac{(6k+2)6k}{3}-1 =12k^2+4k-1, \\
        |S_{2,k}|&=(6k+2)^2-h(S_{2,k})-4 =24k^2+20k+1.
\end{split}
\end{equation}



\subsection{Even Squares, $N\equiv 0 \text{ mod } 3$}\label{evsq0}

\def\ra{.7}

\begin{figure}[H]
    \begin{subfigure}[t]{.4\textwidth}
        \centering
        \begin{tikzpicture}
            \EvenSqZero{12}
        \end{tikzpicture}
        \caption{$S_{0,1}, \ N=12, \ g(39)=101$}
    \end{subfigure}
\hspace{.6cm}
    \begin{subfigure}[t]{.5\textwidth}
    \centering
        \begin{tikzpicture}
            \EvenSqZero{18}
        \end{tikzpicture}
        \caption{$S_{0,2}, \ N=18, \ g(95)=225$}
    \end{subfigure}
    \caption{The first two elements of $S_{0,k}$.}
    \label{fig:S0}
\end{figure}
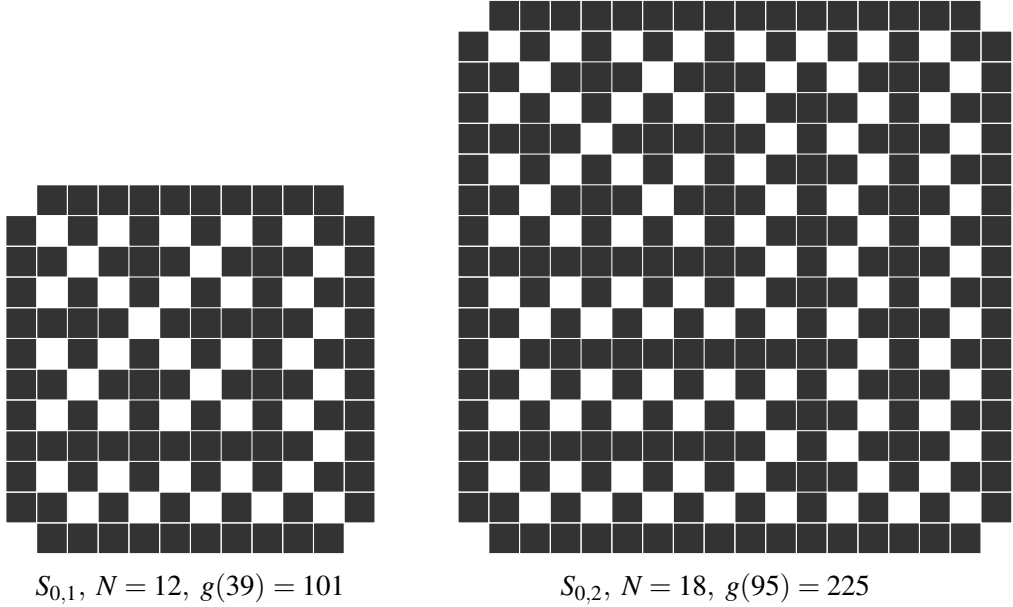

The construction for $N\equiv 0 \text{ mod } 3$ is the same as the previous case, but with $S_3$ instead of $S_2$ in the upper left corner. This increases the side length to $N=3(2k-2)+3+9=6k+6$, and the minimal total area for $t_{N^2}-1$ holes is $N^2-4$. Therefore by construction $S_{0,k}$ is an efficiently structured polyomino, and hence crystallized for all $k\ge1$, with

\begin{equation}
    \begin{split}
        h(S_{2,k})&=\frac{N(N-2)}{3}-1=\frac{(6k+6)(6k+4)}{3}-1=12k^2+20k+7,\\
        |S_{2,k}|&=(6k+6)^2-h(S_{2,k})-4=24k^2+52k+25.
    \end{split}
\end{equation}

\subsection{Pronic Rectangles, $N\equiv 0 \text{ mod } 3$}\label{prn0}

\def\ra{.7}

\begin{figure}[H]
    \centering
    \begin{subfigure}[t]{.2\textwidth}
        \centering
        \begin{tikzpicture}
            \coiledsnake{1}
        \end{tikzpicture}
    \caption{$R_{0,1}, \ N=6,$\\[2pt] $g(9)=30$}
    \end{subfigure}
\hspace{.3cm}
    \begin{subfigure}[t]{.3\textwidth}
    \centering
        \begin{tikzpicture}
            \coiledsnake{2}
        \end{tikzpicture}
    \caption{$R_{0,2}, \ N=9,$\\[2pt] $g(23)=64$}
    \end{subfigure}
\hspace{.3cm}
    \begin{subfigure}[t]{.4\textwidth}
    \centering
        \begin{tikzpicture}
            \coiledsnake{3}
        \end{tikzpicture}
    \caption{$R_{0,3}, \ N=12,$\\[2pt] $g(43)=110$}
    \end{subfigure}
\caption{The first three elements of $R_{0,k}$.}
\label{fig:R0}
\end{figure}
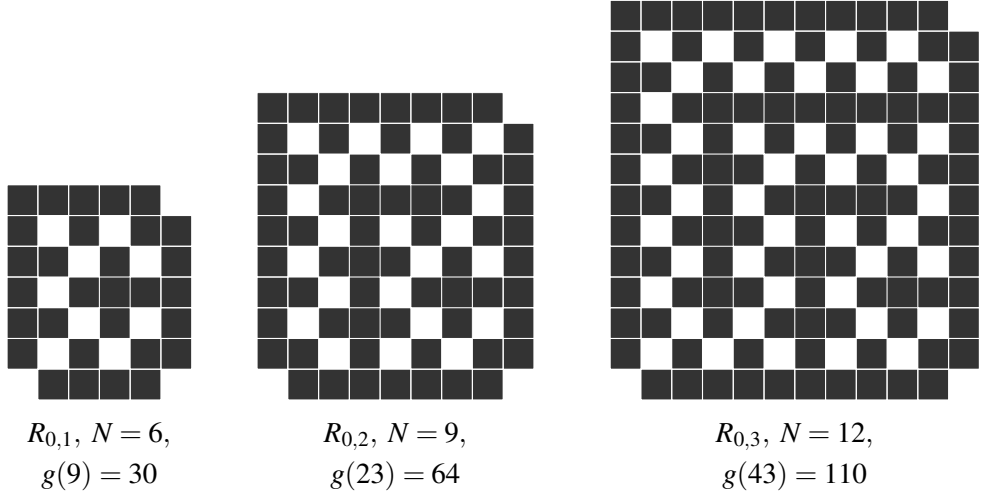

 Let $R_{0,k}$ have boundary $D_1$ with the top left corner filled. The interior is then filled with $k$ plus trees whose roots alternate between the right boundary and the bottom boundary. Right-rooted trees move left and then turn straight down, and bottom-rooted trees move up and then turn right, fitting tightly around the previous tree.

The height has $k$ plus trees and two additional rows on top and bottom, for a total of $3k+4$. The width is one less than the height, since the first tree is rooted on the side and so there are is only one column on the right which is not accounted for by the trees, instead of the usual two. Thus the bounding rectangle has dimensions $N\times(N+1)$ for $N=3k+3$, and by construction $R_{0,k}$ has efficient structure with total area $N(N+1)-3$. Therefore $R_{0,k}$ is crystallized for all $k\ge1$, with

\begin{equation}
    \begin{split}
        h(R_{0,k}) & = \frac{N(N-1)}{3}-1= \frac{(3k+3)(3k+2)}{3}-1 = 3k^2+5k+1,\\
        |R_{0,k}| & = N(N+1)-h(R_{0,k})-3= 6k^2+16k+8.
    \end{split}
\end{equation}


\subsection{Pronic Rectangles, $N\equiv 0 \text{ mod } 3$}\label{prn1}

\def\ra{.6}

\begin{figure}[H]
    \centering
    \begin{subfigure}[t]{.22\textwidth}
        \centering
        \begin{tikzpicture}
            \coiledsnakeTwo{1}
        \end{tikzpicture}
        \caption{$R_{1,1}, \ N=7,$\\[2pt] $g(13)=40$}
    \end{subfigure}
\hspace{.1cm}
    \begin{subfigure}[t]{.3\textwidth}
    \centering
        \begin{tikzpicture}
            \coiledsnakeTwo{2}
        \end{tikzpicture}
        \caption{$R_{1,2}, \ N=10,$\\[2pt] $g(29)=78$}
    \end{subfigure}
\hspace{.1cm}
    \begin{subfigure}[t]{.385\textwidth}
    \centering
        \begin{tikzpicture}
            \coiledsnakeTwo{3}
        \end{tikzpicture}
        \caption{$R_{2,3}, \ N=13,$\\[2pt] $g(51)=128$}
    \end{subfigure}
\caption{The first three elements of $R_{1,k}$.}
\label{fig:R1}
\end{figure}
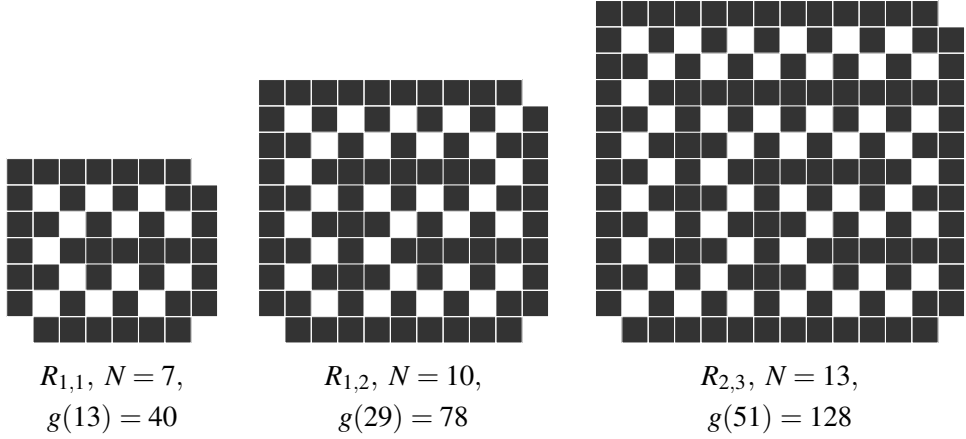

For $N\equiv 1 \text{ mod } 3$ we define $R_{1,k}$ to be the same as $R_{0,k}$, but with the initial tree having two pluses. In this case the width increases by two, and the height is the same, so its bounding rectangle is $N\times(N+1)$ for $N=3k+4$. The equation in terms of $N$ is the same as the above, and so $R_{1,k}$ is crystallized for all $k\ge1$, with

\begin{equation}
    \begin{split}
        h(R_{1,k}) & = \frac{N(N-1)}{3}-1= \frac{(3k+4)(3k+3)}{3}-1= 3k^2+7k+3,\\
        |R_{1,k}| & = N(N+1)-h(R_{1,k})-3= 6k^2+20k+14.
    \end{split}
\end{equation}

\subsection{Pronic Rectangles, $N\equiv 2 \text{ mod } 3$}\label{pr2}

\def\ra{.7}

\begin{figure}[H]
    \begin{subfigure}[t]{.4\textwidth}
        \centering
        \begin{tikzpicture}
            \twospirals{14}
        \end{tikzpicture}
        \caption{$R_{2,3}, \ N=14, \ g(59)=146$}
    \end{subfigure}
\hspace{1cm}
    \begin{subfigure}[t]{.4\textwidth}
    \centering
        \begin{tikzpicture}
            \twospirals{17}
        \end{tikzpicture}
        \caption{$R_{2,4}, \ N=17, \ g(89)=212$}
    \end{subfigure}
\caption{The third and fourth elements of the sequence $R_{2,k}$.}
\label{fig:R2}
\end{figure}
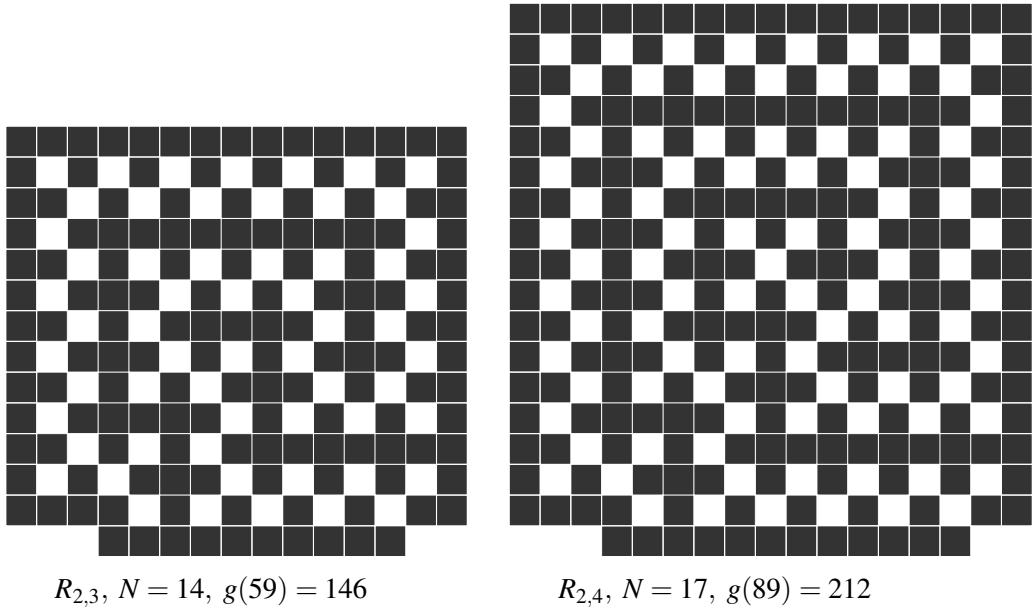

Finally we have the \emph{double spiral} construction for $N\equiv 2 \text{ mod } 3$. In this case, the minimal total area for $t_{N(N+1)}-1$ holes is $N(N+1)-5$. We orient $R_{2,k}$ so that the odd length is the width and the even length is the height, with all five empty spaces of the boundary in the bottom row; three from the left corner and two from the right corner. The interior is checkered and filled with two concentric plus tree spirals, one starting on the left of the bottom boundary and one on the bottom of the right boundary. They indent at the bottom left corner, and otherwise turn at right angles as depicted in Figure \ref{fig:R2}.

The $k$-th elelment of this sequence has $k$ horizontal sections of plus trees and $k$ vertical sections of plus trees, and the extra spaces depend on the parity of $k$. If $k$ is even, the horizontal trees can be counted with an extra two rows taken up by the bottom rooted plus, and thus with two additional rows at the top and only the boundary row on the bottom, the height is $3k+5$. And vertically it an extra plus moving left at the center giving two additional columns, with another two additional columns on each side. So the width is $3k+6$. If $k$ even, the right-rooted spiral can be counted among the $k$ trees, and there is an extra two spaces created by the plus moving down at the center. So the height is $3k+6$. And then horizontally there is a central hole indicating one extra column, with two additional columns on the sides, making the wideth $3k+5$.

So in all cases $N=3k+5\equiv 2 \text{ mod } 3$, and by construction $R_{2,k}$ is an efficiently structured crystal with total area $N(N+1)-5$, and with
\begin{equation}
    \begin{split}
        h(R_{2,k}) & = \frac{(N+1)(N-2)}{3}-1= \frac{(3k+6)(3k+3)}{3}-1 = 3k^2+9k+5,\\
        |R_{2,k}| & = N(N+1)-h(R_{2,k})-5= 6k^2+24k+20.
    \end{split}
\end{equation}

The first two of this sequence are degenerate with respect to the spiral effect and are depicted below.
\vspace{.5cm}
\begin{figure}[H]
    \centering
    \begin{subfigure}[t]{.4\textwidth}
        \centering
        \begin{tikzpicture}
		\spiralborder{8}{9}
		\wtopcheck{8}{9}

		\fill[black!80] (3*\ra,4*\ra) --++(0:\ra) --++(90:\ra) --++(180:\ra) -- cycle;
		\fill[black!80] (5*\ra,4*\ra) --++(0:\ra) --++(90:\ra) --++(180:\ra) -- cycle;
		\fill[black!80] (5*\ra,2*\ra) --++(0:\ra) --++(90:\ra) --++(180:\ra) -- cycle;

		\spiralborderLines{8}{9}
		\wtopcheckLines{8}{9}
        \end{tikzpicture}
    \caption{$R_{2,1}, \ N=8, \ g(17)=50$}
    \end{subfigure}
        \begin{subfigure}[t]{.4\textwidth}
        \centering
        \begin{tikzpicture}
		\twospirals{11}
        \end{tikzpicture}
    \caption{$R_{2,2}, \ N=11, \ g(35)=92$}
    \end{subfigure}
\caption{The first two elements of $R_{2,k}$.}
\end{figure}
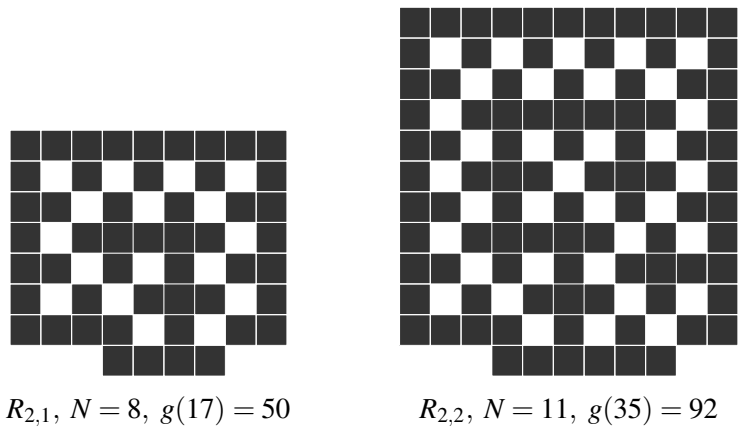

\section{Dismantling and the Proof of Theorem \ref{ghsolved}}\label{S:destroy}
Given a crstallized polyomino $A$ with $h_{\alpha}$ holes, we develop algorithms which remove at each step two tiles and one hole. By Lemma \ref{L:propertiesm}, this will produce constructions achieving $g(h)$ tiles and $h$ holes for all $h$ for which $\alpha$ is the minimum square or pronic number such that $h\le h_{\alpha}$. For $\alpha=N^2$ or $N(N+1)$, if we can remove at least $\lceil N/3\rceil$ holes from $A$ in this manner, then the total area of the remaining polyomino is at most $N(N-1)$ or $N^2$, respectively, and we say we have \emph{dismantled} $A$.  

Consider first the case that a polyomino has a plus tree rooted next to an indented corner, and the tree grows straight along the length of the side boundary. This subarrangement is especially prominent in $S_{1,k}$ in Figure \ref{fig:evens1}, and occurs in all constructions in Section \ref{constructions} except for the double spiral in \ref{pr2}. This can be dismantled by the following process depicted in Figure \ref{fig:dismantle}: (a) fill the lowest hole along the side boundary by pushing in the adjacent boundary tile; (b) remove the two boundary tiles adjacent to the indented corner; and (c) remove the indented corner and the boundary tile which is corner adjacent to it. Steps (a) and (b) combine to remove two tiles and one hole, and step (c) also removes two tiles and one hole, and additionally sets up a new indented corner so that the process can be repeated as long as the tree grows along that boundary. Both removals preserve the properties that the polyomino is acyclic and each hole has an area of one.

\vspace{.5cm}
\def\ra{.85}


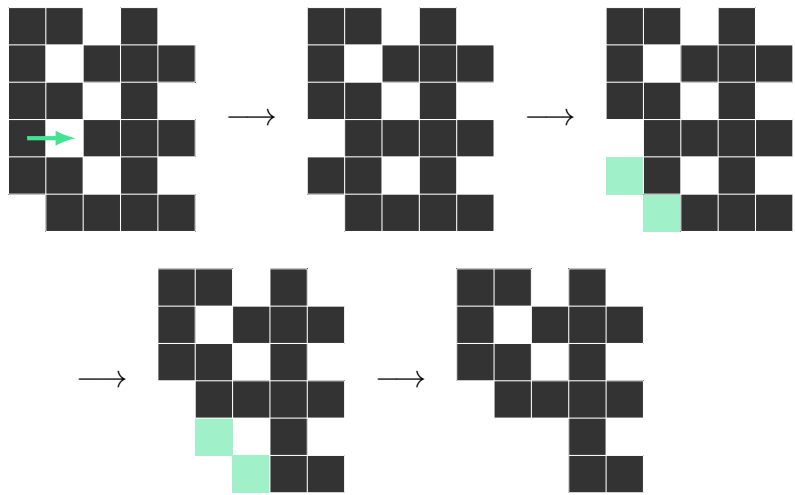
\begin{figure}[H]
\centering
\begin{subfigure}[t]{\textwidth}
\centering
\begin{tikzpicture}
	\foreach \i in {1,...,4} {
		\fill[black!80] (\i*\ra,0) --++(0:\ra) --++(90:\ra) --++(180:\ra) -- cycle;
	}
	
	\foreach \i in {1,...,5} {
		\fill[black!80] (0,\i*\ra) --++(0:\ra) --++(90:\ra) --++(180:\ra) -- cycle;
	}
	
	
	\foreach \i in {1,2} {
		\fill[black!80] (3*\ra,\i*2*\ra) --++(0:\ra) --++(90:\ra) --++(180:\ra) -- cycle;
		\foreach \j in {1,...,3} {
		\fill[black!80] (2*\i*\ra-\ra,\j*2*\ra-\ra) --++(0:\ra) --++(90:\ra) --++(180:\ra) -- cycle;
		}
		
		\foreach \j in {1,2} {
			\fill[black!80] (2*\i*\ra,2*\j*\ra) --++(0:\ra) --++(90:\ra) --++(180:\ra) -- cycle;
		}
	}


\draw[white] (0,\ra) -- (0,6*\ra);
\foreach \i in {1,...,4}{
	\draw[white] (\i*\ra,0) -- (\i*\ra,6*\ra);
}
\foreach \i in {1,...,5}{
	\draw[white] (0,\i*\ra) -- (5*\ra,\i*\ra);
}
\foreach \i in {1,...,3}{
	\draw[white] (5*\ra,2*\i*\ra-2*\ra) -- (5*\ra,2*\i*\ra-\ra);
}

\draw[white] (0,6*\ra) -- (2*\ra,6*\ra);
\draw[white] (3*\ra,6*\ra) -- (4*\ra,6*\ra);
\draw[white] (\ra,0) -- (5*\ra,0);
\draw[ultra thick, newcol, ->,>=latex] (.5*\ra,2.5*\ra) -- (1.8*\ra,2.5*\ra);
	
	\fill[black!80] (8*\ra+\ra,2*\ra) --++(0:\ra) --++(90:\ra) --++(180:\ra) -- cycle;

	\foreach \i in {1,...,4} {
		\fill[black!80] (8*\ra+\i*\ra,0) --++(0:\ra) --++(90:\ra) --++(180:\ra) -- cycle;
	}
	
	\foreach \i in {1,3,4,5} {
		\fill[black!80] (8*\ra,\i*\ra) --++(0:\ra) --++(90:\ra) --++(180:\ra) -- cycle;
	}
	
	
	\foreach \i in {1,2} {
		\fill[black!80] (8*\ra+3*\ra,\i*2*\ra) --++(0:\ra) --++(90:\ra) --++(180:\ra) -- cycle;
		\foreach \j in {1,...,3} {
		\fill[black!80] (8*\ra+2*\i*\ra-\ra,\j*2*\ra-\ra) --++(0:\ra) --++(90:\ra) --++(180:\ra) -- cycle;
		}
		
		\foreach \j in {1,2} {
			\fill[black!80] (8*\ra+2*\i*\ra,2*\j*\ra) --++(0:\ra) --++(90:\ra) --++(180:\ra) -- cycle;
		}
	}
	

\draw[white] (8*\ra,\ra) -- (8*\ra,2*\ra);
\draw[white] (8*\ra,3*\ra) -- (8*\ra,6*\ra);
\foreach \i in {1,...,4}{
	\draw[white] (8*\ra+\i*\ra,0) -- (8*\ra+\i*\ra,6*\ra);
}
\foreach \i in {1,...,5}{
	\draw[white] (8*\ra,\i*\ra) -- (8*\ra+5*\ra,\i*\ra);
}
\foreach \i in {1,...,3}{
	\draw[white] (8*\ra+5*\ra,2*\i*\ra-2*\ra) -- (8*\ra+5*\ra,2*\i*\ra-\ra);
}

\draw[white] (8*\ra,6*\ra) -- (8*\ra+2*\ra,6*\ra);
\draw[white] (8*\ra+3*\ra,6*\ra) -- (8*\ra+4*\ra,6*\ra);
\draw[white] (8*\ra+\ra,0) -- (8*\ra+5*\ra,0);
	

	\fill[black!80] (16*\ra+\ra,2*\ra) --++(0:\ra) --++(90:\ra) --++(180:\ra) -- cycle;
	
	\fill[newcol!50] (16*\ra+\ra,0) --++(0:\ra) --++(90:\ra) --++(180:\ra) -- cycle;
	\fill[newcol!50] (16*\ra,\ra) --++(0:\ra) --++(90:\ra) --++(180:\ra) -- cycle;

	\foreach \i in {2,3,4} {
		\fill[black!80] (16*\ra+\i*\ra,0) --++(0:\ra) --++(90:\ra) --++(180:\ra) -- cycle;
	}
	
	\foreach \i in {3,4,5} {
		\fill[black!80] (16*\ra,\i*\ra) --++(0:\ra) --++(90:\ra) --++(180:\ra) -- cycle;
	}
	
	
	\foreach \i in {1,2} {
		\fill[black!80] (16*\ra+3*\ra,\i*2*\ra) --++(0:\ra) --++(90:\ra) --++(180:\ra) -- cycle;
		\foreach \j in {1,...,3} {
		\fill[black!80] (16*\ra+2*\i*\ra-\ra,\j*2*\ra-\ra) --++(0:\ra) --++(90:\ra) --++(180:\ra) -- cycle;
		}
		
		\foreach \j in {1,2} {
			\fill[black!80] (16*\ra+2*\i*\ra,2*\j*\ra) --++(0:\ra) --++(90:\ra) --++(180:\ra) -- cycle;
		}
	}
	

\draw[white] (16*\ra,3*\ra) -- (16*\ra,6*\ra);
\draw[white] (16*\ra+\ra,\ra) -- (16*\ra+\ra,6*\ra);
\foreach \i in {2,3,4}{
	\draw[white] (16*\ra+\i*\ra,0) -- (16*\ra+\i*\ra,6*\ra);
}
\foreach \i in {1,2}{
	\draw[white] (16*\ra+\ra,\i*\ra) -- (16*\ra+5*\ra,\i*\ra);
}
\foreach \i in {3,4,5}{
	\draw[white] (16*\ra,\i*\ra) -- (16*\ra+5*\ra,\i*\ra);
}
\foreach \i in {1,...,3}{
	\draw[white] (16*\ra+5*\ra,2*\i*\ra-2*\ra) -- (16*\ra+5*\ra,2*\i*\ra-\ra);
}

\draw[white] (16*\ra,6*\ra) -- (16*\ra+2*\ra,6*\ra);
\draw[white] (16*\ra+3*\ra,6*\ra) -- (16*\ra+4*\ra,6*\ra);
\draw[white] (16*\ra+2*\ra,0) -- (16*\ra+5*\ra,0);
	

	\fill[black!80] (4*\ra+\ra,2*\ra-7*\ra) --++(0:\ra) --++(90:\ra) --++(180:\ra) -- cycle;
	
	\fill[newcol!50] (4*\ra+2*\ra,-7*\ra) --++(0:\ra) --++(90:\ra) --++(180:\ra) -- cycle;
	\fill[newcol!50] (4*\ra+\ra,\ra-7*\ra) --++(0:\ra) --++(90:\ra) --++(180:\ra) -- cycle;

	\foreach \i in {3,4} {
		\fill[black!80] (4*\ra+\i*\ra,-7*\ra) --++(0:\ra) --++(90:\ra) --++(180:\ra) -- cycle;
	}
	
	\foreach \i in {3,4,5} {
		\fill[black!80] (4*\ra,\i*\ra-7*\ra) --++(0:\ra) --++(90:\ra) --++(180:\ra) -- cycle;
	}
	
	\fill[black!80] (4*\ra+3*\ra,\ra-7*\ra) --++(0:\ra) --++(90:\ra) --++(180:\ra) -- cycle;
	\foreach \i in {1,2} {
		\fill[black!80] (4*\ra+3*\ra,\i*2*\ra-7*\ra) --++(0:\ra) --++(90:\ra) --++(180:\ra) -- cycle;
		\foreach \j in {2,3} {
		\fill[black!80] (4*\ra+2*\i*\ra-\ra,\j*2*\ra-\ra-7*\ra) --++(0:\ra) --++(90:\ra) --++(180:\ra) -- cycle;
		}
		
		\foreach \j in {1,2} {
			\fill[black!80] (4*\ra+2*\i*\ra,2*\j*\ra-7*\ra) --++(0:\ra) --++(90:\ra) --++(180:\ra) -- cycle;
		}
	}
	

\draw[white] (4*\ra,3*\ra-7*\ra) -- (4*\ra,6*\ra-7*\ra);
\draw[white] (4*\ra+\ra,2*\ra-7*\ra) -- (4*\ra+5*\ra,2*\ra-7*\ra);
\foreach \i in {1,2}{
	\draw[white] (4*\ra+\i*\ra,2*\ra-7*\ra) -- (4*\ra+\i*\ra,6*\ra-7*\ra);
}

\foreach \i in {3,4}{
	\draw[white] (4*\ra+\i*\ra,-7*\ra) -- (4*\ra+\i*\ra,6*\ra-7*\ra);
}
\foreach \i in {1,2}{
	\draw[white] (4*\ra+3*\ra,\i*\ra-\ra-7*\ra) -- (4*\ra+5*\ra,\i*\ra-\ra-7*\ra);
}

\foreach \i in {3,4,5}{
	\draw[white] (4*\ra,\i*\ra-7*\ra) -- (4*\ra+5*\ra,\i*\ra-7*\ra);
}
\foreach \i in {1,...,3}{
	\draw[white] (4*\ra+5*\ra,2*\i*\ra-2*\ra-7*\ra) -- (4*\ra+5*\ra,2*\i*\ra-\ra-7*\ra);
}

\draw[white] (4*\ra,6*\ra-7*\ra) -- (4*\ra+2*\ra,6*\ra-7*\ra);
\draw[white] (4*\ra+3*\ra,6*\ra-7*\ra) -- (4*\ra+4*\ra,6*\ra-7*\ra);


	\fill[black!80] (12*\ra+\ra,2*\ra-7*\ra) --++(0:\ra) --++(90:\ra) --++(180:\ra) -- cycle;

	\foreach \i in {3,4} {
		\fill[black!80] (12*\ra+\i*\ra,-7*\ra) --++(0:\ra) --++(90:\ra) --++(180:\ra) -- cycle;
	}
	
	\foreach \i in {3,4,5} {
		\fill[black!80] (12*\ra,\i*\ra-7*\ra) --++(0:\ra) --++(90:\ra) --++(180:\ra) -- cycle;
	}
	
	\fill[black!80] (12*\ra+3*\ra,\ra-7*\ra) --++(0:\ra) --++(90:\ra) --++(180:\ra) -- cycle;
	\foreach \i in {1,2} {
		\fill[black!80] (12*\ra+3*\ra,\i*2*\ra-7*\ra) --++(0:\ra) --++(90:\ra) --++(180:\ra) -- cycle;
		\foreach \j in {2,3} {
		\fill[black!80] (12*\ra+2*\i*\ra-\ra,\j*2*\ra-\ra-7*\ra) --++(0:\ra) --++(90:\ra) --++(180:\ra) -- cycle;
		}
		
		\foreach \j in {1,2} {
			\fill[black!80] (12*\ra+2*\i*\ra,2*\j*\ra-7*\ra) --++(0:\ra) --++(90:\ra) --++(180:\ra) -- cycle;
		}
	}
	

\draw[white] (12*\ra,3*\ra-7*\ra) -- (12*\ra,6*\ra-7*\ra);
\draw[white] (12*\ra+\ra,2*\ra-7*\ra) -- (12*\ra+5*\ra,2*\ra-7*\ra);
\foreach \i in {1,2}{
	\draw[white] (12*\ra+\i*\ra,2*\ra-7*\ra) -- (12*\ra+\i*\ra,6*\ra-7*\ra);
}

\foreach \i in {3,4}{
	\draw[white] (12*\ra+\i*\ra,-7*\ra) -- (12*\ra+\i*\ra,6*\ra-7*\ra);
}
\foreach \i in {1,2}{
	\draw[white] (12*\ra+3*\ra,\i*\ra-\ra-7*\ra) -- (12*\ra+5*\ra,\i*\ra-\ra-7*\ra);
}

\foreach \i in {3,4,5}{
	\draw[white] (12*\ra,\i*\ra-7*\ra) -- (12*\ra+5*\ra,\i*\ra-7*\ra);
}
\foreach \i in {1,...,3}{
	\draw[white] (12*\ra+5*\ra,2*\i*\ra-2*\ra-7*\ra) -- (12*\ra+5*\ra,2*\i*\ra-\ra-7*\ra);
}

\draw[white] (12*\ra,6*\ra-7*\ra) -- (12*\ra+2*\ra,6*\ra-7*\ra);
\draw[white] (12*\ra+3*\ra,6*\ra-7*\ra) -- (12*\ra+4*\ra,6*\ra-7*\ra);

\foreach \i in {0,1}{
	\node at (6.5*\ra+8*\i*\ra,3*\ra) {$\longrightarrow$};
}
\foreach \i in {0,1}{
	\node at (2.5*\ra+8*\i*\ra,3*\ra-7*\ra) {$\longrightarrow$};
}
\end{tikzpicture}
\end{subfigure}
\caption{Iterative process for removing two tiles and one hole.}
\label{fig:dismantle}
\end{figure}

In the pronic mod 2 case there is not always a long rooted tree, but the tree along the top can be rooted so that this process can be implemented. To do so, as in Figure \ref{fig:rootit}: (a) move the top corner tile to the adjacent interior corner; (b) remove the two tiles that were adjacent to the corner tile; (c) push the boundary tile on the bottom row up to root the tree; (d) remove the two tiles which had been adjacent to the tile that was pushed up.


\def\ra{.85}

\begin{figure}[H]
\centering
    \begin{subfigure}[t]{.3\textwidth}
        \centering
        \resizebox{!}{\textwidth}{
            \begin{tikzpicture}
                \twospirals{11}
                \draw[ultra thick, newcol, ->, >=latex] (0.5*\ra,11.5*\ra) -- (1.6*\ra,10.5*\ra);
                \node at (12.5*\ra,6*\ra) {$\longrightarrow$};
            \end{tikzpicture}
        }
    \end{subfigure}
\hspace{.3cm}
    \begin{subfigure}[t]{.3\textwidth}
        \centering
        \resizebox{!}{\textwidth}{
            \begin{tikzpicture}
                \fill[black!80] (\ra,10*\ra) --++(0:\ra) --++(90:\ra) --++(180:\ra) -- cycle;

                \foreach \i in {1,...,9}{
                    \fill[black!80] (0,\i*\ra) --++(0:\ra) --++(90:\ra) --++(180:\ra) -- cycle;
                }
                \foreach \i in {2,...,10}{
                    \fill[black!80] (\i*\ra,11*\ra) --++(0:\ra) --++(90:\ra) --++(180:\ra) -- cycle;
                }
                \foreach \i in {1,...,10}{
                    \fill[black!80] (10*\ra,\i*\ra) --++(0:\ra) --++(90:\ra) --++(180:\ra) -- cycle;
                }
                \foreach \i in {3,...,8}{
                    \fill[black!80] (\i*\ra,0) --++(0:\ra) --++(90:\ra) --++(180:\ra) -- cycle;
                }
                \fill[black!80] (2*\ra,\ra) --++(0:\ra) --++(90:\ra) --++(180:\ra) -- cycle;

                \twospiralsNoB{11}
                
                \foreach \i in {1,...,9}{
                \draw[white] (0,\i*\ra) --++(0:\ra) --++(90:\ra) --++(180:\ra) --++(270:\ra);
                }
                \foreach \i in {2,...,10}{
                    \draw[white] (\i*\ra,11*\ra) --++(0:\ra) --++(90:\ra) --++(180:\ra) --++(270:\ra);
                }
                \foreach \i in {1,...,10}{
                    \draw[white] (10*\ra,\i*\ra) --++(0:\ra) --++(90:\ra) --++(180:\ra) --++(270:\ra);
                }
                \foreach \i in {3,...,8}{
                    \draw[white] (\i*\ra,0) --++(0:\ra) --++(90:\ra) --++(180:\ra) --++(270:\ra);
                }
                \draw[white] (2*\ra,\ra) --++(0:\ra) --++(90:\ra) --++(180:\ra) --++(270:\ra);

                \fill[newcol!40] (0,10*\ra) --++(0:\ra) --++(90:\ra) --++(180:\ra) -- cycle;
                \fill[newcol!40] (\ra,11*\ra) --++(0:\ra) --++(90:\ra) --++(180:\ra) -- cycle;

                \draw[ultra thick, newcol, ->, >=latex] (0.5*\ra,8.5*\ra) -- (1.6*\ra,8.5*\ra);

                \node at (12.5*\ra,6*\ra) {$\longrightarrow$};
            \end{tikzpicture}
        }
    \end{subfigure}
\hspace{.3cm}
    \begin{subfigure}[t]{.3\textwidth}
        \centering
        \resizebox{!}{\textwidth}{
            \begin{tikzpicture}
                \fill[black!80] (\ra,10*\ra) --++(0:\ra) --++(90:\ra) --++(180:\ra) -- cycle;

                \foreach \i in {1,...,6}{
                    \fill[black!80] (0,\i*\ra) --++(0:\ra) --++(90:\ra) --++(180:\ra) -- cycle;
                }
                \foreach \i in {2,...,10}{
                    \fill[black!80] (\i*\ra,11*\ra) --++(0:\ra) --++(90:\ra) --++(180:\ra) -- cycle;
                }
                \foreach \i in {1,...,10}{
                    \fill[black!80] (10*\ra,\i*\ra) --++(0:\ra) --++(90:\ra) --++(180:\ra) -- cycle;
                }
                \foreach \i in {3,...,8}{
                    \fill[black!80] (\i*\ra,0) --++(0:\ra) --++(90:\ra) --++(180:\ra) -- cycle;
                }
                \fill[black!80] (2*\ra,\ra) --++(0:\ra) --++(90:\ra) --++(180:\ra) -- cycle;

                \fill[black!80] (\ra,8*\ra) --++(0:\ra) --++(90:\ra) --++(180:\ra) -- cycle;

                \twospiralsNoB{11}

                \foreach \i in {1,...,9}{
                \draw[white] (0,\i*\ra) --++(0:\ra) --++(90:\ra) --++(180:\ra) --++(270:\ra);
                }
                \foreach \i in {2,...,10}{
                    \draw[white] (\i*\ra,11*\ra) --++(0:\ra) --++(90:\ra) --++(180:\ra) --++(270:\ra);
                }
                \foreach \i in {1,...,10}{
                    \draw[white] (10*\ra,\i*\ra) --++(0:\ra) --++(90:\ra) --++(180:\ra) --++(270:\ra);
                }
                \foreach \i in {3,...,8}{
                    \draw[white] (\i*\ra,0) --++(0:\ra) --++(90:\ra) --++(180:\ra) --++(270:\ra);
                }
                \draw[white] (2*\ra,\ra) --++(0:\ra) --++(90:\ra) --++(180:\ra) --++(270:\ra);

                \foreach \i in {7,9}{
                	\fill[newcol!40] (0,\i*\ra) --++(0:\ra) --++(90:\ra) --++(180:\ra) -- cycle;
                }

            \end{tikzpicture}
        }
    \end{subfigure}
    \caption{Process for rooting a tree in the double spiral construction.}
    \label{fig:rootit}
\end{figure}
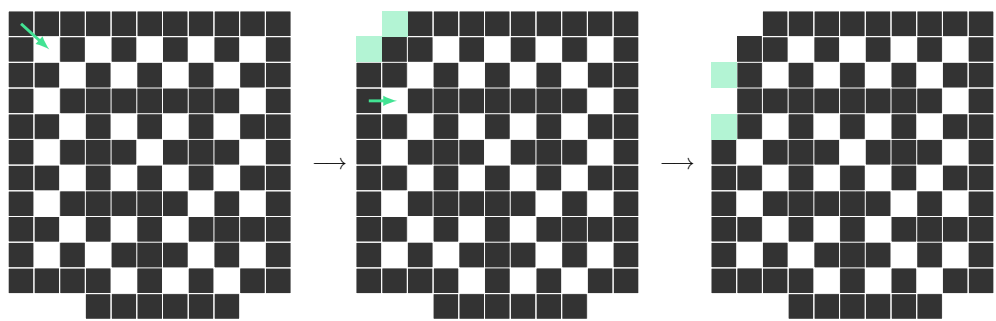

Dismantling odd squares becomes increasingly difficult upon subsequent expansions. There will always be holes in all four corners by parity, and these holes are always easily removed via methods used in the previous two examples. If the outer corner space is filled, we can apply the first removal from Figure \ref{fig:rootit}, and if it is empty we can apply the  second removal from Figure \ref{fig:dismantle}.

Removing the corners subtracts 12 total spaces, which suffices to dismantle up to an $11\times 11$ square. For larger squares, after the corners are dealt with we remove the next hole via a complete rearrangement which is a modified version of the double spiral construction, with unique central configurations depending on $N \text{ mod } 3$, as depicted in Figure \ref{fig:rearrange}.


\def\ra{.45}

\begin{figure}[H]
\centering
    \begin{subfigure}[t]{.27\textwidth}
        \centering
            \begin{tikzpicture}
                \rearrangeOne{13}
            \end{tikzpicture}
    \end{subfigure}
\hspace{.3cm}
    \begin{subfigure}[t]{.3\textwidth}
        \centering
            \begin{tikzpicture}
                \rearrangeZero{15}
            \end{tikzpicture}
    \end{subfigure}
\hspace{.3cm}
    \begin{subfigure}[t]{.32\textwidth}
        \centering
            \begin{tikzpicture}
                \rearrangeTwo{17}
            \end{tikzpicture}
    \end{subfigure}
    \caption{Rearrangements for $N=$ 13, 15, and 17.}
    \label{fig:rearrange}
\end{figure}
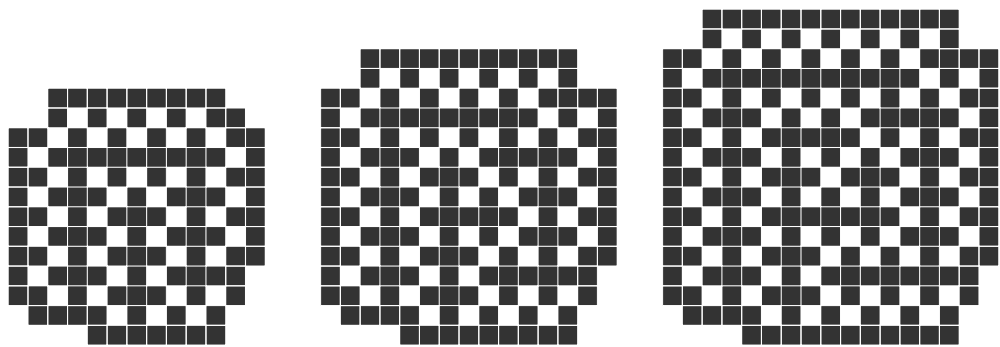

These arrangements generalize straightforwardly for larger $N$ by expanding the boundary and extending each spiral in two directions. For example, the 1 mod 3 construction on the left in Figure \ref{fig:rearrange} extends to $N=19$ by taking the spiral that stops at the bottom right and extending it left and then up, and the spiral that is pointing up extends right and then down. This adds six to both the width and height, and thus preserves $N \text{ mod } 3$.

The $1 \text{ mod 3}$ construction has 18 empty spaces around the boundary, and the other two have 19 empty spaces. Recall that for $N\equiv 1 \text{ mod }3$, the construction for $h_{N^2}$ holes has three empty corners, and the constructions for $N\not\equiv 1 \text{ mod 3}$ both have four empty corners. Then another 12 spaces are removed by dismantling the four corner holes, and this rearrangement removes one additional hole, giving the above numbers. 

The only exception is the sequence $S_l$, which starts with a single empty corner and requires one extra removal. In removing the four corners from $S_l$, three become indented as in Figure \ref{fig:rootit}. Then we get the additional removal by moving one these three indented corners to the center and removing the other two. Observe that these rearrangements preserve the acyclic structure of the polyomino and the property that all holes have area one.

\def\ra{.85}
\begin{figure}[H]
\centering
    \begin{subfigure}[t]{.27\textwidth}
        \centering
        \resizebox{!}{\textwidth}{
            \begin{tikzpicture}
                \fullborder{9}{9}
                \wtopcheck{9}{9}
                \foreach \i in {0,1}{
                    \fill[black!80] (4*\ra,4*\i*\ra+2*\ra) --++(0:\ra) --++(90:\ra) --++(180:\ra) -- cycle;
                    \fill[black!80] (4*\i*\ra+2*\ra,4*\ra) --++(0:\ra) --++(90:\ra) --++(180:\ra) -- cycle;
                }
                \fullborderLines{9}{9}
                \wtopcheckLines{9}{9}
                \draw[ultra thick, newcol, ->, >=latex] (0.5*\ra,8.5*\ra) -- (1.6*\ra,7.4*\ra);
                \draw[ultra thick, newcol, ->, >=latex] (0.5*\ra,0.5*\ra) -- (1.6*\ra,1.6*\ra);
                \draw[ultra thick, newcol, ->, >=latex] (8.5*\ra,8.5*\ra) -- (7.4*\ra,7.4*\ra);
                
                \node at (10.3*\ra,4.5*\ra) {$\longrightarrow$};
            \end{tikzpicture}
        }
    \end{subfigure}
\hspace{.8cm}
    \begin{subfigure}[t]{.27\textwidth}
        \centering
        \resizebox{!}{\textwidth}{
           \begin{tikzpicture}
                \foreach \i in {2,...,6}{
                    \fill[black!80] (0,\i*\ra) --++(0:\ra) --++(90:\ra) --++(180:\ra) -- cycle;
                    \fill[black!80] (\i*\ra,0) --++(0:\ra) --++(90:\ra) --++(180:\ra) -- cycle;
                    \fill[black!80] (8*\ra,\i*\ra) --++(0:\ra) --++(90:\ra) --++(180:\ra) -- cycle;
                    \fill[black!80] (\i*\ra,8*\ra) --++(0:\ra) --++(90:\ra) --++(180:\ra) -- cycle;
                }
                \wtopcheck{9}{9}
                \foreach \i in {0,1}{
                    \fill[black!80] (4*\ra,4*\i*\ra+2*\ra) --++(0:\ra) --++(90:\ra) --++(180:\ra) -- cycle;
                    \fill[black!80] (4*\i*\ra+2*\ra,4*\ra) --++(0:\ra) --++(90:\ra) --++(180:\ra) -- cycle;
                }
                \fill[black!80] (\ra,\ra) --++(0:\ra) --++(90:\ra) --++(180:\ra) -- cycle;
                \fill[black!80] (\ra,7*\ra) --++(0:\ra) --++(90:\ra) --++(180:\ra) -- cycle;
                \fill[black!80] (7*\ra,7*\ra) --++(0:\ra) --++(90:\ra) --++(180:\ra) -- cycle;

                \foreach \i in {0,1}{
                    \draw[white] (4*\ra,4*\i*\ra+2*\ra) --++(0:\ra) --++(90:\ra) --++(180:\ra) --++(270:\ra);
                    \draw[white] (4*\i*\ra+2*\ra,4*\ra) --++(0:\ra) --++(90:\ra) --++(180:\ra) --++(270:\ra);
                }
                \foreach \i in {2,...,6}{
                    \draw[white] (0,\i*\ra) --++(0:\ra) --++(90:\ra) --++(180:\ra) --++(270:\ra);
                    \draw[white] (\i*\ra,0) --++(0:\ra) --++(90:\ra) --++(180:\ra) --++(270:\ra);
                    \draw[white] (8*\ra,\i*\ra) --++(0:\ra) --++(90:\ra) --++(180:\ra) --++(270:\ra);
                    \draw[white] (\i*\ra,8*\ra) --++(0:\ra) --++(90:\ra) --++(180:\ra) --++(270:\ra);
                }
                \draw[white] (\ra,\ra) --++(0:\ra) --++(90:\ra) --++(180:\ra) --++(270:\ra);
                \draw[white] (\ra,7*\ra) --++(0:\ra) --++(90:\ra) --++(180:\ra) --++(270:\ra);
                \draw[white] (7*\ra,7*\ra) --++(0:\ra) --++(90:\ra) --++(180:\ra) --++(270:\ra);

                \wtopcheckLines{9}{9}

                \foreach \i in {1,7}{
                    \fill[newcol!40] (\i*\ra,0) --++(0:\ra) --++(90:\ra) --++(180:\ra) -- cycle;
                    \fill[newcol!40] (0,\i*\ra) --++(0:\ra) --++(90:\ra) --++(180:\ra) -- cycle;
                    \fill[newcol!40] (\i*\ra,8*\ra) --++(0:\ra) --++(90:\ra) --++(180:\ra) -- cycle;
                    \fill[newcol!40] (8*\ra,\i*\ra) --++(0:\ra) --++(90:\ra) --++(180:\ra) -- cycle;
                }
            
                \draw[ultra thick, newcol, ->, >=latex] (1.5*\ra,1.5*\ra) -- (4.5*\ra,4.5*\ra);
                
                \node at (10.3*\ra,4.5*\ra) {$\longrightarrow$};

            \end{tikzpicture}
        }
    \end{subfigure}
\hspace{.8cm}
    \begin{subfigure}[t]{.27\textwidth}
        \centering
        \resizebox{!}{\textwidth}{
            \begin{tikzpicture}
                \foreach \i in {2,...,6}{
                    \fill[black!80] (0,\i*\ra) --++(0:\ra) --++(90:\ra) --++(180:\ra) -- cycle;
                    \fill[black!80] (\i*\ra,0) --++(0:\ra) --++(90:\ra) --++(180:\ra) -- cycle;
                    \fill[black!80] (8*\ra,\i*\ra) --++(0:\ra) --++(90:\ra) --++(180:\ra) -- cycle;
                    \fill[black!80] (\i*\ra,8*\ra) --++(0:\ra) --++(90:\ra) --++(180:\ra) -- cycle;
                }
                \wtopcheck{9}{9}
                \foreach \i in {0,1}{
                    \fill[black!80] (4*\ra,4*\i*\ra+2*\ra) --++(0:\ra) --++(90:\ra) --++(180:\ra) -- cycle;
                    \fill[black!80] (4*\i*\ra+2*\ra,4*\ra) --++(0:\ra) --++(90:\ra) --++(180:\ra) -- cycle;
                }

                \fill[black!80] (4*\ra,4*\ra) --++(0:\ra) --++(90:\ra) --++(180:\ra) -- cycle;

                \foreach \i in {0,1}{
                    \draw[white] (4*\ra,4*\i*\ra+2*\ra) --++(0:\ra) --++(90:\ra) --++(180:\ra) --++(270:\ra);
                    \draw[white] (4*\i*\ra+2*\ra,4*\ra) --++(0:\ra) --++(90:\ra) --++(180:\ra) --++(270:\ra);
                }
                \foreach \i in {2,...,6}{
                    \draw[white] (0,\i*\ra) --++(0:\ra) --++(90:\ra) --++(180:\ra) --++(270:\ra);
                    \draw[white] (\i*\ra,0) --++(0:\ra) --++(90:\ra) --++(180:\ra) --++(270:\ra);
                    \draw[white] (8*\ra,\i*\ra) --++(0:\ra) --++(90:\ra) --++(180:\ra) --++(270:\ra);
                    \draw[white] (\i*\ra,8*\ra) --++(0:\ra) --++(90:\ra) --++(180:\ra) --++(270:\ra);
                }
                \draw[white] (\ra,\ra) --++(0:\ra) --++(90:\ra) --++(180:\ra) --++(270:\ra);
                \draw[white] (\ra,7*\ra) --++(0:\ra) --++(90:\ra) --++(180:\ra) --++(270:\ra);
                \draw[white] (7*\ra,7*\ra) --++(0:\ra) --++(90:\ra) --++(180:\ra) --++(270:\ra);

                \wtopcheckLines{9}{9}

                \fill[newcol!40] (\ra,7*\ra) --++(0:\ra) --++(90:\ra) --++(180:\ra) -- cycle;
                \fill[newcol!40] (7*\ra,7*\ra) --++(0:\ra) --++(90:\ra) --++(180:\ra) -- cycle;

            \end{tikzpicture}
        }
    \end{subfigure}
        \caption{General method to remove five holes from $S_l$.}
        \label{fig:KRdismantle}
\end{figure}
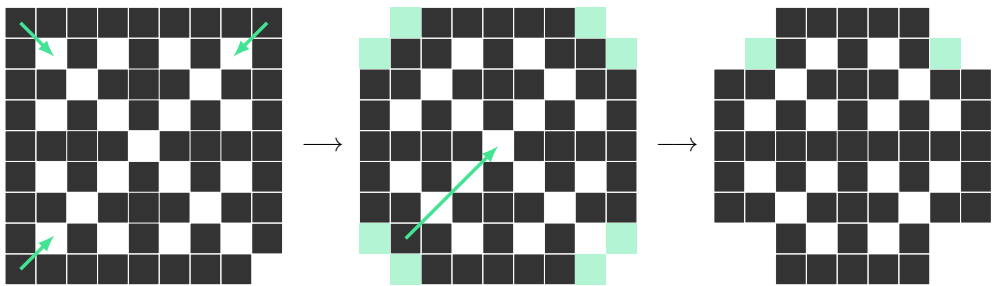

\begin{proof}[\textbf{Proof of Theorem \ref{ghsolved}}]
As noted above, all constructions in Section \ref{constructions} either have a rooted plus tree growing next to an indented corner along the full side of the polyomino, or have a plus tree which can be rooted as shown in Figure \ref{fig:rootit}, and any crystallized polyomino with $h_{\alpha}$ holes for $\alpha$ an odd square which is constructed via expansion can be rearranged after some initial removals to have a rooted plus tree growing along a full side of the polyomino. Then, implementing the removal process for holes along a rooted tree clearly suffices to dismantle all of these crystallized polyominoes, as they remove at least $N-4$ holes, far more than the roughly $N/3$ necessary.

This implies that $g(h)\le g(h_{\alpha})-2(h_{\alpha}-h)$ for any $h$ such that $\alpha$ is the minimum square or pronic number with $h \le h_{\alpha}$. Let $\alpha' < \alpha$ be the next largest square or pronic number. Then since $m(h_{\alpha})=g(h_{\alpha})$, Lemma \ref{L:propertiesm} ensures that $g(h_{\alpha})-2(h_{\alpha}-h)=m(h)$ for all $h$ such that $h+m(h)\ge\alpha'$. When $\alpha'=(2^l+1)^2$, we have that $$(h_{\alpha'}+1)+m(h_{\alpha'}+1) > \alpha,'$$ and otherwise $h_{\alpha'}+1=t_{\alpha'}$. Then $m(t_{\alpha'}+1)=m(t_{\alpha'})+3$ by Lemma \ref{L:propertiesm}, and by Lemmas \ref{nosqs}, \ref{nodiv2}, and \ref{L:oddintobst} we have that $$g(t_{\alpha'})=m(t_{\alpha'})+1=m(t_{\alpha'}+1)-2.$$
Therefore $g(h_{\alpha})-2(h_{\alpha}-h)=g(h)$ for all $h$ such that $\alpha$ is the minimum square or pronic number such that $h\le h_{\alpha}$.
\end{proof}

Observe that these removals will quickly disrupt the property of having minimal outer perimeter. In fact, none of the constructions in Figure \ref{fig:rearrange} are crystallized, since in all three cases $N^2-18$ is already less than $N(N-1)$. 

The jumps from $h_{\alpha}$ to $h_{\alpha}+1$ are the only jumps of three for $g(h)$ once $h > 6$. We also demonstrate a constructive procedure to show that three tiles is always sufficient to produce an additional hole when $h > 6$. Whenever there is a plus rooted to a border, the local transformation in Figure \ref{fig:plus3} maintains all routes of connectivity in the polyomino and does not create any cycles.
\def\ra{.85}


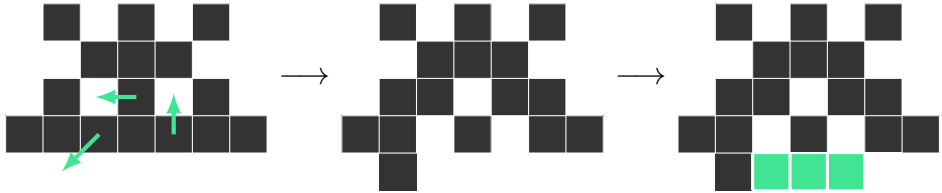
\begin{figure}[H]
\centering

\begin{tikzpicture}
	\foreach \i in {0,...,6} {
		\fill[black!80] (\i*\ra,0) --++(0:\ra) --++(90:\ra) --++(180:\ra) -- cycle;
	}
	
	\foreach \i in {1,...,3} {
		\foreach \j in {1,2}{
			\fill[black!80] (2*\i*\ra-\ra,2*\j*\ra-\ra) --++(0:\ra) --++(90:\ra) --++(180:\ra) -- cycle;
		}
	}
	
	\foreach \i in {1,2}{
			\fill[black!80] (2*\i*\ra,2*\ra) --++(0:\ra) --++(90:\ra) --++(180:\ra) -- cycle;
		}
	
	\fill[black!80] (3*\ra,2*\ra) --++(0:\ra) --++(90:\ra) --++(180:\ra) -- cycle;

\foreach \i in {0,1}{
	\draw[white] (0,\i*\ra) -- (7*\ra,\i*\ra);
}
\foreach \i in {1,...,6}{
	\draw[white] (\i*\ra,0) -- (\i*\ra,4*\ra);
}
\foreach \i in {2,3,4}{
	\draw[white] (\ra,\i*\ra) -- (6*\ra,\i*\ra);
}

\foreach \i in {0,7}{
	\draw[white] (\i*\ra,0) -- (\i*\ra,\ra);
}

\draw[ultra thick, newcol, ->, >=latex] (4.5*\ra,0.5*\ra) -- (4.5*\ra,1.6*\ra);
\draw[ultra thick, newcol, ->, >=latex] (3.5*\ra,1.5*\ra) -- (2.4*\ra,1.5*\ra);
\draw[ultra thick, newcol, ->, >=latex] (2.5*\ra,0.5*\ra) -- (1.5*\ra,-0.5*\ra);
	

	\foreach \i in {0,1,3,5,6} {
		\fill[black!80] (9*\ra+\i*\ra,0) --++(0:\ra) --++(90:\ra) --++(180:\ra) -- cycle;
	}
	
	\foreach \i in {1,2,4,5} {
		\fill[black!80] (9*\ra+\i*\ra,\ra) --++(0:\ra) --++(90:\ra) --++(180:\ra) -- cycle;
	}
	
	\foreach \i in {2,3,4} {
		\fill[black!80] (9*\ra+\i*\ra,2*\ra) --++(0:\ra) --++(90:\ra) --++(180:\ra) -- cycle;
	}
	
	\foreach \i in {1,3,5} {
		\fill[black!80] (9*\ra+\i*\ra,3*\ra) --++(0:\ra) --++(90:\ra) --++(180:\ra) -- cycle;
	}

	\fill[black!80] (9*\ra+\ra,-\ra) --++(0:\ra) --++(90:\ra) --++(180:\ra) -- cycle;

\foreach \i in {0,1}{
	\draw[white] (9*\ra,\i*\ra) -- (9*\ra+7*\ra,\i*\ra);
}
\foreach \i in {1,...,6}{
	\draw[white] (9*\ra+\i*\ra,0) -- (9*\ra+\i*\ra,4*\ra);
}
\foreach \i in {2,3,4}{
	\draw[white] (9*\ra+\ra,\i*\ra) -- (9*\ra+6*\ra,\i*\ra);
}

\foreach \i in {0,7}{
	\draw[white] (9*\ra+\i*\ra,0) -- (9*\ra+\i*\ra,\ra);
}


	\foreach \i in {0,1,3,5,6} {
		\fill[black!80] (18*\ra+\i*\ra,0) --++(0:\ra) --++(90:\ra) --++(180:\ra) -- cycle;
	}
	
	\foreach \i in {1,2,4,5} {
		\fill[black!80] (18*\ra+\i*\ra,\ra) --++(0:\ra) --++(90:\ra) --++(180:\ra) -- cycle;
	}
	
	\foreach \i in {2,3,4} {
		\fill[black!80] (18*\ra+\i*\ra,2*\ra) --++(0:\ra) --++(90:\ra) --++(180:\ra) -- cycle;
	}
	
	\foreach \i in {1,3,5} {
		\fill[black!80] (18*\ra+\i*\ra,3*\ra) --++(0:\ra) --++(90:\ra) --++(180:\ra) -- cycle;
	}

	\fill[black!80] (18*\ra+\ra,-\ra) --++(0:\ra) --++(90:\ra) --++(180:\ra) -- cycle;
	
	\foreach \i in {2,3,4} {
		\fill[newcol] (18*\ra+\i*\ra,-\ra) --++(0:\ra) --++(90:\ra) --++(180:\ra) -- cycle;
	}

\foreach \i in {0,1}{
	\draw[white] (18*\ra,\i*\ra) -- (18*\ra+7*\ra,\i*\ra);
}
\foreach \i in {1,...,6}{
	\draw[white] (18*\ra+\i*\ra,0) -- (18*\ra+\i*\ra,4*\ra);
}
\foreach \i in {2,3,4}{
	\draw[white] (18*\ra+\ra,\i*\ra) -- (18*\ra+6*\ra,\i*\ra);
}

\foreach \i in {0,7}{
	\draw[white] (18*\ra+\i*\ra,0) -- (18*\ra+\i*\ra,\ra);
}
\foreach \i in {2,3,4,5}{
	\draw[very thick, white] (18*\ra+\i*\ra,-\ra) -- (18*\ra+\i*\ra,0);
}
\foreach \i in {1,2}{
	\draw[thick, white] (18*\ra+2*\ra,\ra-\i*\ra) -- (18*\ra+6*\ra,\ra-\i*\ra);
}

\node at (8*\ra,2*\ra) {$\longrightarrow$};
\node at (17*\ra,2*\ra) {$\longrightarrow$};
\end{tikzpicture}
\caption{Adding one hole and three tiles to a rooted plus on the boundary.}
\label{fig:plus3}
\end{figure}

The $h=5$ case in a $5\times 5$ square is the largest crystallized polyomino without a rooted plus, and in particular the jump from $g(5)=19$ to $g(6)=23$ is the last time there is a jump of more than 3 in $g(h)$.

\section{Concluding Remarks and Open Problems}


We have completely determined the sequence $g(h)$ which gives the minimum number of tiles needed to construct a polyomino with $h$ holes, and determined precisely when these crystallized polyominoes attain efficient structural conditions. Using Theorems \ref{T:holesminper}, \ref{T:gpronicperfectsquare}, and \ref{ghsolved}, we have computed in Table 2 the values of $g(h)$ for $9\leq h\leq 113$.

We have also continued the enumeration of polyominoes with maximally many holes started in \cite{bworld} by proving that crystallized polyominoes in the K--R sequence constructed in \cite{kahle2018polyominoes} are unique. It remains an open problem to enumerate the rest of the crystallized polyominoes, that is, to find the number of free crystallized polyominoes with $h$ holes for $h\neq (2^{2l}-1)/3$ and $h>8$. Expansion can perhaps be leveraged in this pursuit, but in its current formulation it does not capture small changes to the boundary. For instance the the tile on the bottom of the left boundary section in $S_{1,1}$ in Figure \ref{fig:evens1} can be moved to the bottom right corner. But for any $l$, $E^l(S_{1,1})$ will necessarily have holes in all four interior corners, so we cannot simply copy the same boundary structure when a corner and one of its adjacent tiles are both missing, and these crystallized arrangements do not expand up.

There is a subfamily of polyominoes with holes, called \textit{punctured polyominoes} that have been studied before but without asking the extremal problem of maximizing the number of holes within that subfamily. For definitions and results on punctured polyominoes see Chapter 2, 8, and 11 of \cite{guttmann2009polygons}. An important observation is that punctured polyominoes do not capture the topological structure that we are interested in studying in this paper.

Viewing polyominoes as embedded 2-dimensional cubical complexes and the number of holes as the rank of the first homology group with coefficients in $\mathbb{Z}_2$, this problem also generalizes immediately to the question of maximizing the rank of homology in higher dimensional cubical complexes. For example, viewing polycubes as 3-dimensional cubical complexes, the problem corresponds to maximizing the ranks of the first and the second homology groups with coefficients in $\mathbb{Z}_2$. An analogue of the \hyperref[compressionlem]{Compression Lemma} can be formulated in this context, but the extent to which this will preserve efficient conditions in that setting is not yet known.

\section*{Acknowledgments.}
This work was supported by HFSP RGP0051/2017, NSF DMS 17-13012, and NSF DMS-1352386.

\begin{table}[H]
\begin{center}
\begin{tabular}{c c|c c| c c | c c |c c }
    \hline
    h & g(h) & h & g(h) & h &g(h) & h & g(h) & h & g(h) \\
    \hline
    \textbf{9} & 30 & 30 & 81 &\textbf{51} &128 &72 & 175 & 93&221\\ 
    10 & 33 & 31 & 83& 52 &131 & 73 & 177 & 94&223\\
    \textbf{11} &  35 & \textbf{32} & 85& 53& 133 &\textbf{74} & 179 & \textbf{95}& 225\\
    12 & 38 & 33 & 88 & 54 &135  & 75& 182 & 96&228\\
    \textbf{13} & 40 & 34 & 90 & \textbf{55} &137 &76 & 184 &97 &230\\
    14 & 43 & \textbf{35} & 92 & 56 &140 &77 & 186 &98&232\\
    \textbf{15} & 45 & 36 & 95 & 57 &142  &78 & 188 &99&234\\
    16 & 48 & 37 & 97 &58& 144 & \textbf{79} & 190 & 100 &236\\
    \textbf{17} &  50 & 38 & 99 &\textbf{59}& 146  & 80& 193 & \textbf{101}&238\\
    18 & 53 & \textbf{39}& 101& 60 & 149 & 81 & 195& 102&241\\
    19 & 55 & 40 &104 &61 & 151  & 82 & 197 & 103&243\\
    20 & 57 & 41 &106 &62 &  153 & 83 & 199 & 104&245\\
    \textbf{21} & 59 &  42& 108 &63 &155 & \textbf{84} & 201  & 105&247\\
    22 &  62 & \textbf{43} & 110 & \textbf{64} &157 & 85 & 204& 106&249\\
    \textbf{23} &  64 & 44 &113&65 & 160 & 86 & 206 & \textbf{107}&251\\
    24 &  67 & 45 &115 &66 & 162 & 87 & 208 & 108&254\\
    25 &  69 & 46 & 117&67 & 164 & 88 & 210 &109&256\\
    \textbf{26} &  71 & \textbf{47} & 119&68 & 166 & \textbf{89} & 212 &110&258\\
    27 &  74 & 48 &122 &\textbf{69} & 168 & 90 &215 &111&260\\
    28 &  76 & 49 &124 &70 & 171 & 91 & 217& 112 & 262\\
    \textbf{29} & 78 & 50 &126 &71 & 173 & 92 & 219 &\textbf{113}&264\\

    \hline

\end{tabular}
\caption{Values of $g(h)$ for $9\leq h\leq 114$. We indicate in bold the values of $h$ if $h=h_\alpha$ for an $\alpha \in \{N^2, N(N+1)\}$. For $h<9$, see Table \ref{tab:my_label}.}
\label{Table:113}
\end{center}
\end{table}
\newpage
\bibliographystyle{plain}
\bibliography{bibliography.bib}

\end{document}